\documentclass[11pt]{amsart}
\usepackage{amsmath,amssymb,amsthm,mathdots}
\usepackage{hyperref}
\usepackage[a4paper,hmargin=2.5cm,vmargin=2.5cm]{geometry}
\usepackage[dvips,dvipdfm,pdftex]{graphicx}
\usepackage{tikz}
\usepackage{booktabs}
\usepackage{mathrsfs}
\usepackage{float}
\usepackage{todonotes}
\usepackage{subcaption}
\usepackage{multicol}
\usepackage{ytableau}
\usepackage{longtable}

\renewcommand{\arraystretch}{2}

\newcommand{\cross}{\includegraphics[width=0.43cm]{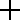}}
\newcommand{\bcross}{\includegraphics[width=0.43cm]{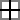}}
\newcommand{\elbows}{\includegraphics[width=0.43cm]{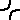}}
\newcommand{\belbows}{\includegraphics[width=0.43cm]{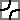}}
\newcommand{\elbow}{\includegraphics[width=0.43cm]{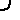}}
\newcommand{\elbowup}{\includegraphics[width=0.43cm]{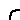}}
\newcommand{\belbowup}{\includegraphics[width=0.43cm]{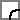}}
\newcommand{\elbowsign}{\includegraphics[width=0.43cm]{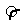}}
\newcommand{\belbowsign}{\includegraphics[width=0.43cm]{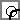}}
\newcommand{\elbowssign}{\includegraphics[width=0.43cm]{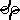}}
\newcommand{\belbowssign}{\includegraphics[width=0.43cm]{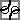}}
\newcommand{\crosssign}{\includegraphics[width=0.43cm]{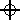}}
\newcommand{\bcrosssign}{\includegraphics[width=0.43cm]{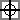}}

\newcommand{\boxx}{\includegraphics[scale=0.3]{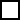}}
\newcommand{\crosscirc}{\includegraphics[width=0.43cm]{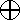}}

\newlength\cellsize 

\savebox2{%
	\begin{picture}(12,12)
	\put(0,0){\line(1,0){12}}
	\put(0,0){\line(0,1){12}}
	\put(12,0){\line(0,1){12}}
	\put(0,12){\line(1,0){12}}
	\end{picture}}

\newcommand\cellify[1]{\def\thearg{#1}\def\nothing{}%
	\ifx\thearg\nothing\vrule width0pt height\cellsize depth0pt%
	\else\hbox to 0pt{\usebox2\hss}\fi%
	\vbox to 12\unitlength{\vss\hbox to 12\unitlength{\hss$#1$\hss}\vss}}

\newcommand\tableau[1]{\vtop{\let\\=\cr
		\setlength\baselineskip{-12000pt}
		\setlength\lineskiplimit{12000pt}
		\setlength\lineskip{0pt}
		\halign{&\cellify{##}\cr#1\crcr}}}

\newcommand{\CC}{{\mathbb{C}}}

\newcommand{\ZZ}{{\mathbb{Z}}}

\newcommand{\QQ}{{\mathbb{Q}}}

\newcommand{\NN}{{\mathbb{N}}}

\DeclareMathOperator{\GL}{GL}

\DeclareMathOperator{\SO}{SO}
\DeclareMathOperator{\Sp}{Sp}

\DeclareMathOperator{\pt}{pt}

\DeclareMathOperator{\id}{id}

\theoremstyle{plain}
\newtheorem{theorem}{Theorem}[section]

\newtheorem{proposition}[theorem]{Proposition}

\newtheorem{corollary}[theorem]{Corollary}

\theoremstyle{definition}
\newtheorem{definition}[theorem]{Definition}
\newtheorem{example}[theorem]{Example}

\theoremstyle{remark}
\newtheorem{remark}[theorem]{Remark}

\DeclareMathOperator{\wt}{wt}
\DeclareMathOperator{\var}{var}
\DeclareMathOperator{\word}{word}
\DeclareMathOperator{\PD}{PD}
\DeclareMathOperator{\DPD}{DPD}
\DeclareMathOperator{\SSYT}{SSYT}
\DeclareMathOperator{\CYT}{CYT}

\newcommand{\Sc}{\mathcal S}
\newcommand{\Bc}{\mathcal B}
\newcommand{\Cc}{\mathcal C}
\newcommand{\BCc}{\mathcal{BC}}
\newcommand{\Dc}{\mathcal D}
\newcommand{\Fc}{\mathcal F}
\newcommand{\Sf}{\mathfrak S}
\newcommand{\Bf}{\mathfrak B}
\newcommand{\Cf}{\mathfrak C}
\newcommand{\Df}{\mathfrak D}
\newcommand{\Ff}{\mathfrak F}

\newcommand{\Ar}{\mathscr A}

\newcommand{\pp}{\mathbf p}
\newcommand{\xx}{\mathbf x}
\newcommand{\zz}{\mathbf z}
\newcommand{\ttt}{\mathbf t}
\newcommand{\aaa}{\mathbf a}
\newcommand{\bb}{\mathbf b}

\title{Pipe dreams for Schubert polynomials of the classical groups}
\author{Evgeny Smirnov}
\email{esmirnov@hse.ru}
\address{HSE University, Russian Federation, ul. Usacheva 6, 119048 Moscow, Russia}
\address{Independent University of Moscow, Bolshoi Vlassievskii per. 11, 119002 Moscow, Russia}
\author{Anna Tutubalina}
\email{anna.tutubalina@gmail.com}
\address{HSE University, Russian Federation, ul. Usacheva 6, 119048 Moscow, Russia}
\date{\today}

\begin{document}

\setcounter{MaxMatrixCols}{30}

\begin{abstract}
 Schubert polynomials for the classical groups were defined by S.\,Billey and M.\,Haiman in 1995; they are polynomial representatives of Schubert classes in a full flag variety over a classical group. We provide a combinatorial description for these polynomials, as well as their double versions, by introducing analogues of pipe dreams, or RC-graphs, for Weyl groups of classical types. 
\end{abstract}

\maketitle

\tableofcontents

\section{Introduction}\label{sec:intro}

\subsection{Flag varieties and Schubert polynomials} Let $G$ be a reductive group over $\CC$. Let us fix a Borel subgroup $B$ in $G$ and the corresponding maximal torus $T\subset B$.  A classical result of Borel~\cite{Borel53} states that the cohomology ring $H^*(G/B,\ZZ)$ is isomorphic as a graded ring to the \emph{coinvariant} ring of the Weyl group $W$ of $G$, i.e. to the quotient of the polynomial ring in $\dim T$ variables modulo the ideal generated by $W$-invariants of positive degree. 

The cohomology ring of $G/B$ has a nice additive basis formed by \emph{Schubert classes} $\sigma_w$: the classes of Schubert varieties, i.e. of the closures of $B$-orbits in $G/B$. It is indexed by the elements of the Weyl group $w\in W$. A natural question is to construct a system of polynomial representatives of Schubert classes.

This question was answered in the case $G=\GL_n$ by I.\,N.\,Bernstein, I.\,M.\,Gelfand and S.\,I.\,Gelfand \cite{BernsteinGelfandGelfand73} who showed that the polynomials representing classes of Schubert varieties can be obtained from a polynomial representing the top class (the class of a point) by a sequence of \emph{divided difference operators}.  A.\,Lascoux and M.-P.\,Sch\"utzenberger~\cite{LascouxSchutzenberger82} considered one particularly nice choice of the top class and defined a system of polynomials $\Sf_w\in\ZZ[z_1,\dots,z_n]$, the \emph{Schubert polynomials}, with many good combinatorial and geometric properties. Here $w\in\Sc_n$ runs over the group of permutations in $n$ letters, i.e. the Weyl group of $\GL_n$, 

This system of polynomials is \emph{stable} in the following sense: for the embedding $\Sc_n\hookrightarrow \Sc_{n+1}$, the representative $\Sf_w$ of a given permutation $w\in\Sc_n\subset \Sc_{n+1}$ does not depend upon $n$. So the polynomials $\Sf_w$ can be viewed as elements of the ring $\ZZ[z_1,z_2,\dots]$ in countably many variables, indexed by the \emph{finitary} permutations $w\in\Sc_\infty=\varinjlim\Sc_n$. In~\cite{LascouxSchutzenberger82}, it is shown that these polynomials form a $\ZZ$-\emph{basis} of this ring (also cf.~\cite{Macdonald91}).

Geometrically this can be interpreted as follows. The standard embedding $\GL_n\hookrightarrow\GL_{n+1}$ together with the embedding of Borel subgroups $B_n\hookrightarrow B_{n+1}$ defines an embedding of flag varieties $G_n/B_n\hookrightarrow G_{n+1}/B_{n+1}$. This gives a surjective map $H^*(G_{n+1}/B_{n+1},\ZZ)\to H^*(G_n/B_n,\ZZ)$. This map is compatible with the Schubert classes: for each $w\in\Sc_\infty$, there exists a  \emph{stable} Schubert class $\sigma_w^{(\infty)}=\varprojlim \sigma_w^{(n)}\in \varprojlim H^*(G_n/B_n,\ZZ)$. \emph{A priori} this class is a homogeneous power series in $z_1,z_2,\dots$. However it turns out that it can be represented by a unique \emph{polynomial}, which is the Schubert polynomial $\Sf_w$. The elements $\{\Sf_w\}$ are thus obtained as the \emph{unique} solutions of an infinite system of equations involving divided difference operations. 

The polynomials $\Sf_w$ have nonnegative integer coefficients; their nonnegativity is completely unobvious from their definition involving divided difference operators. Their combinatorial description was obtained independently by S.\,Fomin and An.\,Kirillov~\cite{FominKirillov96} and S.\,Billey and N.\,Bergeron~\cite{BilleyBergeron93}. The monomials in $\Sf_w$ are indexed by certain diagrams, called \emph{pipe dreams} or \emph{rc-graphs}, see Sec.~\ref{ssec:pdtypea} below. A geometric interpretation of pipe dreams was obtained by A.\,Knutson and E.\,Miller in~\cite{KnutsonMiller05}: they showed that the pipe dreams for a permutation $w$ are in bijection with the irreducible components of a certain flat Gr\"obner degeneration of the corresponding matrix Schubert variety $\overline{X_w}$.

There is an obvious analogy between Schubert and Schur polynomials: the monomials of the latters are indexed by semistandard Young tableaux. In fact, if $w$ is a Grassmannian permutation, i.e. it has a unique descent, the corresponding Schubert polynomial equals the Schur polynomial $s_\lambda(w)$, where $\lambda(w)$ is a partition obtained from $w$ (see~Sec.\,\ref{ssec:grassmannschur} for details). This equality can be easily obtained from the following geometric argument: the Schur polynomials represent the classes of Schubert varieties in Grassmannians $G/P$, where $P$ is a maximal parabolic group in $G$, and the map $H^*(G/P)\to H^*(G/B)$ sends them to Schubert classes of Grassmannian permutations in a full flag variety. There is also a purely combinatorial proof: one can establish a bijection between the Young tableaux indexing the monomials in $s_{\lambda(w)}$ and the pipe dreams indexing the monomials in $\Sf_w$. We recall this proof in Theorem~\ref{thm:a_grass}; see also~\cite{Knutson12}. 

In~\cite{BilleyBergeron93} the authors introduced the notion of  the \emph{bottom pipe dream} for each Schubert polynomial. It is a maximal pipe dream according to some partial order on pipe dreams, defined in combinatorial terms; such a pipe dream exists and is unique for each permutation. This allowed them to show that the basis change matrix between the Schubert polynomials and the monomial basis in $\ZZ[z_1,z_2,\dots]$ is unitriangular, and hence Schubert polynomials form a $\ZZ$-basis in $\ZZ[z_1,z_2,\dots]$.

The construction of Schubert polynomials can be extended as follows. Instead of the cohomology ring $H^*(G/B,\ZZ)$ we can consider the $T$-equivariant cohomology ring $H^*_T(G/B,\ZZ)$. The map $G/B\to \pt$ defines the module structure on $H^*_T(G/B,\ZZ)$ with respect to the polynomial ring $\ZZ[t_1,\dots,t_n]=H^*_T(\pt)$. One can be interested in polynomial representatives of the $T$-equivariant classes of $[X_w]$. These are also classical objects, known as \emph{double Schubert polynomials} $\Sf_w(z,t)$; these are homogeneous polynomials in $2n$ variables: $z_1,\dots,z_n$ and $t_1,\dots,t_n$. The specialization $t_i=0$ gives us the usual Schubert polynomials  $\Sf_w(x)$ in $n$ variables $z_1,\dots,z_n$. Double Schubert polynomials also have a description in terms of pipe dreams.

\subsection{The case of symplectic and orthogonal groups}

It is interesting to replace $\GL_n$ by another reductive complex algebraic group and ask the same series of questions. We will be interested in the classical groups of types $B_n$, $C_n$, and $D_n$: these are $\SO_{2n+1}$, $\Sp_{2n}$, and $\SO_{2n}$, respectively.  The Weyl group $W$ for such a group is a hyperoctahedral group, or the group of signed permutations on $n$ letters, for the types $B$ and $C$, or, in the type $D$, the subgroup of this group consisting only of signed permutations with even number of sign changes.

The (generalized) full flag variety $G/B$ also has a cellular decomposition $G/B=\bigsqcup_{w\in W} BwB/B$, with the cells indexed by the elements of the Weyl group $W$; and the cell closures $X_w=\overline{BwB/B}$ define a basis $\sigma_w\in H^*(G/B,\ZZ)$. 

As in the case of $\GL_n$, one can consider the embeddings $G_n\hookrightarrow G_{n+1}$, where $G_n$ and $G_{n+1}$ are classical groups of the same type ($B$, $C$ or $D$) and of the rank $n$ and $n+1$, respectively. We can fix an element $w\in W_\infty=\bigcup W_n$ of the ``limit Weyl group'' and take the projective limit of Schubert cycles $\sigma^{(\infty)}_w\in\varprojlim H^*(G_n/B_n,\ZZ)$. The question is as follows: what is a ``nice'' representative of this class in the ring of power series $\ZZ[[z_1,z_2,\dots]]$?

In~\cite{BilleyHaiman95} S.\,Billey and M.\,Haiman define \emph{Schubert polynomials for the classical groups}. They show that $\sigma^{\infty}_w$ is represented by a unique element $\Ff_w$ of the subring $\ZZ[z_1,z_2,\dots,p_1(\zz),p_3(\zz),\dots]$, where $p_{k}(\zz)=\sum_{i=1}^\infty z_i^k$ are the Newton power sums.  Note that we need to take only the odd power sums; the images of the even power sums vanish in each of the coinvariant rings for each of the types $B$, $C$, and $D$. Similarly to $\Sf_w$, the functions $\Ff_w$ are obtained as unique solutions of an infinite system of equations involving divided difference operators.  They also form a basis of the corresponding ring. 

Schubert polynomials for the classical groups can be also expressed via the usual Schubert polynomials and the \emph{Stanley symmetric functions}. This implies that they are nonnegative integer combinations of monomials in $z_1,z_2,\dots$ and $p_1,p_3,\dots$.

Like in the previous case, we can consider the Schubert classes coming from cycles on Lagrangian/orthogonal Grassmannians. The results of P.\,Pragacz~\cite{Pragacz91} imply that the corresponding Schubert polynomials are equal to Schur's $P$- and $Q$-functions; the details follow in~Sec.~\ref{ssec:grassmannpq}. 

One also can introduce double versions of Schubert polynomials, which represent the classes of $[X_w]$ in the $T$-equivariant cohomology ring $H^*_T(G/B,\ZZ)$, where $T\cong(\CC^*)^n$ is a maximal torus in $G$. They depend upon $z_1,z_2,\dots$, $p_1,p_3,\dots$, and another set of variables $t_1,t_2,\dots$ Formulas for double Schubert polynomials $\Ff_w(\zz,\pp,\ttt)$ were found by T.\,Ikeda, L.\,Mihalcea, and H.\,Naruse in~\cite{IkedaMihalceaNaruse11}.

Finally, in the paper~\cite{KirillovNaruse17}, An.\,Kirillov and H.\,Naruse provide a combinatorial construction of Schubert polynomials for classical Weyl groups outside the type $A$. They define analogues of pipe dreams whose generating functions are equal to (ordinary or equivariant) Schubert polynomials.  



Let us also mention an earlier paper~\cite{FominKirillov96b} by S.\,Fomin and An.\,Kirillov, where the authors were constructing Schubert polynomials of the type $B$ in a completely different way: they were looking for a family of polynomials indexed by the hyperoctahedral group and satisfying certain five properties, similar to those of Schubert polynomials in the type $A$ (stability, nonnegativity of coefficients etc.). They have shown that such a family of polynomials does not exist; then they considered families of polynomials defined by all but one of these properties. This led them to several different families of polynomials. We should note that they are also different from the Schubert polynomials of type $B$ defined by Billey and Haiman and considered in this paper.

\subsection{Our results} This paper is devoted to combinatorial study of Schubert polynomials for groups of the types $B$, $C$, and $D$. We provide an alternative construction of pipe dreams for these cases; it is given in~Sec.\,\ref{sec:main}. 
The pipe dreams are configurations of strands, similar to those in the type $A$; some of the strands may be equipped with an additional element, called \emph{faucet}, which represents the sign change. To such a pipe dream we can associate a permutation $w\in W$, called the \emph{shape} of pipe dream, and a monomial; our main results, Theorems~\ref{bdreams}, \ref{cdreams} and~\ref{ddreams}, state that the Schubert polynomial $\Ff_w(\zz,\pp)$ equals the sum of monomials for all pipe dreams of shape $w\in W$ of type $B$, $C$ or $D$, respectively. We also give a similar description of double Schubert polynomials of these types in Corollary~\ref{cor:double}.

Our construction of pipe dreams is similar to the one by Kirillov and Naruse; the main difference is that in their setting each pipe dream corresponds to a couple of terms in the Schubert polynomial (namely, to the product of several binomials), as opposed to one monomial in our setting. We provide a detailed comparison of these two constructions in Section~\ref{sec:kirnar}.

Then we introduce a notion of admissible moves on pipe dreams of a given type; these are certain transformations not changing the shape of a pipe dream, that turn it into a poset. We show that each pipe dream for a permutation of each of the types $B$--$D$ can be reduced to a certain canonical form, referred to as the bottom pipe dream, by a sequence of admissible moves, similar to the ladder moves in the type $A$; thus the poset of pipe dreams corresponding to a given Weyl group element is shown to have a unique minimal element. The existence of bottom pipe dreams is established in Theorems~\ref{bbottom},~\ref{cbottom}, and~\ref{dbottom}. This allows us to give a new proof of Theorem~\ref{thm:BCDbasis}, which states that the Schubert polynomials (of each type) form a basis of the ring $\QQ[\zz,p_1,p_3,\dots]$.

In the last section we study Grassmannian permutations; we give a bijective proof of the fact that the Schubert polynomial of such a permutation is Schur's $P$- or $Q$-function (Theorem~\ref{PQschub}). For this we recall the definition of the latter functions involving shifted Young tableaux and establish a bijection between these tableaux and the pipe dreams of the corresponding Grassmannian permutation.

\subsection{Related questions: bumpless pipe dreams and specializations of Schubert polynomials} Recently, T.\,Lam, S.\,J.\,Lee, and M.\,Shimozono~\cite{LamLeeShimozono21} introduced objects called \emph{bumpless pipe dreams}. They also can be used to obtain a formula for double Schubert polynomials; however, there is no weight-preserving bijection between them and the usual pipe dreams, so the two presentations are genuinely different. 

In her recent preprint~\cite{Weigandt21}, A.\,Weigandt defines a generalization of bumpless pipe dreams, which provide a similar description of Grothendieck polynomials, and shows that they correspond to alternating sign matrices. It would be interesting to find analogues of these results, as well as the notion of bumpless pipe dreams, for the groups of classical types outside the type $A$.

Another series of questions concerns specializations of Schubert polynomials. In~\cite{Macdonald91}, I.\,Macdonald conjectured a formula for the principal specialization of a Schubert polynomial; this formula was proven algebraically by S.\,Fomin and R.\,Stanley~\cite{FominStanley91}. Recently, a combinatorial proof was found by S.\,Billey, A.\,Holroyd, and B.\,Young~\cite{BilleyHolroydYoung19}. Taking an appropriate limit of Schubert polynomials gives the so-called \emph{backstable Schubert polynomials}, also defined in the aforementioned paper~\cite{LamLeeShimozono21}. Similarly to Schubert polynomials of types $B$, $C$ and $D$, they depend upon two sets of variables, being polynomials in the first set and symmetric functions in the second set of variables. They also admit a similar expression for their principal specializations. E.\,Marberg and B.\,Pawlowski~\cite{MarbergPawlowski21} prove analogues of Macdonald's formula for Schubert polynomials of types $B$, $C$, and $D$. A natural (however challenging) question is to find a combinatorial proof of their results by means of pipe dreams, either those defined in this paper or the bumpless ones.

\subsection{Structure of the paper} This text is organized as follows. In Sec.~\ref{sec:prelim} we recall the main definitions and notions concerning the Weyl groups of classical types, the Schubert polynomials for these groups, as well as the construction of pipe dreams in the type $A$. In Sec.~\ref{sec:main} we describe the constructions of pipe dreams for Schubert polynomials of the types $B$, $C$, and $D$ (Sec.~\ref{ssec:mainc}--\ref{ssec:infinite}). In~Sec.~\ref{ssec:double} we are dealing with the double Schubert polynomials; we provide a generalization of the previous construction to this case. Finally, in~Sec.~\ref{ssec:examples} we compute several examples of Schubert polynomials. Sec.~\ref{sec:bottom} is devoted to the proof of existence of bottom pipe dreams. Sec.~\ref{sec:kirnar} is devoted to comparing our construction with the ``excited extended Young diagrams" by Kirillov and Naruse: we recall their construction in Sec.~\ref{ssec:kirnardef} and relate it to the ours in~Sec.~\ref{ssec:kirnarbijection}. In~Sec.~\ref{ssec:lehmer} we introduce the admissible moves, and in Sec.~\ref{ssec:bottoma} we recall the situation in the type $A$. The next three subsections are devoted to the cases $B$, $C$, and $D$ respectively. The last section of this paper,  Sec.~\ref{sec:grassmann}, is devoted to Grassmannian permutations; we show that the Schubert polynomial of such a permutation is a $P$- or $Q$-Schur function. In the appendix we provide more examples: namely, we list all the elements of the group $\BCc_2$ and some elements of the group $\Dc_3$ and for each of these elements list all its pipe dreams.

\subsection*{Acknowledgements} We are grateful to the anonymous referees for their useful comments and remarks, in particular, for bringing the paper~\cite{KirillovNaruse17} to our attention. This research was supported by the HSE University Basic Research Program and the Theoretical Physics and Mathematics Advancement Foundation ``BASIS''. E.S. was also partially supported by the RFBR grant 20-01-00091-a and the Simons-IUM Fellowship.

\section{Preliminaries}\label{sec:prelim}

\subsection{Weyl groups of the classical types}\label{ssec:weylgroups}

Let $\QQ[\zz]=\QQ[z_1,z_2,\dots]$ be the ring of polynomials in countably many variables $z_1,z_2,\dots$. Let $p_k=z_1^k+z_2^k+\dots$ denote the $k$-th power sum; this is not a polynomial, but rather a symmetric function in the $z_i$'s. Consider the ring of power series $\QQ[\zz,p_1,p_3,\dots]=\QQ[z_1,z_2,\dots,p_1,p_3,\dots]$ which are polynomial in the $z_i$ and the $p_k$ for $k$ odd. All its generators are algebraically independent, so this ring can be viewed just as the polynomial ring in $z_1,z_2,\dots$ and $p_1,p_3,\dots$.

Denote by $\Sc_n$ the symmetric group in $n$ variables, and let $\Sc_\infty=\varinjlim\Sc_n$ be the group of finitary permutations of $\ZZ_{>0}$, i.e. the group of permutations fixing all but finitely many points. It is generated by the simple transpositions $s_i=(i\leftrightarrow i+1)$. The group $S_\infty$ acts on $\QQ[\zz,\dots,p_1,p_3,\dots]$ by permuting the $z_i$'s. 

Let us also consider the  group of signed permutations $\BCc_n$ (also called the \emph{hyperoctahedral group}). It can be viewed as the group of permutations $w$ of the $2n$-element set $\{1,-1, \dots,n,-n\}$ satisfying the condition $w(i)=j$ iff $w(-i)=-j$. The groups $\BCc_n$ are embedded one into another in the obvious way, so we can consider the injective limit $\BCc_\infty$, which is the group of finitary signed permutations. The standard generators for $\BCc_\infty$ are $s_i=(i\leftrightarrow i+1)$ for $i\geq 1$ and $s_0=(1\leftrightarrow -1)$. The group $\BCc_\infty$ acts on the formal power series ring $\QQ[[z_1,z_2,\dots]]$ by letting $s_i$ interchange $z_i$ and $z_{i+1}$ for $i\geq 1$, and by letting $s_0$ send $z_1$ into $-z_1$. We can restrict this action to the rings $\QQ[\zz]$ and $\QQ[\zz,p_1,p_3,\dots]$ and see that the $s_i$ for $i\geq 1$ fix the $p_k$, while $s_0p_k=p_k-2z_1^k$. 

We denote by $\Dc_n$ the subgroup of $\BCc_n$ of signed permutations with an even number of sign changes. The union of these subgroups is denoted by $\Dc_\infty$. The standard generators for these groups are $s_i$ for $i\geq 1$ and an additional generator $s_{\hat 1}=s_0s_1s_0$ which replaces $z_1$ with $-z_2$ and $z_2$ with $-z_1$.

These groups satisfy the relations $s_i^2=\id$ and the \emph{Coxeter relations} (cf., for instance,~\cite{Bourbaki46} or~\cite{Humphreys90}):

\begin{description}
	\item[$\Sc_\infty$] $s_is_j=s_js_i$ for $|i-j|\geq 2$, $s_is_{i+1}s_i=s_{i+1}s_is_{i+1}$ for $i\geq 1$;
	
	\item[$\BCc_\infty$]  $s_is_j=s_js_i$ for $|i-j|\geq 2$, $s_is_{i+1}s_i=s_{i+1}s_is_{i+1}$ for $i\geq 1$, $s_0s_1s_0s_1=s_1s_0s_1s_0$.

	\item[$\Dc_\infty$]  $s_is_j=s_js_i$ for $|i-j|\geq 2$, $s_is_{i+1}s_i=s_{i+1}s_is_{i+1}$ for $i\geq 1$ (here $i,j\in\{\hat 1,1,2,\dots\}$; while performing arithmetic operations we treat $\hat 1$ as $1$).
\end{description}

We will use the standard terminology for Coxeter groups. Let $\Fc_\infty$ be one of the groups $\Sc_\infty$, $\BCc_\infty$, or $\Dc_\infty$. A \emph{word} $(s_{i_1},\dots,s_{i_k})$ is a finite sequence of its generators. We say that this word \emph{represents} an element $w\in\Fc_\infty$, if $w=s_{i_1}\dots s_{i_k}$. The minimal number $k$ such that there exists a $k$-element word representing $w$ is called the \emph{length} of $w$ and denoted by $k=\ell(w)$.

Often we will use the one-line notation (not to be confused with the cycle notation) for elements $w\in\Sc_n$, writing a permutation $w$ as a sequence $w(1)w(2)\dots w(n)$. For instance, $4132\in\Sc_4$ maps $1$ into $4$, $2$ into $1$, $3$ into itself and $4$ into $2$.

For signed permutations $w\in\BCc_n$ we will write negative numbers $-m$ as $\overline m$ in one-line notation. For example, $3\bar 2\bar 4 1\in \BCc_4$ maps $1$ into $3$, $2$ into $-2$, $3$ into $-4$, and $4$ into $1$.

\subsection{Schubert polynomials for the classical groups}\label{ssec:schubclass}

Let us recall the definition of Schubert polynomials from \cite{LascouxSchutzenberger82} and the definition of Schubert polynomials for the classical groups from \cite{BilleyHaiman95}.

\begin{definition} Define the \emph{divided difference operators} on the rings $\QQ[\zz]$ and $\QQ[\zz,p_1,p_3,\dots]$ as follows:
\begin{eqnarray*}
	\partial_{i} f=\frac{f-s_i f}{z_{i}-z_{i+1}}&  \text{for} &i\geq 1;\\
	\partial_{0} f=\frac{f-s_0 f}{-2 z_{1}};\\
	\partial_{0}^{B} f=2\partial_0 f=\frac{f-s_0 f}{-z_{1}};\\
	\partial_{\hat 1} f=\frac{f-s_{\hat 1} f}{-z_{1}-z_{2}}.
\end{eqnarray*}
\end{definition}

\begin{definition}
	\emph{Schubert polynomials} (of type $A$) are homogeneous polynomials $\Sf_w(\zz)\in\QQ[\zz]$ indexed by the permutations $w\in\Sc_\infty$ and satisfying the relations
\begin{eqnarray*}
	\Sf_{\id}&=&1;\\
	\partial_i\Sf_w&=&\begin{cases} \Sf_{ws_i}&\text{if}\quad \ell(ws_i)<\ell(w);\\
	0&\text{otherwise}
	\end{cases}
\end{eqnarray*}
for each $i\geq 1$.
\end{definition}

\begin{definition}
	\emph{Schubert polynomials of type $C$} are elements $\Cf_w(\zz,\pp)\in\QQ[\zz,p_1,p_3,\dots]$ that are homogeneous in $\zz$, indexed by the permutations $w\in\BCc_\infty$ and satisfy the relations
\begin{eqnarray}
	\Cf_{\id}&=&1;\nonumber\\
	\partial_i\Cf_w&=&\begin{cases} \Cf_{ws_i}&\text{if}\quad \ell(ws_i)<\ell(w);\\
	0&\text{otherwise}\label{eq:deftypeC}
	\end{cases}
\end{eqnarray}
for each $i\geq 0$.
\end{definition}

\begin{definition}\emph{Schubert polynomials of type $B$} are defined as $\Bf_w=2^{-s(w)}\Cf_w$, where $s(w)$ is the number of entries from $\{1,2,\dots\}$ changing their sign under the action of $w$. They also can be defined by the relations obtained from~(\ref{eq:deftypeC}) by replacing $\partial_0$ by $\partial_0^B$. 
\end{definition}

\begin{definition} \emph{Schubert polynomials of type $D$}
 are elements $\Df_w(\zz,\pp)\in\QQ[\zz,p_1,p_3,\dots]$ that are homogeneous in $\zz$,  indexed by the permutations $w\in\Dc_\infty$ and satisfy the relations
\begin{eqnarray*}
	\Df_{\id}&=&1;\\
	\partial_i\Df_w&=&\begin{cases} \Df_{ws_i}&\text{if}\quad \ell(ws_i)<\ell(w);\\
	0&\text{otherwise}
	\end{cases}
\end{eqnarray*}
for each $i=\hat 1, 1,2,\dots$.
\end{definition}

It was shown by S.\,Billey and M.\,Haiman~\cite{BilleyHaiman95} that such polynomials exist and are uniquely determined by these properties. We recall these results below as~Theorems~\ref{thm:sschub}--\ref{thm:dschub}.

For convenience we replace the ring $\QQ[\zz,p_1(\zz),p_3(\zz),\ldots]$ by an isomorphic ring $\QQ[\zz,p_1(\xx),p_3(\xx),\ldots]$. Here $p_k(\xx)=\sum_{i=1}^\infty x_i^k$, and the isomorphism is the unique ring isomorphism with $z_i\mapsto z_i$ and  $p_k(\xx)\mapsto -p_k(\zz)/2$. Then the Schubert polynomials of types $B$, $C$, and $D$ are polynomials in $z_1,z_2,\ldots$ and symmetric functions in $x_1,x_2,\ldots$. 

It is easy to see that the generators of the groups $\BCc_\infty$ and $\Dc_\infty$ act on the ring $\QQ[\zz,p_1(\xx),p_3(\xx),\ldots]$ in the following way:
\begin{eqnarray*}
	s_if(z_1,\ldots,z_i,z_{i+1},\ldots;x_1,x_2,\ldots)&=&f(z_1,\ldots,z_{i+1},z_i,\ldots;x_1,x_2,\ldots)\quad\text{for $i\geq 1$;}\\
	s_0f(z_1,z_2,z_3,\ldots;x_1,x_2,\ldots)&=&f(-z_1,z_2,z_3\ldots;z_1,x_1,x_2,\ldots);\\
	s_{\hat 1}f(z_1,z_2,z_3\ldots;x_1,x_2,\ldots)&=&f(-z_2,-z_1,z_3\ldots;z_1,z_2,x_1,x_2,\ldots).
\end{eqnarray*}

\subsection{Stanley symmetric functions and Schubert polynomials}\label{ssec:stanley}

In this subsection we discuss the expression of Schubert polynomials via Stanley symmetric functions. We start with introducing some notation.

	\begin{itemize}
		\item Let $w$ be an element of $\Sc_\infty$, $\BCc_\infty$  or $\Dc_\infty$. Denote the set of all reduced words for $w$ by $R(w)$.
		\item 	Let $\aaa=s_{a_1}s_{a_2}\ldots s_{a_l}\in R(w)$ be a reduced word for $w\in \BCc_\infty$ or $w\in\Dc_\infty$. We define the set of \emph{peaks} of $\aaa$ as the set
\[
P(\aaa)=\{i \in\{2,3,\ldots, l-1\}\mid a_{i-1}<a_{i}>a_{i+1}\}.
\]
For the group  $\Dc_\infty$ we assume that $\hat 1=1<2<3<\ldots$
		\item A weakly decreasing sequence of nonnegative integers  $j_1\geq j_2\geq \ldots \geq j_l$ is said to be \emph{$x$-admissible} for the word $\aaa$, if $j_{i-1}=j_{i}=j_{i+1}$ implies that $i\not\in P(\aaa)$. 

		To each $x$-admissible sequence $j_1\geq j_2\geq \ldots \geq j_l$ we can assign a monomial $x_{j_1}x_{j_2}\ldots x_{j_l}=x_1^{\alpha_1}\ldots x_m^{\alpha_m}=\xx^{\alpha}$. Such monomials are also called \emph{$x$-admissible}. We denote the set of all $x$-admissible monomials for $\aaa$ by $\Ar_{x}(\aaa)$.
		\item 	Let $w\in \Sc_n$ and let $\aaa=s_{a_1}s_{a_2}\ldots s_{a_l}\in R(w)$ be a reduced word for $w$. A weakly increasing sequence of positive integers  $j_1\leq j_2\leq\ldots \leq j_l$ is said to be  \emph{$z$-admissible} for  $\aaa$ if for each $i$ we have $j_i\leq a_i$, and the equality $j_i=j_{i+1}$ implies that $a_i>a_{i+1}$.
		
		To each $z$-admissible sequence $j_1\leq j_2\leq \ldots \leq j_l$ we can assign a monomial $z_{j_1}z_{j_2}\ldots z_{j_l}=z_1^{\beta_1}\ldots z_m^{\beta_m}=\zz^{\beta}$. Such monomials are also called \emph{$z$-admissible}; we denote the set of $z$-admissible monomials for $\aaa$ by $\Ar_{z}(\aaa)$.
		
		\item Denote by $i(\alpha)$ the number of distinct variables occuring in  $\xx^{\alpha}$ in nonzero powers.		
		\item For the group $\Dc_\infty$ we denote by $o(\aaa)$ the total number of occurences of the letters  $s_1$ and $s_{\hat 1}$ in $\aaa$.
	\end{itemize}

\begin{definition}
The \emph{Stanley symmetric functions} are defined as follows. For $w\in\BCc_\infty$ we define 
	$$
	F_{w}(\xx)=\sum_{\substack{\aaa\in R(w) \\ \xx^{\alpha} \in \Ar_{x}(\aaa)}} 2^{i(\alpha)} \xx^{\alpha}.
	$$
For $w\in\Dc_\infty$ we define
	$$
	E_{w}(\xx)=\sum_{\substack{\aaa \in R(w) \\ \xx^{\alpha} \in \Ar_{x}(\aaa)}} 2^{i(\alpha)-o(\aaa)} \xx^{\alpha}.
	$$
\end{definition}

The following theorem, describing the relation of Schubert polynomials and Stanley symmetric functions, is due to S.\,Fomin and R.\,Stanley~\cite{FominStanley91} and S.\,Billey, W.\,Jockush and R.\,Stanley~\cite{BilleyJockushStanley93}.

	\begin{theorem}[{\cite[Thm.~1.1]{BilleyJockushStanley93}}]\label{thm:sschub}
		For each $w\in \Sc_\infty$ we have
		$$
		\Sf_{w}\left(\zz\right)=\sum_{\substack{\aaa \in R(w)\\ \zz^{\beta} \in \Ar_{z}(\aaa)}} \zz^{\beta}.
		$$
	\end{theorem}

Its analogues for the cases of $\BCc_\infty$ and $\Dc_\infty$ were obtained by S.\,Billey and M.\,Haiman~\cite{BilleyHaiman95}.
	
	\begin{theorem}[{\cite[Thm~3A]{BilleyHaiman95}}]\label{thm:cschub} Let $w\in\BCc_\infty$. 
The Schubert polynomial $\Cf_w$ satisfies the equality		
\[
	\Cf_{w}(\zz; \xx)=\sum_{\substack{u v=w \\ {\ell(u)+\ell(v)=\ell(w)} \\ v \in \Sc_{n}}} F_{u}(\xx) \Sf_{v}(\zz)=\sum_{\substack{u v=w\\ \ell(u)+\ell(v)=\ell(w) \\ v \in \Sc_{n}}} \sum_{\substack{\aaa \in R(u)\\ \xx^{\alpha} \in \Ar_{x}(\mathbf{a})}} \sum_{\substack{\bb \in R(v)\\ \zz^{\beta} \in \Ar_{z}(\bb)}} 2^{i(\alpha)} \xx^{\alpha} \zz^{\beta}.
\]	
\end{theorem}

	\begin{theorem}[{\cite[Thm~4A]{BilleyHaiman95}}]\label{thm:dschub} Let $w\in\Dc_\infty$. 
The Schubert polynomial $\Df_w$ satisfies the equality
\[
\Df_{w}(\zz; \xx)=\sum_{\substack{u v=w \\ {\ell(u)+\ell(v)=\ell(w)} \\ v \in \Sc_{n}}} E_{u}(\xx) \Sf_{v}(\zz)
	=\sum_{\substack{u v=w\\ \ell(u)+\ell(v)=\ell(w) \\ v \in \Sc_{n}}} \sum_{\substack{\aaa \in R(u)\\ \xx^{\alpha} \in \Ar_{x}(\mathbf{a})}} \sum_{\substack{\bb \in R(v)\\ \zz^{\beta} \in \Ar_{z}(\bb)}} 2^{i(\alpha)-o(\aaa)} \xx^{\alpha} \zz^{\beta}.
\]
\end{theorem}

\subsection{Pipe dreams in type $A$}\label{ssec:pdtypea}

In this subsection we briefly recall the construction of pipe dreams for the symmetric group $\Sc_n$. We fix a positive integer $n$ and consider the staircase Young diagram with rows of length $n-1,\dots,2,1$. We will refer to this diagram  as the \emph{base}   $B_{\Sc_n}$. Let us index all the boxes in the $i$-th row by the variable $z_i$ (cf.~Fig.~\ref{A-base}).

Now let us fill this diagram by elements of two types:  \emph{elbow joints} $\elbows$ and \emph{crosses} $\cross$.  We also put a single elbow  $\elbow$ to the right of the rightmost box of each row.  We obtain a diagram consisting of $n$ strands, or \emph{pipes} (hence the name ``pipe dream'').

\begin{figure}[ht]
\includegraphics{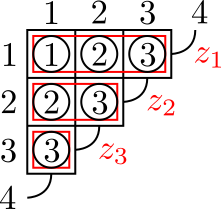}
\caption{Base diagram for pipe dreams of size $n=4$}
\label{A-base}
\end{figure}

For each box $\boxx$ we introduce the following notation:
\begin{itemize}
	\item $\wt(\boxx)$ is the weight of an element in this box. It is equal to $1$ for a cross $\cross$ and to $0$ for an elbow joint $\elbows$.
	\item $\sigma(\boxx)\in\Sc_n$ is an element  equal to $s_i$ for a cross in the box with the number $i$ and to the identity permutation $\id$ for an elbow.
	\item $\var(\boxx)$ is the variable corresponding to this box, i.e. $z_i$, where $i$ is the row number.
\end{itemize}

To each pipe dream $D$ we can assign:
\begin{itemize}
	\item a monomial $\zz^{\beta(D)}=\prod_{\boxx\in B_{\Sc_n}}\var(\boxx)^{\wt(\boxx)}$
	(i.e. the product of all variables $z_i$ over all the crosses in $D$);
	\item A word $\word(D)$ obtained by writing out all the letters $\sigma(\boxx)$ while reading all the boxes in the base from right to left, from top to bottom;	
	\item If  $\word(D)$ is a reduced word for the permutation $w\in\Sc_n$, then the pipe dream $D$ is said to be \emph{reduced}. We shall say that $w=w(D)$ is the \emph{shape} of $D$.
\end{itemize}
	
	This definition can be restated as follows: a pipe dream is said to be reduced if each two strands cross at most once. Its shape is the permutation $w\in\Sc_n$ such that the strands connect the numbers $1,2,\ldots,n$ on the left-hand side of the diagram with the numbers $w(1),w(2),\ldots,w(n)$ on its top.

\begin{example} Here are two examples of pipe dreams of size $n=4$.\vspace{0.5cm}
	
	\begin{minipage}{\textwidth}
		\begin{minipage}{0.49\textwidth}
			$$D_1=\begingroup
			\setlength\arraycolsep{0pt}
			\renewcommand{\arraystretch}{0.06}
			\begin{matrix}
			& 1 & 2 & 3 & 4  \\
			1 & \cross & \cross & \cross & \elbow & \\
			2 & \elbows & \elbows & \elbow & \\
			3 & \cross & \elbow & \\
			4 & \elbow & \\
			\end{matrix}
			\endgroup$$
			$$\zz^{\beta(D_1)}=z_1^3z_3$$
			$$\word(D_1)=s_3s_2s_1s_3$$
			$$w(D_1)=4132\in\Sc_4$$
		\end{minipage}		
		\begin{minipage}{0.49\textwidth}
			$$D_2=\begingroup
			\setlength\arraycolsep{0pt}
			\renewcommand{\arraystretch}{0.06}
			\begin{matrix}
			& 1 & 2 & 3 & 4  \\
			1 & \elbows & \cross & \elbows & \elbow & \\
			2 & \cross & \cross & \elbow & \\
			3 & \cross & \elbow & \\
			4 & \elbow & \\
			\end{matrix}
			\endgroup$$
			$$\zz^{\beta(D_2)}=z_1z_2^2z_3$$
			$$\word(D_2)=s_2s_3s_2s_3 $$
			\begin{center}
				$D_2$ is non-reduced.
			\end{center}
		\end{minipage}
	\end{minipage}
\end{example}

Let us show that for a reduced pipe dream $D$ of shape $w$ the monomial  $\zz^{\beta(D)}=z_{j_1}z_{j_2}\ldots z_{j_l}$ is  $z$-admissible for the word $\word(D)=s_{a_1}s_{a_2}\ldots s_{a_l}$. 

Indeed, the $i$-th cross is positioned in the $j_i$-th row and $a_i$-th diagonal (crosses are ordered from right to left, from top to bottom, so $j_1\leq j_2\leq\ldots\leq j_l$). Since the numbers of the diagonals decrease while reading each row right to left, then $j_i=j_{i+1}$ implies $a_i>a_{i+1}$. Since the row number of the box does not exceed the number of the corresponding diagonal, we have $j_i\leq a_i$.

Let  $w\in\Sc_n$, $\aaa=s_{a_1}s_{a_2}\ldots s_{a_l}\in R(w)$, and $\zz^{\beta}=z_{j_1}z_{j_2}\ldots z_{j_l}\in\Ar_z(\aaa)$. There exists a unique pipe dream $D$ such that $\word(D)=\aaa$ and $\zz^{\beta(D)}=\zz^{\beta}$. It can be obtained by placing a cross at the intersection of the $j_i$-th row and the  $a_i$-th diagonal for each $i$.

It is clear that the set of all reduced pipe dreams of shape $w\in S_n$ does not depend upon $n$, i.e. is stable under the embedding $S_n\hookrightarrow S_{n+1}$. Denote this set by $\PD_{A}(w)$. Theorem~$\ref{thm:sschub}$ can be restated as follows.

\begin{theorem}[{\cite{BilleyBergeron93},\cite{FominKirillov96}}]\label{sdreams}
For a permutation $w\in\Sc_n$, we have
\[
\Sf_w(\zz)=\sum_{D\in\PD_{A}(w)}\zz^{\beta(D)}.
\]
\end{theorem}

\begin{example}
	Let $w=s_i\in\Sc_n$, $i<n$. If $D\in\PD_{A}(w)$ is a reduced pipe dream, then $D$ contains exactly one cross. It is located in the $i$-th diagonal; the row number of this cross can be equal to $1, 2,\dots,(i-1),i$. So
\[
\Sf_{s_i}(\zz)=z_1+z_2+\ldots+z_i.
\]
\end{example}
\begin{example}
	Let $w=1432\in\Sc_4$. There are five reduced pipe dreams of shape $w$:
\[
	\begingroup
	\setlength\arraycolsep{0pt}
	\renewcommand{\arraystretch}{0.06}
	\begin{matrix}
	& 1 & 2 & 3 & 4  \\
	1 & \elbows & \cross & \cross & \elbow & \\
	2 & \elbows & \cross & \elbow & \\
	3 & \elbows & \elbow & \\
	4 & \elbow & \\
	\end{matrix}\quad
	\begin{matrix}
	& 1 & 2 & 3 & 4  \\
	1 & \elbows & \cross & \cross & \elbow & \\
	2 & \elbows & \elbows & \elbow & \\
	3 & \cross & \elbow & \\
	4 & \elbow & \\
	\end{matrix}
	\quad\begin{matrix}
	& 1 & 2 & 3 & 4  \\
	1 & \elbows & \cross & \elbows & \elbow & \\
	2 & \cross & \cross & \elbow & \\
	3 & \elbows & \elbow & \\
	4 & \elbow & \\
	\end{matrix}
	\quad\begin{matrix}
	& 1 & 2 & 3 & 4  \\
	1 & \elbows & \elbows & \cross & \elbow & \\
	2 & \cross & \elbows & \elbow & \\
	3 & \cross & \elbow & \\
	4 & \elbow & \\
	\end{matrix}
	\quad\begin{matrix}
	& 1 & 2 & 3 & 4  \\
	1 & \elbows & \elbows & \elbows & \elbow & \\
	2 & \cross & \cross & \elbow & \\
	3 & \cross & \elbow & \\
	4 & \elbow & \\
	\end{matrix}
	\quad
	\endgroup
\]
The Schubert polynomial is thus equal to
\[
\Sf_{1432}(\zz)=z_1^2z_2+z_1^2z_3+z_1z_2^2+z_1z_2z_3+z_2^2z_3.
\]
\end{example}

\section{Construction of pipe dreams in the types $B$, $C$, and $D$}\label{sec:main}

\subsection{Type $C$}\label{ssec:mainc} 

The Schubert polynomials of the types $B$, $C$, and $D$ are elements of the ring $\QQ[\zz,p_1(\xx),p_3(\xx),\dots]$, i.e. they are polynomials in $\zz$ and symmetric functions in $\xx$. For our purposes, it will be sometimes more convenient to deal with polynomials both in $\zz$ and $\xx$. For this let us give the following technical definition.

\begin{definition} Let $F\in\QQ[\zz,p_1(\xx),p_3(\xx),\dots]$, and let $k\geq 0$ be a nonnegative integer. We define \emph{$k$-truncation} of $F$ as follows:
\[
F^{[k]}(\zz,x_1,\dots,x_k)=F(\zz,x_1,\dots,x_k,0,0,\dots).
\]
It is a symmetric polynomial in $x_1,\dots,x_k$.
\end{definition}

Let $w\in\BCc_n$ be a signed permutation of $n$ variables. Consider the Schubert polynomial $\Cf_w(\zz,\xx)$ of the type $C$. 
As above,  denote by $\Cf_w^{[k]}(\zz,\xx)$ its $k$-truncation.

The pipe dreams for $\Cf_w^{[k]}$ are obtained by putting elements inside boxes of the shape $B^k_{\Cc_n}$ shown on Fig.~\ref{C-base}.  It consists of a \emph{staircase block}, which is the staircase Young diagram $(n-1,n-2,\dots,2,1)$, and of $k$ consecutive blocks of a different shape, which we call \emph{$\Gamma$-blocks}. Each $\Gamma$-block consists of an $n\times 1$ column and a $1\times n$ row; all the boxes outside these blocks are filled by elbow joints (this part is sometimes referred to as the \emph{sea of elbows}). The boxes in each row and each column of $\Gamma$-blocks are indexed by the integers from $0$ to $n-1$, from top to bottom and from left to right respectively.

\begin{figure}[ht]
	$$\includegraphics[width=0.8\textwidth]{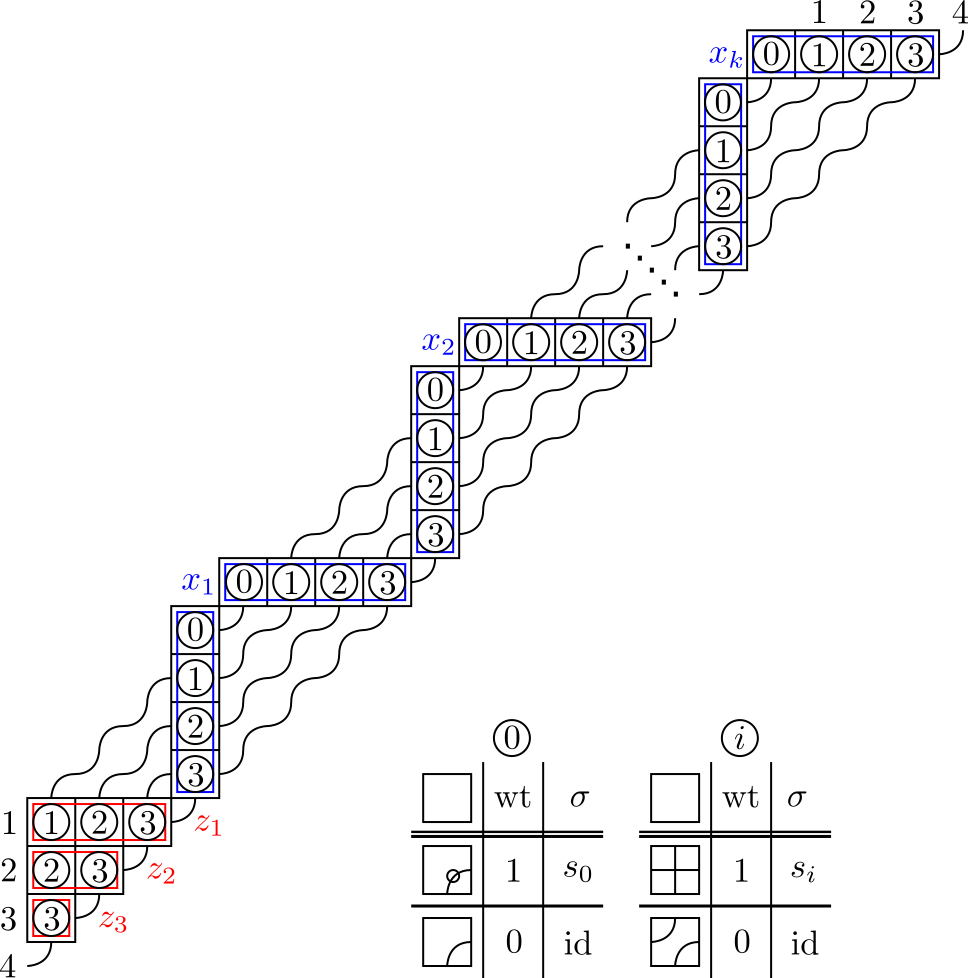}$$
	\caption{The base $B^k_{\Cc_4}$ for $c$-signed pipe dreams}
	\label{C-base}
\end{figure}

The topmost box of each column and the leftmost box of each row can be filled by the elements of two different kinds:  the single elbow joint $\elbowup$ or by a new kind of element, \emph{an elbow joint with a faucet} $\elbowsign$. All the remaining boxes are filled either with elbow joints $\elbows$ or crosses $\cross$. We will refer to crosses and to elbows with faucets as \emph{significant elements}. The diagram obtained by this construction is called a \emph{c-signed pipe dream}. It consists of $n$ strands, or pipes, connecting the left edge of the diagram with its  top edge. We index the left and the top edges of the pipes by $1,\dots,n$; as in the case of classical pipe dreams, this defines a permutation. Moreover, some of the pipes have faucets on them; each faucet is responsible for the sign change for the corresponding variable.

In more formal terms, we can assign to each box $\boxx$ the following data:
\begin{itemize}
	\item the weight $\wt(\boxx)$ of the element inside it. It is equal to $1$ for crosses and elbows with faucets $\elbowsign$ (i.e. for significant elements) and $0$ for elbows $\elbowup,\elbows$.
	\item A letter $\sigma(\boxx) \in \Cc_n$.  It is equal to $s_i$ for a cross inside a box indexed by $i>0$, to $s_0$ for an elbow with a faucet in a box indexed by zero, and to nothing (the identity permutation $\id$) for elbow joints.
	\item A variable $\var(\boxx)$. It is equal to $z_i$ for a box in the  $i$-th row of the staircase, counted from above, and to $x_i$ for a box in the  $i$-th $\Gamma$-block, counted from below.
\end{itemize}

To each $c$-signed pipe dream $D$ we can assign the  following data:
\begin{itemize}
	\item A monomial $\xx^{\alpha(D)}\zz^{\beta(D)}=\prod_{\boxx\in B^k_{\Cc_n}}\var(\boxx)^{\wt(\boxx)}$
	(i.e. the product of the variables for all significant elements in $D$).
	\item A word $\word(D)$ obtained as the result of reading the letters $\sigma(\boxx)$ while reading the boxes of the base from right to left, top to bottom.
	\item The word $\word(D)$ is subdivided into two parts, $\word_x(D)$ and $\word_z(D)$. The first of them consists of the letters  $\sigma(\boxx)$ from the $\Gamma$-blocks, while the second one corresponds to the staircase part of the pipe dream.
	\item If $\word(D)$ is a reduced word for a permutation $w=w(D)\in\Cc_n$, then the pipe dream  $D$ is said to be \emph{reduced}, and the permutation $w$ is called the \emph{shape} of $D$.
	
The last definition can be restated as follows. A $c$-signed pipe dream $D$	is reduced, if:
	\begin{itemize}
		\item each strand has no more than one faucet on it;
		\item if two strands intersect twice, then exactly one of these strands has  a faucet located between these two intersections.
	\end{itemize}
	The shape of a reduced $c$-signed pipe dream is the permutation $w\in\Cc_n$ such that each strand connects the number $i$ on the left of the diagram with $w(i)$ or $-w(i)$ above it. If the strand has a faucet on it, this corresponds to the minus sign; no faucet corresponds to plus.
	\item If $\word(D)$ is a reduced word for $w\in\Cc_n$, then $\word_x(D)$ and $\word_z(D)$ are reduced words for the permutations $u=u(D)\in\Cc_n$ and $v=v(D)\in\Sc_n$, respectively. Note that $uv=w$ and $\ell(u)+\ell(v)=\ell(w)$.
\end{itemize}

\begin{example}
	Here are three examples of $c$-signed pipe dreams of size $n=3$ with $k=2$ $\Gamma$-blocks.\vspace{0.5cm}
	
	\begin{minipage}{\textwidth}
		\begin{minipage}{0.31\textwidth}
			$$\begingroup
			\setlength\arraycolsep{-0.3pt}
			\renewcommand{\arraystretch}{0.0}
			D_1=\begin{matrix}
			&&&&&&&&&1&2&3\\[2pt]
			&&&&&&&& \belbowup&\bcross&\belbows&\elbow\\
			&&&&&&&\belbowup&\elbows&\elbows&\elbow\\
			&&&&&&\elbowup&\belbows&\elbows&\elbow\\
			&&&&&\elbowup&\elbows&\belbows&\elbow\\		
			&&&&\belbowsign&\bcross&\bcross&\elbow\\
			&&&\belbowup&\elbows&\elbows&\elbow\\
			&&\elbowup&\belbows&\elbows&\elbow\\
			&\elbowup&\elbows&\belbows&\elbow\\
			1\,&\bcross&\belbows&\elbow\\
			2\,&\bcross&\elbow\\
			3\,&\elbow
			\end{matrix}
			\endgroup$$
			$$\xx^{\alpha(D_1)}\zz^{\beta(D_1)}=z_1 z_2 x_1^3 x_2$$
			$$\word(D_1)=s_1s_2s_1s_0s_1s_2$$
			$$w(D_1)=21\overline 3\in\BCc_3$$
		\end{minipage}
		\begin{minipage}{0.31\textwidth}
			$$\begingroup
			\setlength\arraycolsep{-0.3pt}
			\renewcommand{\arraystretch}{0.0}
			D_2=\begin{matrix}
			&&&&&&&&&1&2&3\\[2pt]
			&&&&&&&& \belbowup&\bcross&\belbows&\elbow\\
			&&&&&&&\belbowup&\elbows&\elbows&\elbow\\
			&&&&&&\elbowup&\belbows&\elbows&\elbow\\
			&&&&&\elbowup&\elbows&\belbows&\elbow\\
			&&&&\belbowup&\bcross&\bcross&\elbow\\
			&&&\belbowsign&\elbows&\elbows&\elbow\\
			&&\elbowup&\belbows&\elbows&\elbow\\
			&\elbowup&\elbows&\belbows&\elbow\\
			1\,&\bcross&\bcross&\elbow\\
			2\,&\bcross&\elbow\\
			3\,&\elbow
			\end{matrix}
			\endgroup$$
			$$\xx^{\alpha(D_2)}\zz^{\beta(D_2)}=z_1^2 z_2 x_1^3 x_2$$
			$$\word(D_2)=s_1s_2s_1s_0s_2s_1s_2$$
			\begin{center}
				$D_2$ is non-reduced.
			\end{center}
		\end{minipage}
		\begin{minipage}{0.31\textwidth}
			$$\begingroup
			\setlength\arraycolsep{-0.3pt}
			\renewcommand{\arraystretch}{0.0}
			D_3=\begin{matrix}
			&&&&&&&&&1&2&3\\[2pt]
			&&&&&&&& \belbowsign&\bcross&\belbows&\elbow\\
			&&&&&&&\belbowup&\elbows&\elbows&\elbow\\
			&&&&&&\elbowup&\belbows&\elbows&\elbow\\
			&&&&&\elbowup&\elbows&\belbows&\elbow\\
			&&&&\belbowup&\belbows&\bcross&\elbow\\
			&&&\belbowsign&\elbows&\elbows&\elbow\\
			&&\elbowup&\belbows&\elbows&\elbow\\
			&\elbowup&\elbows&\belbows&\elbow\\
			1\,&\bcross&\belbows&\elbow\\
			2\,&\bcross&\elbow\\
			3\,&\elbow
			\end{matrix}
			\endgroup$$
			$$\xx^{\alpha(D_3)}\zz^{\beta(D_3)}=z_1 z_2 x_1^2 x_2^2$$
			$$\word(D_3)=s_1s_0s_2s_0s_1s_2$$
			\begin{center}
				$D_3$ is non-reduced.
			\end{center}
		\end{minipage}
	\end{minipage}
\end{example}

If  $D$ is a reduced  $c$-signed pipe dream, the monomial $\zz^{\beta(D)}$  is $z$-admissible for $\word_z(D)$ (the proof is similar to the case of $\Sc_n$, cf. Sec.~\ref{ssec:pdtypea}), while the monomial $\xx^{\alpha(D)}=x_{j_1}x_{j_2}\ldots x_{j_l}$ is  $x$-admissible for the word $\word_x(D)=s_{a_1}s_{a_2}\ldots s_{a_l}$. Indeed, $j_1\geq j_2\geq\ldots\geq j_l$ are exactly the numbers of $\Gamma$-blocks that contain significant elements. While reading each $\Gamma$-block right to left, top to bottom the numbers $a_i$ inside the boxes first decrease and then increase. This means that the equality $j_{i-1}=j_i=j_{i+1}$ (three neighboring significant elements are located in the same $\Gamma$-block) implies that either  $a_{i-1}> a_i$, or $a_i< a_{i+1}$, and hence $i\not\in P\left(\word_x(D)\right)$.

\begin{proposition}
Let $w\in\Cc_n$ be decomposed into a product $uv=w$, where $v\in\Sc_n, \ell(u)+\ell(v)=\ell(w)$. Let $\aaa=s_{a_1}s_{a_2}\ldots s_{a_l} \in R(u)$ and $\bb\in R(v)$  be reduced words, and let $\xx^\alpha=x_{j_1}x_{j_2}\ldots x_{j_l}\in\Ar_x(\aaa), \zz^\beta\in\Ar_z(\bb)$ be admissilbe monomials, such that $k\geq j_1\geq j_2\ldots\geq j_l$. Then there exist exactly $2^{i(\alpha)}$ reduced $c$-signed pipe dreams $D$, such that $\word_x(D)=\aaa, \word_z(D)=\bb, \xx^{\alpha(D)}=\xx^\alpha, \zz^{\beta(D)}=\zz^{\beta}$.
\end{proposition}
\begin{proof}
There exists exactly one placement of crosses in the staircase block (the proof is similar to the case of $\Sc_n$) and exactly two ways to place a nonzero number of significant elements in each of the $\Gamma$-blocks. Indeed, let $j_i=j_{i+1}=\ldots=j_{i+m}$. Then the sequence $a_{i}, a_{i+1},\ldots, a_{i+m}$ has no peaks, which means that it starts with a strictly decreasing segment (since a reduced word cannot contain two consecutive identical letters) and then strictly increases. Let  $a_{i+t}$ be the smallest number in this sequence.

The significant element for $a_{i+t}$ can be placed in two possible ways: in the horizontal or the vertical part of the $j_i$-th $\Gamma$-block. The significant elements for $a_{i},\ldots, a_{i+t-1}$ can be placed only in its horizontal part, and the remaining elements corresponding to $a_{i+t+1},\ldots, a_{i+m}$ are necessarily placed in the vertical part.

Since there are exactly $i(\alpha)$ blocks containing significant elements, and each of them can be filled in two possible ways, there are  $2^{i(\alpha)}$ pipe dreams for such a word.
\end{proof}

Denote by $\PD_{\Cc_n}^k(w)$ the set of all reduced $c$-signed pipe dreams with $k$ $\Gamma$-blocks of shape $w\in\Cc_n$. Then the previous discussion can be summarized as the following theorem.
\begin{theorem}\label{cdreams}
Let $w\in \BCc_n, k\geq 0$. Then the $k$-truncated Schubert polynomial  $\Cf_w^k(\zz,\xx)$ equals the sum of monomials over all reduced $c$-signed pipe dreams in $\PD_{\Cc_n}^k(w)$:
\[
\sum_{D\in\PD_{\Cc_n}^k(w)} \xx^{\alpha(D)}\zz^{\beta(D)}=\sum_{\substack{u v=w\\ \ell(u)+\ell(v)=\ell(w) \\ v \in \Sc_{n}}} \sum_{\substack{\aaa \in R(u)\\ \xx^{\alpha} \in \Ar_{x}(\mathbf{a})\\\alpha_{k+1}=\alpha_{k+2}=\ldots=0}} \sum_{\substack{\bb \in R(v)\\ \zz^{\beta} \in \Ar_{z}(\bb)}} 2^{i(\alpha)} \xx^{\alpha} \zz^{\beta}=\Cf_w^{[k]}(\zz,\xx).
\]
\end{theorem}

\subsection{Type $B$}\label{ssec:mainb}

The Schubert polynomials  $\Bf_w$ of type $B$ differ from  $\Cf_w$ only by a multiplicative coefficient: $\Bf_w=2^{-s(w)}\Cf_w$, where $s(w)$ is the number of integers from $1$ to $n$ that change their sign under the action of $w$. The number of sign changes in $w$ is equal to the number of entries of $s_0$ in each  reduced word $\aaa\in R(w)$. 

We can use the formula from Theorem~\ref{thm:cschub} to get a similar formula for the polynomials $\Bf_w$: 
	$$\Bf_{w}(\zz; \xx)=\sum_{\substack{u v=w\\ \ell(u)+\ell(v)=\ell(w) \\ v \in \Sc_{n}}} \sum_{\substack{\aaa \in R(u)\\ \xx^{\alpha} \in \Ar_{x}(\mathbf{a})}} \sum_{\substack{\bb \in R(v)\\ \zz^{\beta} \in \Ar_{z}(\bb)}} 2^{i(\alpha)-s(w)} \xx^{\alpha} \zz^{\beta}.
$$

Here we describe a construction of $b$-signed pipe dreams.  Let us modify the base for $c$-signed pipe dreams in the following way: in each $\Gamma$-block we make the vertical and the horizontal part overlap, so that the whole block would consist of $2n-1$ boxes. A base with $k$ blocks will be denoted by $B^k_{\Bc_n}$. We will fill the boxes of this base in a similar way. The diagram we obtain is called a \emph{$b$-signed pipe dream} (cf.~Fig.~\ref{B-base}).
\begin{figure}[ht]
	$$\includegraphics[width=0.8\textwidth]{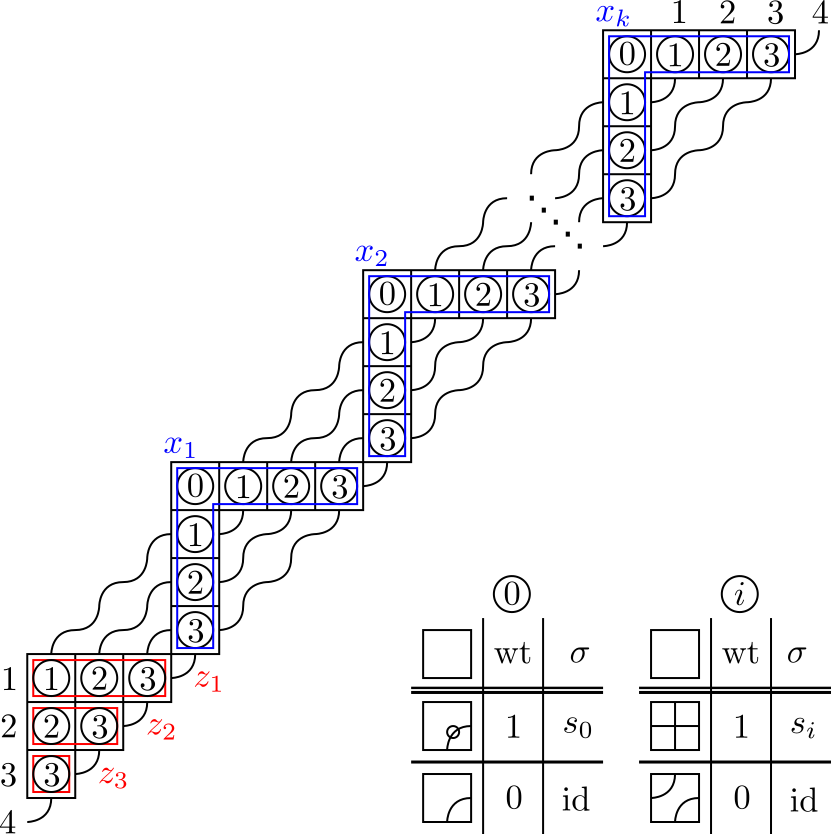}$$
	\caption{The base $B^k_{\Bc_4}$ for $b$-signed pipe dreams}
	\label{B-base}
\end{figure}

For each $b$-signed pipe dream $D$ the words $\word_x(D)$ and $\word_z(D)$, the monomial $\xx^{\alpha(D)}\zz^{\beta(D)}$, as well as the notion of reducibility, the shape $w(D)$ and the permutations $u(D), v(D)$ are defined similarly to $c$-signed pipe dreams.

In the same way we can show that for a reduced $b$-signed pipe dream $D$ the monomial  $\xx^{\alpha(D)}$ is $x$-admissible for the permutation $\word_x(D)$, and the monomial $\zz^{\beta(D)}$ is $z$-admissible for $\word_z(D)$.

\begin{example}
	Here are three examples of $b$-signed pipe dreams of size $n=3$ with $k=2$ $\Gamma$-blocks.\vspace{0.5cm}
	
	\begin{minipage}{\textwidth}
		\begin{minipage}{0.31\textwidth}
			$$\begingroup
			\setlength\arraycolsep{-0.3pt}
			\renewcommand{\arraystretch}{0.0}
			D_1=\begin{matrix}
			&&&&&&&1&2&3\\[2pt]
			&&&&&& \belbowup&\bcross&\belbows&\elbow\\
			&&&&&\elbowup&\belbows&\elbows&\elbow\\
			&&&&\elbowup&\elbows&\belbows&\elbow\\		
			&&&\belbowsign&\bcross&\bcross&\elbow\\
			&&\elbowup&\belbows&\elbows&\elbow\\
			&\elbowup&\elbows&\belbows&\elbow\\
			1\,&\bcross&\belbows&\elbow\\
			2\,&\bcross&\elbow\\
			3\,&\elbow
			\end{matrix}
			\endgroup$$
			$$\xx^{\alpha(D_1)}\zz^{\beta(D_1)}=z_1 z_2 x_1^3 x_2$$
			$$\word(D_1)=s_1s_2s_1s_0s_1s_2$$
			$$w(D_1)=21\overline 3\in\BCc_3$$
		\end{minipage}
		\begin{minipage}{0.31\textwidth}
			$$\begingroup
			\setlength\arraycolsep{-0.3pt}
			\renewcommand{\arraystretch}{0.0}
			D_2=\begin{matrix}
			&&&&&&&1&2&3\\[2pt]
			&&&&&& \belbowup&\bcross&\belbows&\elbow\\
			&&&&&\elbowup&\belbows&\elbows&\elbow\\
			&&&&\elbowup&\elbows&\belbows&\elbow\\
			&&&\belbowsign&\bcross&\bcross&\elbow\\
			&&\elbowup&\belbows&\elbows&\elbow\\
			&\elbowup&\elbows&\belbows&\elbow\\
			1\,&\bcross&\bcross&\elbow\\
			2\,&\bcross&\elbow\\
			3\,&\elbow
			\end{matrix}
			\endgroup$$
			$$\xx^{\alpha(D_2)}\zz^{\beta(D_2)}=z_1^2 z_2 x_1^3 x_2$$
			$$\word(D_2)=s_1s_2s_1s_0s_2s_1s_2$$
			\begin{center}
				$D_2$ is non-reduced.
			\end{center}
		\end{minipage}
		\begin{minipage}{0.31\textwidth}
			$$\begingroup
			\setlength\arraycolsep{-0.3pt}
			\renewcommand{\arraystretch}{0.0}
			D_3=\begin{matrix}
			&&&&&&&1&2&3\\[2pt]
			&&&&&& \belbowsign&\bcross&\belbows&\elbow\\
			&&&&&\elbowup&\belbows&\elbows&\elbow\\
			&&&&\elbowup&\elbows&\belbows&\elbow\\
			&&&\belbowsign&\belbows&\bcross&\elbow\\
			&&\elbowup&\belbows&\elbows&\elbow\\
			&\elbowup&\elbows&\belbows&\elbow\\
			1\,&\bcross&\belbows&\elbow\\
			2\,&\bcross&\elbow\\
			3\,&\elbow
			\end{matrix}
			\endgroup$$
			$$\xx^{\alpha(D_3)}\zz^{\beta(D_3)}=z_1 z_2 x_1^2 x_2^2$$
			$$\word(D_3)=s_1s_0s_2s_0s_1s_2$$
			\begin{center}
				$D_3$ is non-reduced.
			\end{center}
		\end{minipage}
	\end{minipage}
	
\end{example}

\begin{proposition}
Let  $w\in \BCc_n$ be decomposed into a product $uv=w$, where $v\in\Sc_n, \ell(u)+\ell(v)=\ell(w)$. Let $\aaa=s_{a_1}s_{a_2}\ldots s_{a_l} \in R(u), \bb\in R(v)$ be reduced words, and let  $\xx^\alpha=x_{j_1}x_{j_2}\ldots x_{j_l}\in\Ar_x(\aaa), \zz^\beta\in\Ar_z(\bb)$ be admissible monomials satisfying $k\geq j_1\geq j_2\ldots\geq j_l$.
	
Then there exist exactly  $2^{i(\alpha)-s(w)}$ reduced $b$-signed pipe dreams $D$ such that $\word_x(D)=\aaa, \word_z(D)=\bb, \xx^{\alpha(D)}=\xx^\alpha, \zz^{\beta(D)}=\zz^{\beta}$.
\end{proposition}
\begin{proof}
We need to find the number of ways to put significant elements in all $\Gamma$-blocks. Let $j_i=j_{i+1}=\ldots=j_{i+m}$. The sequence $a_{i}, a_{i+1},\ldots, a_{i+m}$ contains no peaks, so it first decreases and then increases. Let $a_{i+t}$ be the minimal element of this sequence. If $a_{i+t}=0$, we need to place an elbow with a faucet in the upper-left box of the corresponding block; otherwise  $a_{i+t}\ne0$ and the cross corresponding to this letter can be placed either in the horizontal, or in the vertical part of the $\Gamma$-block. The elements corresponding to $a_{i},\ldots, a_{i+t-1}$ can be placed only in the horizontal part of the block, and those corresponding to $a_{i+t+1},\ldots, a_{i+m}$ can be only in the vertical part. Hence for each nonempty block with a faucet there is a unique way to place the elements, and for blocks without faucets there are two ways. The number of nonempty blocks without faucets is equal to $i(\alpha)-s(w)$. The proposition is proved.
\end{proof}

Denote by $\PD_{\Bc_n}^k(w)$ the set of all reduced  $b$-signed pipe dreams with $k$ $\Gamma$-blocks of shape $w\in \Bc_n$. Then the following theorem holds.

\begin{theorem}\label{bdreams}	
Let $w\in \BCc_n$ and $k\geq 0$. The $k$-truncation $\Bf_w^{[k]}(\zz,\xx)$ of the Schubert polynomial  $\Bf_w(\zz,\xx)$  can be obtained as the sum of monomials over all reduced b-signed pipe dreams in $\PD_{\Bc_n}^(w)$:
\[
\sum_{D\in\PD_{\Bc_n}^k(w)} \xx^{\alpha(D)}\zz^{\beta(D)}=\sum_{\substack{u v=w\\ \ell(u)+\ell(v)=\ell(w) \\ v \in \Sc_{n}}} \sum_{\substack{\aaa \in R(u)\\ \xx^{\alpha} \in \Ar_{x}(\mathbf{a})\\\alpha_{k+1}=\alpha_{k+2}=\ldots=0}} \sum_{\substack{\bb \in R(v)\\ \zz^{\beta} \in \Ar_{z}(\bb)}} 2^{i(\alpha)-s(w)} \xx^{\alpha} \zz^{\beta}=\Bf_w^{[k]}(\zz,\xx).
\]
\end{theorem}

\subsection{Type $D$}\label{ssec:maind}


The base $B_{\Dc_n}^k$ consists of a staircase block with $n-1$ stairs and  $k$ $\Gamma$-blocks of height and width $n-1$. We join the blocks by elbows $\elbows$. In each box of the staircase we write the number of its diagonal. The upper left corner of each $\Gamma$-block is indexed by the symbol $1'$, while the remaining boxes are indexed by the integers $2,\dots,n$ from top to bottom and from left to right (cf.~Fig.~\ref{D-base}).

\begin{figure}[ht]
	$$\includegraphics[width=0.8\textwidth]{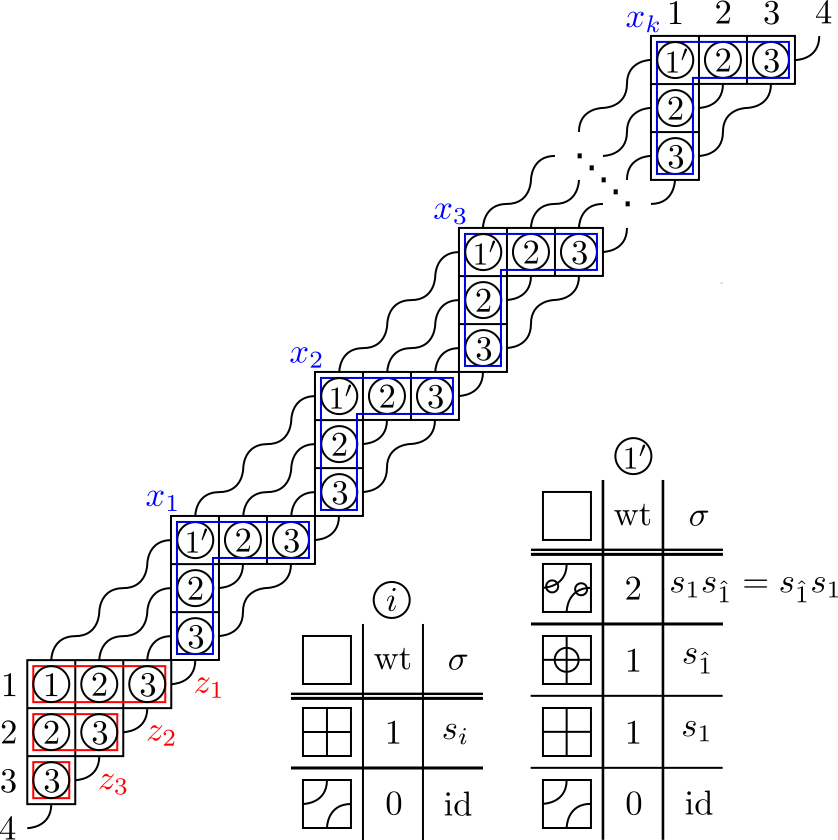}$$
	\caption{The base  $B^k_{\Dc_4}$ for $d$-signed pipe dreams}
	\label{D-base}
\end{figure}

Let us fill the boxes indexed by $1'$ by elements of four possible types: elbows $\elbows$, crosses $\cross$, crosses with a faucet $\crosssign$ and elbows with two faucets $\elbowssign$. All the remaining boxes will be filled by crosses $\cross$ and elbows $\elbows$. The object we obtain will be called a \emph{$d$-signed pipe dream}. As in the previous cases, it consists of strands that connect the left edge with the top one. Let us index these edges by the integers from $1$ to $n$. To each box $\boxx$ of the base we assign the following data:
\begin{itemize}
	\item the weight $\wt(\boxx)$ of the element inside it. It is equal to 2 for elbows with two faucets $\elbowssign$, 1 for crosses $\cross$ and crosses with a faucet $\crosssign$, and 0 for elbows $\elbows$.
	\item zero, one or two letters $\sigma(\boxx) \in \BCc_n$. This set of letters is equal to $s_i$ for a cross $\cross$ in a box indexed by $i$, to $s_1$ for a cross $\cross$ in a box indexed by $1'$, to $s_{\hat{1}}$ for a cross with a faucet $\crosssign$, to a pair of letters $s_1s_{\hat 1}=s_{\hat 1}s_1$  for an elbow with two crosses $\elbowssign$ and nothing (the identity permutation $\id$) for elbows $\elbows$.
	\item a variable $\var(\boxx)$. It is equal to $z_i$ for a box in the  $i$-th row of the staircase (counted from above) and to $x_i$ for a box in the  $i$-th $\Gamma$-block (counted from below).
\end{itemize}

To each $d$-signed pipe dream $D$ we assign:
\begin{itemize}
	\item a monomial $\xx^{\alpha(D)}\zz^{\beta(D)}=\prod_{\boxx\in B^k_{\Dc_n}}\var(\boxx)^{\wt(\boxx)}$;
	\item the number $r(D)$ of elbow joints with two faucets occuring in $D$;
	\item a set of words $\word^1(D), \word^2(D),\ldots,\word^{2^{r(D)}}(D)$. Each of these words is obtained by consecutive reading of all the letters  $\sigma(\boxx)$, where $\boxx$ runs over the set of all boxes of the base from right to left, top to bottom. For each elbow with two faucets there are two possibilities to read the pair of commuting letters  $s_1, s_{\hat 1}$, so the total number of words obtained is equal to $2^{r(D)}$.
	\item Each of these words $\word^p(D)$ can be split into two subwords $\word^p_x(D)$ and $\word_z(D)$ (the latter word is the same for all $p$). The first word is obtained by reading the letters of $\Gamma$-blocks, and the second one corresponds to the staircase block.
	\item If $\word^p(D)$ (and hence all the other words) is a reduced word for the permutation $w=w(D)\in\Dc_n$, the pipe dream  $D$ is said to be \emph{reduced}, and the permutation $w$ is called the \emph{shape} of this pipe dream. 
	\item If $\word^p(D)$ is a reduced word for $w\in\Dc_n$, then $\word^p_x(D)$ and $\word_z(D)$ are reduced words for the permutations $u=u(D)\in\Dc_n$ and $v=v(D)\in\Sc_n$ respectively. Note that $uv=w$ and $\ell(u)+\ell(v)=\ell(w)$.
\end{itemize}

\begin{remark} Let us restate the definition of a reduced $d$-signed pipe dream in terms of the intersection of the strands; it is somewhat more involved than in the previous cases. Consider two strands that intersect twice, in the boxes  $a$ and $b$. A $d$-signed pipe dream is \emph{nonreduced} if it contains any of the eight patterns shown on Fig.~\ref{D-forbidden} (the boxes  $a$ and $b$ are highlighted by grey frames):
\begin{itemize}
	\item the boxes $a$ and $b$ contain crosses $\cross$, and the segments of both strands between the intersection do not have faucets on them;
	\item the boxes $a$ and $b$ contain crosses $\cross$, and the segments of both strands between the intersection have faucets on them;
	\item the boxes $a$ and $b$ contain crosses with faucets $\crosssign$, and the segments of both strands between the intersection do not have faucets on them;
	\item the boxes $a$ and $b$ contain crosses with faucets $\crosssign$, and the segments of both strands between the intersection have faucets on them;
	\item one of the boxes $a,b$ contains a cross $\cross$, and the other contains a cross with a faucet $\crosssign$. Exactly one of the two strands has a faucet between these two boxes;
	\item one of the boxes $a,b$ contains a cross $\cross$, and the other contains an elbow with two faucets $\elbowssign$;
	\item  one of the boxes $a,b$ contains a cross with a faucet $\crosssign$, and the other contains an elbow with two faucets $\elbowssign$;
	\item each of the boxes $a$ and $b$ contains an elbow with two faucets.
\end{itemize}

In the last three cases the existence of an extra faucet on the segments situated between $a$ and $b$ is not important.
	\begin{figure}[ht]
		\includegraphics{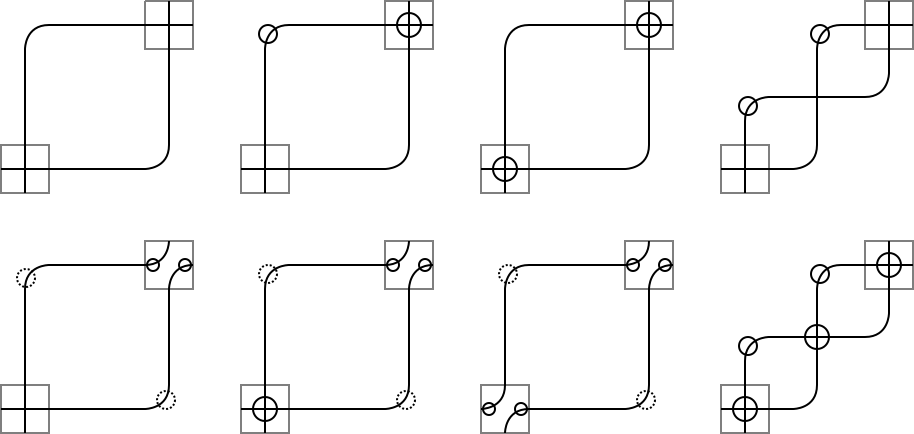}
		\caption{Forbidden patterns not occuring in the reduced $d$-signed pipe dreams (dashed faucets may either be or not be there)}
		\label{D-forbidden}
	\end{figure}
If a $d$-signed pipe dream does not contain any of these eight ``forbidden'' patters, such a pipe dream is \emph{reduced}.

The shape of a reduced $d$-signed pipe dream $D$ can be read in a similar way as in the cases $B$ and $C$: it is the permutation $w=w(D)\in\Dc_n$ such that for each $i=1,\ldots,n$ the number   $i$ on the left-hand side is joined with $(-1)^{c_i(D)}w(i)$ on the top side, where  $c_i(D)$ is the number of crosses on the $i$-th strand.
\end{remark}

\begin{example}
	Here are two examples of $d$-signed pipe dreams of size $n=4$ with $k=2$ $\Gamma$-blocks.\vspace{0.5cm}
	
	\begin{minipage}{\textwidth}
		\begin{minipage}{0.49\textwidth}
			$$\begingroup
			\setlength\arraycolsep{-0.3pt}
			\renewcommand{\arraystretch}{0.0}
			D_1=\begin{matrix}
			&&&&&&&1&2&3&4\\[2pt]
			&&&&&&\elbowup&\belbowssign&\bcross&\belbows&\elbow\\
			&&&&&\elbowup&\elbows&\belbows&\elbows&\elbow\\
			&&&&\elbowup&\elbows&\elbows&\belbows&\elbow\\
			&&&\elbowup&\bcrosssign&\bcross&\bcross&\elbow\\
			&&\elbowup&\elbows&\belbows&\elbows&\elbow\\
			&\elbowup&\elbows&\elbows&\belbows&\elbow\\
			1\,&\bcross&\belbows&\belbows&\elbow\\
			2\,&\belbows&\belbows&\elbow\\
			3\,&\bcross&\elbow\\
			4\,&\elbow
			\end{matrix}
			\endgroup$$
			$$\xx^{\alpha(D_1)}\zz^{\beta(D_1)}=z_1z_3x_1^3x_2^3$$
			$$\word_1(D_1)=s_2s_1s_{\hat 1}s_3s_2s_{\hat 1} s_1 s_3$$
			$$\word_2(D_1)=s_2s_{\hat 1}s_1s_3s_2s_{\hat 1} s_1 s_3$$
			$$w(D_1)=1\overline 4 2 \overline 3\in\Dc_4$$
		\end{minipage}
		\begin{minipage}{0.49\textwidth}
			$$\begingroup
			\setlength\arraycolsep{-0.3pt}
			\renewcommand{\arraystretch}{0.0}
			D_2=\begin{matrix}
			&&&&&&&1&2&3&4\\[2pt]
			&&&&&&\elbowup&\bcrosssign&\bcross&\belbows&\elbow\\
			&&&&&\elbowup&\elbows&\bcross&\elbows&\elbow\\
			&&&&\elbowup&\elbows&\elbows&\belbows&\elbow\\
			&&&\elbowup&\belbowssign&\belbows&\belbows&\elbow\\
			&&\elbowup&\elbows&\belbows&\elbows&\elbow\\
			&\elbowup&\elbows&\elbows&\belbows&\elbow\\
			1\,&\bcross&\bcross&\belbows&\elbow\\
			2\,&\belbows&\belbows&\elbow\\
			3\,&\bcross&\elbow\\
			4\,&\elbow
			\end{matrix}
			\endgroup$$
			$$\xx^{\alpha(D_2)}\zz^{\beta(D_2)}=z_1^2z_3x_1^2x_2^3$$
			$$\word_1(D_2)=s_2s_{\hat 1}s_2 s_1 s_{\hat 1} s_2 s_1 s_3$$
			$$\word_2(D_2)=s_2s_{\hat 1}s_2 s_{\hat 1} s_1 s_2 s_1 s_3$$
			\begin{center}
				$D_2$ is non-reduced
			\end{center}
		\end{minipage}
	\end{minipage}	
\end{example}

Similarly to the previous cases we can show that if $D$ is a reduced $d$-signed pipe dream, the monomial  $\zz^{\beta(D)}$ is $z$-admissible for  $\word_z(D)$,  and the monomial $\xx^{\alpha(D)}$ is  $x$-admissible for each of the words $\word^1_x(D),\ldots,\word^{2^{r(D)}}_x(D)$.

\begin{proposition}
Let $w\in\Dc_n$ be decomposed into a product $uv=w$, where $v\in\Sc_n$ and $\ell(u)+\ell(v)=\ell(w)$. Let $\aaa=s_{a_1}s_{a_2}\ldots s_{a_l} \in R(u)$ and $\bb\in R(v)$ be reduced words, and let $\xx^\alpha=x_{j_1}x_{j_2}\ldots x_{j_l}\in\Ar_x(\aaa)$ and $\zz^\beta\in\Ar_z(\bb)$ be admissible monomials, and $k\geq j_1\geq j_2\ldots\geq j_l$.
	
Let $r(\aaa, \alpha)$ denote the number of  $i$ such that $j_i=j_{i+1}, a_i=1, a_{i+1}=\hat 1$ or, vice versa, $a_i=\hat 1, a_{i+1}=1$.
	
Then there exist exactly $2^{i(\alpha)-o(\aaa)+r(\aaa, \alpha)}$ reduced $d$-signed pipe dreams $D$ such that  $\word^p_x(D)=\aaa$ for some $p$, $\word_z(D)=\bb$, $\xx^{\alpha(D)}=\xx^\alpha$, and $\zz^{\beta(D)}=\zz^{\beta}$. For each of these pipe dreams $D$ we have $r(D)=r(\aaa, \alpha)$.
\end{proposition}

\begin{proof} Similarly to the previous cases we can show that there exists a unique way to place the crosses in the staircase block. We need to find the number of ways of placing the elements in each of the $\Gamma$-blocks. Let $j_i=j_{i+1}=\ldots=j_{i+m}$. We distinguish between the following cases:
	\begin{itemize}
		\item the sequence $a_i,\ldots, a_{i+m}$ does not contain elements equal to 1 or $\hat 1$. There are two possibilities for putting the cross corresponding to the smallest element $a_{i+t}$ of the sequence: it can be located in the horizontal or in the vertical part of the block. The remaining elements are placed in a unique way.
		\item The sequence  $a_i,\ldots, a_{i+m}$ contains  1, but does not contain $\hat 1$. Then the top left corner of the $j_i$-th block contains a cross $\cross$; all the remaining elements are placed uniquely.
		\item The sequence $a_i,\ldots, a_{i+m}$ contains  $\hat 1$ and does not contain 1. Then we put a cross with a faucet $\crosssign$ into the corner of the $j_i$-th block; the remaining elements are placed uniquely.
		\item The sequence $a_i,\ldots, a_{i+m}$ contains both $\hat 1$ and 1. Then we place an elbow with two faucets $\elbowssign$ into the corner of the $j_i$-th block; the positions of the remaining elements are determined uniquely.
	\end{itemize}
This means that for  $i(\alpha)-o(\aaa)+r(\aaa, \alpha)$ blocks there are two possibilities to place the elements, and all the remaining blocks are filled uniquely. This concludes the proof.
\end{proof}

Denote by $\PD_{\Dc_n}^k(w)$ the set of all reduced $d$-signed pipe dreams with $k$ $\Gamma$-blocks of shape $w\in \Dc_n$. Since each pipe dream  $D$ corresponds to $2^{r(D)}$ pairs $(\aaa,\xx^\alpha)$, where $\aaa$ is a reduced word, and  $\xx^\alpha$ is an admissible monomial,  each such pair corresponds to $2^{i(\alpha)-o(\aaa)+r(\aaa, \alpha)}$ pipe dreams, and $r(D)=r(\aaa,\alpha)$, this gives us the following theorem.

\begin{theorem}\label{ddreams}
Let $w\in \Dc_n$ and $k\geq 0$. Then the $k$-truncation $\Df_w^{[k]}(\zz,\xx)$  of the  Schubert polynomial $\Df_w(\zz,\xx)$ is equal to the sum of monomials over all reduced $d$-signed pipe dreams in $\PD_{\Dc_n}^k(w)$:
\[
\sum_{D\in\PD_{\Dc_n}^k(w)}\xx^{\alpha(D)}\zz^{\beta(D)}=\sum_{\substack{u v=w\\ \ell(u)+\ell(v)=\ell(w) \\ v \in \Sc_{n}}} \sum_{\substack{\aaa \in R(u)\\ \xx^{\alpha} \in \Ar_{x}(\mathbf{a})\\\alpha_{k+1}=\alpha_{k+2}=\ldots=0}} \sum_{\substack{\bb \in R(v)\\ \zz^{\beta} \in \Ar_{z}(\bb)}} 2^{i(\alpha)-o(\aaa)} \xx^{\alpha} \zz^{\beta}=\Df_w^{[k]}(\zz,\xx).
\]
\end{theorem}

\subsection{Infinite pipe dreams}\label{ssec:infinite}

Let  $\Fc_n$ denote either $\BCc_n$ or $\Dc_n$. Denote by $\Ff_w(\xx,\zz)$ the Schubert polynomial of type $B$, $C$ or $D$ respectively for $w\in\Fc_n$.

We can consider bases of pipe dreams with countably many $\Gamma$-blocks, going infinitely in the northeastern direction:
$$
B_{\Fc_n}:=\bigcup_{k\in \NN}B^k_{\Fc_n}.
$$
These bases can be filled with elements according to the same rules as in the finite case. Reducibility is defined the same way as above. Note that if a  signed pipe dream $D$ of arbitrary type is reduced, it contains finitely many significant elements. This means that for a reduced infinite pipe dreams  $D$ the monomial $\xx^{\alpha(D)}\zz^{\beta(D)}$ and the shape  $w=w(D)\in\Fc_n$ are well defined.

Denote by $\PD_{\Fc_n}(w)$ the set of reduced infinite signed pipe dreams of shape $w\in\Fc_n$:
$$
\PD_{\Fc_n}(w)=\bigcup_{k\in \NN}\PD^k_{\Fc_n}(w)
$$
(finite pipe dreams are assumed to have an infinite ``tail'' of elbows $\elbows$). Since Schubert polynomials equal the projective limits of their truncations
$$
\Ff_w(\xx,\zz)=\varprojlim_{k\to\infty}\Ff^{[k]}_w(\xx,\zz),
$$
they are equal to the sums of monomials over infinite signed pipe dreams:
\begin{corollary}\label{infinite}
	For each $w\in \Fc_n$, we have
\[
\Ff_w(\xx,\zz)=\sum_{D\in\PD_{\Fc_n}(w)}\xx^{\alpha(D)}\zz^{\beta(D)}.
\]
\end{corollary}

\subsection{Double Schubert polynomials}\label{ssec:double}

Double Schubert polynomials for the classical groups were introduced in~\cite{IkedaMihalceaNaruse11}. They are elements of the ring $\QQ[\zz,p_1(\xx),p_3(\xx),\ldots,\ttt]$, where  $\ttt=t_1,t_2,\ldots$ is another countable set of variables, responsible for the action of a maximal torus of the corresponding classical group.

This ring is equipped with an action of the group $\BCc_\infty\times\BCc_\infty$ and  hence of its index 4 subgroup $\Dc_\infty\times \Dc_\infty$. The first copy of $\BCc_\infty$ acts as before, permuting the variables $z_i$:
\begin{eqnarray*}
s^z_if(z_1,\ldots,z_i,z_{i+1},\ldots;x_1,x_2,\ldots;t_1,t_2,\ldots)&=&f(z_1,\ldots,z_{i+1},z_i,\ldots;x_1,x_2,\ldots;t_1,t_2,\ldots) \;\text{for $i\geq 1$};\\
s^z_0f(z_1,z_2,z_3,\ldots;x_1,x_2,\ldots;t_1,t_2,\ldots)&=&f(-z_1,z_2,z_3\ldots;z_1,x_1,x_2,\ldots;t_1,t_2,\ldots);\\
s^z_{\hat 1}f(z_1,z_2,z_3\ldots;x_1,x_2,\ldots;t_1,t_2,\ldots)&=&f(-z_2,-z_1,z_3\ldots;z_1,z_2,x_1,x_2,\ldots;t_1,t_2,\ldots).
\end{eqnarray*}	
(recall that $s_{\hat 1}=s_0s_1s_0$).

The second copy of $\BCc_\infty$ acts by permuting the variables $t_i$:
\begin{eqnarray*}
s^t_if(z_1, z_2\ldots;x_1,x_2,\ldots;t_1, \ldots, t_i, t_{i+1},\ldots)&=&f(z_1, z_2\ldots;x_1,x_2,\ldots;t_1, \ldots, t_{i+1}, t_i,\ldots)\;\text{for $i\geq 1$};\\
s^t_0f(z_1,z_2,\ldots;x_1,x_2,\ldots;t_1,t_2,t_3,\ldots)&=&f(z_1,z_2,\ldots;-t_1,x_1,x_2,\ldots;-t_1,t_2,t_3,\ldots);\\
s^t_{\hat 1}f(z_1,z_2,\ldots;x_1,x_2,\ldots;t_1,t_2,t_3\ldots)&=&f(z_1,z_2,z_3\ldots;-t_1,-t_2,x_1,x_2,\ldots;-t_2,-t_1,t_3,\ldots).
\end{eqnarray*}

It is easy to check explicitly that this defines an action, and that the actions of the two copies of $\BCc_\infty$ commute. 

This allows us to define two families of divided difference operators. Following~\cite{IkedaMihalceaNaruse11}, we  call them $\partial_i$ and $\delta_i$.

\begin{definition}[{\cite[Sec.~2.5]{IkedaMihalceaNaruse11}}]
Let us define divided difference operators $\partial$ and $\delta$ acting on the ring  $\QQ[\zz,p_1(\xx),p_3(\xx),\ldots,\ttt]$ as follows:
\[
\partial_{i} f=\frac{f-s^z_i f}{z_{i}-z_{i+1}},\;\;\; \delta_i f=\frac{f-s^t_i f}{t_{i+1}-t_i}
\]
for $i\geq 1$;
\[
\partial_{0} f=\frac{f-s^z_0 f}{-2 z_{1}},\;\;\; \delta_0f=\frac{f-s^t_0 f}{2t_1};
\]
\[
\partial_{0}^{B} f=2\partial_0 f=\frac{f-s^z_0 f}{-z_{1}},\;\;\; \delta^B_0f=2\delta_0 f =\frac{f-s^t_0 f}{t_1};
\]
\[
\partial_{\hat 1} f=\frac{f-s^z_{\hat 1} f}{-z_{1}-z_{2}},\;\;\; \delta_{\hat 1}f=\frac{f-s^t_{\hat 1}f}{t_1+t_2}.
\]
\end{definition}
\begin{definition}
\emph{Double Schubert polynomials}  $\Bf_w(\zz,\xx,\ttt),\Cf_w(\zz,\xx,\ttt)$ and $\Df_w(\zz,\xx,\ttt)$ (as before, we denote by $\Ff_w(\zz,\xx,\ttt)$ any of these three polynomials) are elements of the ring $\QQ[\zz,p_1(\xx),p_3(\xx),\ldots,\ttt]$ that are indexed by the permutations $w\in\BCc_\infty$ or $w\in\Dc_\infty$ and satisfy the equations
\[
\Ff_{id}=1,
\]
\[
	\partial_{i} \Ff_{w}=\begin{cases}
	\Ff_{w s_{i}}, & \text { if } \ell\left(w s_{i}\right)<\ell(w), \\
	0, & \text { if } \ell\left(w s_{i}\right)>\ell(w),
	\end{cases}\;\;\;
	\delta_{i} \Ff_{w}=\begin{cases}
	\Ff_{s_{i}w}, & \text { if } \ell\left(s_{i}w\right)<\ell(w), \\
	0, & \text { if } \ell\left(s_{i}w\right)>\ell(w),
	\end{cases}
\]
	for each $i=1,2,\ldots$,
\[
	\partial_0 \Cf_{w}=\begin{cases}
	\Cf_{w s_{0}}, & \text { if } \ell\left(w s_{0}\right)<\ell(w), \\
	0, & \text { if } \ell\left(w s_{0}\right)>\ell(w),
	\end{cases}\;\;\;
	\delta_{0} \Cf_{w}=\begin{cases}
	\Cf_{s_{0}w}, & \text { if } \ell\left(s_{0}w\right)<\ell(w), \\
	0, & \text { if } \ell\left(s_{0}w\right)>\ell(w),
	\end{cases}
\]
\[
	\partial^B_0 \Bf_{w}=\begin{cases}
	\Bf_{w s_{0}}, & \text { if } \ell\left(w s_{0}\right)<\ell(w), \\
	0, & \text { if } \ell\left(w s_{0}\right)>\ell(w),
	\end{cases}\;\;\;
	\delta^B_{0} \Bf_{w}=\begin{cases}
	\Bf_{s_{0}w}, & \text { if } \ell\left(s_{0}w\right)<\ell(w), \\
	0, & \text { if } \ell\left(s_{0}w\right)>\ell(w),
	\end{cases}
\]
\[
	\partial_{\hat 1} \Df_{w}=\begin{cases}
	\Df_{w s_{\hat 1}}, & \text { if } \ell\left(w s_{\hat 1}\right)<\ell(w), \\
	0, & \text { if } \ell\left(w s_{\hat 1}\right)>\ell(w),
	\end{cases}\;\;\;
	\delta_{\hat 1} \Df_{w}=\begin{cases}
	\Df_{s_{\hat 1}w}, & \text { if } \ell\left(s_{\hat 1}w\right)<\ell(w), \\
	0, & \text { if } \ell\left(s_{\hat 1}w\right)>\ell(w).
	\end{cases}
\]
\end{definition}

\subsection{Double Schubert polynomials as sums over pipe dreams}\label{ssec:doublepipes}

The goal of this subsection is to express double Schubert polynomials of types $B$, $C$, and $D$ as sums over the set of pipe dreams. We start with recalling the following theorem, which expresses double Schubert polynomials via the ordinary ones.

\begin{theorem}[Cauchy expansion formula, {\cite[Cor.~8.10]{IkedaMihalceaNaruse11}}]\label{doubleschub}
Double Schubert polynomials  $\Ff_w(\zz,\xx,\ttt)$ satisfy the equation
\[
	\Ff_{w}(\zz,\xx,\ttt)=\sum_{\substack{vu=w\\\ell(v)+\ell(u)=\ell(w)\\v \in \Sc_n}}\Sf_{v^{-1}}(-\ttt)\Ff_u(\zz,\xx).
\]
\end{theorem}

The following easy proposition describes the pipe dreams for the inverse permutation.

\begin{proposition}\label{invschub}
Let  $v\in \Sc_n$, and let $\boxx$ be a box in the base $B_{\Sc_n}$. Let $\var'(\boxx)=-t_j$, where  $j$ is the number of the column containing this box. To a pipe dream $D$ we can assign a monomial
\[
(-\ttt)^{\beta'(D)}=\prod_{\boxx\in B_{\Sc_n}}\var'(\boxx)^{\wt(\boxx)}
\]
(the product of all the variables $-t_j$ over all the crosses in $D$). Then
\[
\Sf_{v^{-1}}(-\ttt)=\sum_{D\in\PD_{A}(v)}(-\ttt)^{\beta'(D)}.
\]
\end{proposition}
	\begin{figure}[ht]
		$$\includegraphics{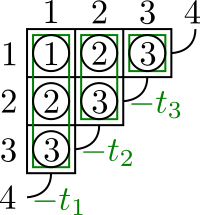}$$
		\caption{The base for pipe dreams of size $n=4$ indicating the variables $\var'$}
		\label{A-inv-base}
	\end{figure}

\begin{proof}
Denote by $D^T$ the \emph{transposed} pipe dream: the pipe dream obtained by reflecting $D$ with respect to the NW-SE diagonal. It is easy to see that $D\in\PD_{A}(v)$ iff $D^T\in\PD_{A}(v^{-1})$. The proposition follows from this assertion and Theorem~\ref{sdreams}. 
\end{proof}

As before, let us denote the truncations of Schubert polynomials as follows:
\[
\Ff_w^{[k]}(\zz,\xx,\ttt)=\Ff_w(\zz,x_1,x_2,\ldots,x_k,0,0,\ldots,\ttt).
\]
Theorems~\ref{cdreams}, \ref{bdreams}, \ref{ddreams} and Proposition~\ref{invschub} imply that the formula from Theorem~\ref{doubleschub} can be rewritten as follows:
\begin{corollary}\label{doubleschub1}
Let $w\in\Fc_n$ and $k\geq 0$. Then we have\[
\Ff^{[k]}_w(\zz,\xx,\ttt)=\sum_{\substack{vu=w\\\ell(v)+\ell(u)=\ell(w)\\v \in \Sc_n}}\sum_{D' \in \PD^k_{\Fc_n}(u)}\sum_{D'' \in \PD_{A}(v)}\xx^{\alpha(D')}\zz^{\beta(D')}(-\ttt)^{\beta'(D'')}.
\]
\end{corollary}

The base for double Schubert polynomials  $DB^k_{\Fc_n}$ can be obtained by adding another staircase of size $(n-1,\dots,2,1)$ (shown on Fig.~\ref{A-inv-base}) in the northeastern part of the base  $B^k_{\Fc_n}$ and joining them by elbows $\elbows$ (cf.~Fig.~\ref{double-base}). The boxes of the ``double'' base $DB^k_{\Fc_n}$ are filled according to the same rules as in the case of $B^k_{\Fc_n}$; the object obtained is called a \emph{$b$-, $c$- or $d$-double signed pipe dream}.
\begin{figure}[ht]
	$$\includegraphics[width=0.55\textwidth]{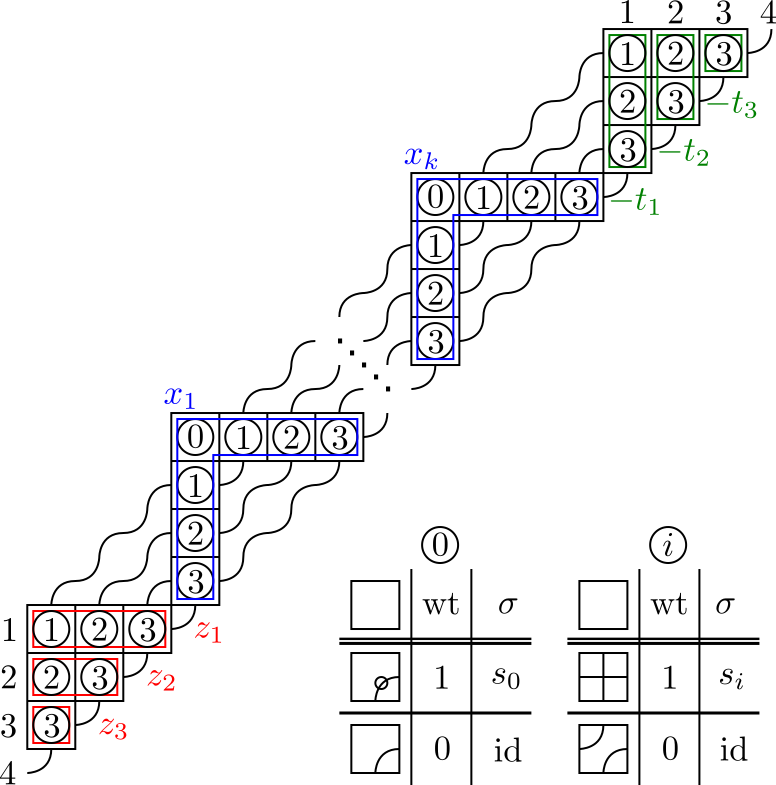}\;\;\;\includegraphics[width=0.55\textwidth]{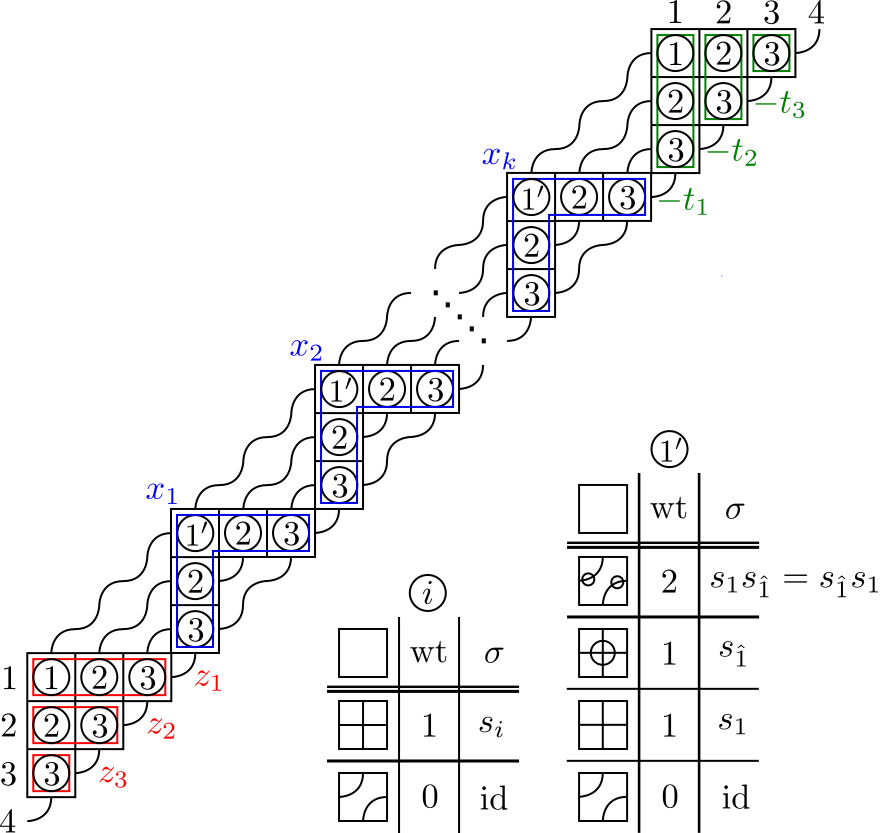}$$
	\caption{Bases $DB^k_{\Bc_4}$ and $DB^k_{\Dc_4}$ for $b$- and $d$-double signed pipe dreams}
	\label{double-base}
\end{figure}

For a double  signed pipe dream $D$ of any of the three types we can define similarly to the ordinary case the following data:
\begin{itemize}
	\item a monomial $\xx^{\alpha(D)}\zz^{\beta(D)}(-\ttt)^{\beta'(D)}$;
	\item the reducibility property;
	\item for $D$ reduced, the shape $w=w(D)\in\Fc_n$.
\end{itemize}
We can subdivide the base $DB^k_{\Fc_n}$ into the base for usual signed pipe dreams $B^k_{\Fc_n}$ and the staircase $B'_{\Sc_n}$ in the northeastern part. So, if  $D$ is reduced and its shape is equal to $w\in\Fc_n$, it is subdivided into  $D'\in\PD^k_{\Fc_n}(u)$ and  $D''\in\PD_{A}(v)$. Obviously, $\ell(u)+\ell(v)=\ell(w)$ and $vu=w$. 

The converse is also true: if $\ell(u)+\ell(v)=\ell(w)$ and $vu=w,v\in \Sc_n$, then the pipe dreams $D'\in\PD^k_{\Fc_n}(u)$ and  $D''\in\PD_{A}(v)$ add up to a reduced double signed pipe dream of shape $w$.

Denote by $\DPD^k_{\Fc_n}(w)$ the set of all reduced double signed pipe dreams of type $B$, $C$ or $D$ of shape $w\in\Fc_n$. Then the formula from Corollary~\ref{doubleschub1} can be rewritten as follows:
\begin{corollary}\label{cor:double}
If $w\in\Fc_n$ and $k\geq 0$, the $k$-truncation of a double Schubert polynomial is obtained as the sum of monomials over all reduced double signed pipe dreams of a given shape:
\[
\Ff^{[k]}_w(\zz,\xx,\ttt)=\sum_{D \in \DPD^k_{\Fc_n}(w)}\xx^{\alpha(D)}\zz^{\beta(D)}(-\ttt)^{\beta'(D)}.
\]
\end{corollary}

\begin{remark} We can consider the base for double signed pipe dreams with countably many $\Gamma$-blocks in the middle. This allows us to obtain double Schubert polynomials $\Ff_w(\zz,\xx,\ttt)$ as sums of monomials over infinite reduced  double signed pipe dreams.
\end{remark}

\subsection{Examples}\label{ssec:examples}

\begin{example}
Let  $w=s_0\in\BCc_\infty$. All pipe dreams  $D\in\PD_{\Bc_\infty}(s_0)$ have exactly one significant element: a faucet in a corner of any $\Gamma$-block. Hence
\begin{eqnarray*}
\Bf_{s_0}(\zz,\xx)=\sum x_i=p_1(\xx);\\
\Cf_{s_0}(\zz,\xx)=2\sum x_i=2p_1(\xx).
\end{eqnarray*}
\end{example}

\begin{example}
Let $w=s_i\in \BCc_\infty$ and $i\geq 1$, or $w=s_i\in\Dc_\infty$ and $i\geq 2$. Each pipe dream (of any type) of this shape has a unique significant element: a cross in a box labeled by $i$. It can be located either in one of the first $i$ rows of the staircase block, or in any of the two parts of any $\Gamma$-block. Hence
\[
\Bf_{s_i}(\zz,\xx)=\Cf_{s_i}(\zz,\xx)=\Df_{s_i}(\zz,\xx)=z_1+z_2+\ldots+z_i+2p_1(\xx).
\]
Now consider double signed pipe dreams of shape $s_i$. The cross can be also located in one of the first $i$ columns of the upper staircase block, so
\[
\Bf_{s_i}(\zz,\xx,\ttt)=\Cf_{s_i}(\zz,\xx,\ttt)=\Df_{s_i}(\zz,\xx,\ttt)=(z_1-t_1)+(z_2-t_2)+\ldots+(z_i-t_i)+2p_1(\xx).
\]
\end{example}

\begin{example}
In a reduced $d$-signed pipe dream of shape $s_1\in\Dc_\infty$ the unique cross can be located either in the corner of the staircase block, or in the corner of any $\Gamma$-block. In the case of  $d$-double signed pipe dreams the cross can also be situated in the corner of the upper staircase block. This means that
\begin{eqnarray*}
\Df_{s_1}(\zz,\xx)=z_1+p_1(\xx);\\
\Df_{s_1}(\zz,\xx,\ttt)=(z_1-t_1)+p_1(\xx).
\end{eqnarray*}
	
Now consider a reduced $d$-signed pipe dream of shape $s_{\hat 1}\in\Dc_n$. They contain one cross with a faucet, which can be situated only in the corner of a $\Gamma$-blocks, hence
\[
\Df_{s_{\hat 1}}(\zz,\xx)=\Df_{s_{\hat 1}}(\zz,\xx,\ttt)=p_1(\xx).
\]
\end{example}

\begin{example}
Let $w=s_1s_{\hat 1}\in\Dc_\infty$. If  $D$ is a reduced $d$-signed pipe dream of shape $w$, the following cases may occur:
\begin{itemize}
		\item There is a cross in the corner of the staircase block and a cross with a faucet in the corner of the $i$-th $\Gamma$-block. For each $i$ we obtain a summand $z_1x_i$.

		\item In the corner of the $i$-th $\Gamma$-block there is a cross, and in the $j$-th $\Gamma$-block there is a cross with a faucet. For each $i\ne j$ this gives us a summand $x_ix_j$.

		\item In the corner of the  $i$-th $\Gamma$-block there is an elbow joint with two faucets. For each $i$ we obtain a summand $x_i^2$.
\end{itemize}

So we have
\[
\Dc_{s_1s_{\hat 1}}(\zz,\xx)=z_1\sum_i x_i+\sum_i x_i^2+\sum_{i \ne j} x_ix_j=z_1p_1(\xx)+p_1^2(\xx).
\]
If  $D$ is a reduced $d$-double signed pipe dream, another possibility occurs: a cross can be located in the corner of the upper staircase block, while a cross with a faucet is in the corner of the $i$-th $\Gamma$-block. These cases provide summands of the form $-t_1x_i$, hence
\[
\Dc_{s_1s_{\hat 1}}(\zz,\xx,\ttt)=(z_1-t_1)p_1(\xx)+p_1^2(\xx).
\]
\end{example}

\begin{example}
Let $w=s_1s_0s_1=1\overline{2}\in\BCc_\infty$. If $D\in\PD_{\Bc_2}(w)$, the pipe dream $D$ has two crosses in the boxes indexed by $1$ and a faucet between them. The following cases may occur:
	\begin{itemize}
		\item The lower cross is in the corner of the staircase block, the faucet and the upper cross are in the same ($i$-th) $\Gamma$-block. For each $i$ we obtain $z_1x_i^2$.
		\item  The lower cross is in the corner of the staircase block, the faucet is in the $i$-th $\Gamma$-block, the upper cross is in any of of the two parts of the $j$-th $\Gamma$-block. For each  $i<j$ we obtain a summand  $2z_1x_ix_j$.
		\item All three elements are situated in the  $i$-th $\Gamma$-block. For each $i$ we obtain a summand $x_i^3$.
		\item The lower cross is in any of the two parts of the  $i$-th $\Gamma$-block, the upper cross and the faucet are in the $j$-th $\Gamma$-block. For each $i<j$ we get a summand $2x_ix_j^2$.
		\item The lower cross and the faucet are in the  $i$-th $\Gamma$-block, and the upper cross is in any of the two parts of the $j$-th $\Gamma$-block. For each $i<j$ we have $2x_i^2x_j$.
		\item The lower cross is in any of the two parts of the $i$-th $\Gamma$-block, the faucet is in the  $j$-th $\Gamma$-block, the upper cross is in any of the two parts of the $k$-th $\Gamma$-block. For each $i<j<k$ we have a summand  $4x_ix_jx_k.$
	\end{itemize}
Summarizing, we obtain
\begin{multline*}
\Bf_{s_1s_0s_1}(\zz,\xx)=z_1\sum_i x_i^2+2z_1\sum_{i< j}x_ix_j+\sum_i x_i^3+\\+2\sum_{i < j}x_i^2 x_j+2\sum_{i < j}x_i x_j^2+4\sum_{i<j<k}x_ix_jx_k=
z_1 p_1^2(\xx)+\frac 2 3 p_1^3(\xx)+\frac 1 3 p_3(\xx).
\end{multline*}
\end{example}

\section{Relation to extended excited Young diagrams}\label{sec:kirnar}

\subsection{The construction by Kirillov and Naruse}\label{ssec:kirnardef}
In this section we compare our construction to the construction of pipe dreams described by An.~Kirillov and H.~Naruse in~\cite{KirillovNaruse17} by means of ``extended excited Young diagrams''. We refer the reader to Sec.~8.2 of this paper.

Let us recall their construction for type $B$. They start with a trapezoidal skew Young diagram\footnote{We turn the figures from \cite{KirillovNaruse17} 90 degrees clockwise; this allows us to draw Young diagrams in the English notation, as opposed to the French one, used in the cited paper.} consisting of a staircase block $(n-1,n-2,\dots,2,1)$ with $n$ rows of length $n$ above it, with each upper row shifted by one to the right with respect to the row below it, see Figure~\ref{fig:knpipedream}. This diagram is called an extended excited  Young diagram (extended EYD for short). A pipe dream is a subset of boxes of this diagram, marked by crosses.

Each box of an extended EYD corresponds to the sum of two variables: $z_i+x_j$, $x_i+x_j$, or $x_i-t_j$ (our variables   $z_i$ and $-t_j$ are denoted in  \cite{KirillovNaruse17} by $a_i$ and $b_j$, respectively). The \emph{weight} of a pipe dream is the product of all such binomials corresponding to the boxes marked by crosses.

Moreover, each box of the extended EYD corresponds to a simple reflection in the corresponding Weyl group $s_0,\dots,s_{n-1}$ for the types $B$ and $C$, $s_1,s_{\hat 1},s_2,\dots,s_{n-1}$ for the type $D$. The \emph{permutation} corresponding to a pipe dream is the product of all simple reflections corresponding to the marked boxes, read from right to left and from top to bottom. 

Thus, the (truncated) Schubert polynomial $\Fc_w^{[n]}(\zz,\xx,\ttt)$ is equal to the sum of the weights for all reduced pipe dreams with the permutation $w$ (see \cite[Thm 5, Thm 6]{KirillovNaruse17}).

Here we reproduce Example~8 from the aforementioned paper.

\begin{example}  Consider the signed permutation $w=2\bar{3}1=s_2s_1s_2s_0s_1\in\BCc_3$. Here is an example of pipe dream corresponding to this permutation. Its weight is equal to $(x_3-t_2)(x_3-t_1)(x_2-t_1)x_2(z_2+x_1)$.
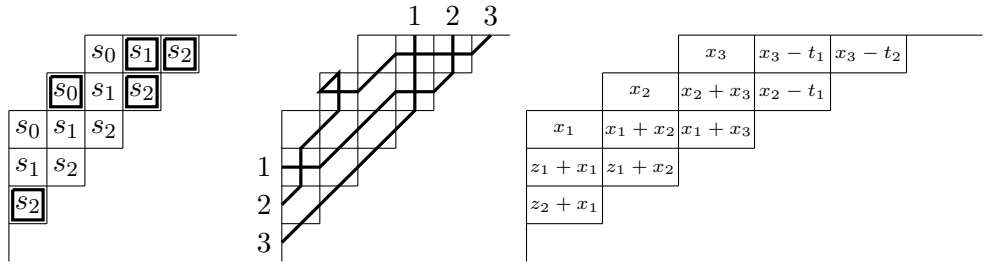
\begin{figure}[h!]\label{KN_example}
\begin{tikzpicture}[scale=0.5]
\draw (0,-6)--(0,-2);
\draw (1,-5)--(1,-1);
\draw (2,-4)--(2,0);
\draw (3,-3)--(3,0);
\draw(4,-2)--(4,0);
\draw (5,-1)--(5,0);
\draw (2,0)--(6,0);
\draw (1,-1)--(5,-1);
\draw (0,-2)--(4,-2);
\draw (0,-3)--(3,-3);
\draw (0,-4)--(2,-4);
\draw (0,-5)--(1,-5);
\node at (2.5,-0.5) {$s_0$};
\node at (1.5,-1.5) {$s_0$};
\node at (.5,-2.5) {$s_0$};
\node at (3.5,-0.5) {$s_1$};
\node at (2.5,-1.5) {$s_1$};
\node at (1.5,-2.5) {$s_1$};
\node at (.5,-3.5) {$s_1$};
\node at (4.5,-0.5) {$s_2$};
\node at (3.5,-1.5) {$s_2$};
\node at (2.5,-2.5) {$s_2$};
\node at (1.5,-3.5) {$s_2$};
\node at (.5,-4.5) {$s_2$};

\draw[very thick] (3.1,-0.1)--(3.9,-0.1)--(3.9,-0.9)--(3.1,-0.9)--(3.1,-0.1);
\draw[very thick] (4.1,-0.1)--(4.9,-0.1)--(4.9,-0.9)--(4.1,-0.9)--(4.1,-0.1);
\draw[very thick] (3.1,-1.1)--(3.9,-1.1)--(3.9,-1.9)--(3.1,-1.9)--(3.1,-1.1);
\draw[very thick] (1.1,-1.1)--(1.9,-1.1)--(1.9,-1.9)--(1.1,-1.9)--(1.1,-1.1);
\draw[very thick] (0.1,-4.1)--(0.9,-4.1)--(0.9,-4.9)--(0.1,-4.9)--(0.1,-4.1);
\end{tikzpicture}
\begin{tikzpicture}[scale=0.5]
\draw (0,-6)--(0,-2);
\draw (1,-5)--(1,-1);
\draw (2,-4)--(2,0);
\draw (3,-3)--(3,0);
\draw(4,-2)--(4,0);
\draw (5,-1)--(5,0);
\draw (2,0)--(6,0);
\draw (1,-1)--(5,-1);
\draw (0,-2)--(4,-2);
\draw (0,-3)--(3,-3);
\draw (0,-4)--(2,-4);
\draw (0,-5)--(1,-5);
\node [above] at (3.5,0) {$1$};
\node [above] at (4.5,0) {$2$};
\node [above] at (5.5,0) {$3$};
\node [left] at (0,-3.5) {$1$};
\node [left] at (0,-4.5) {$2$};
\node [left] at (0,-5.5) {$3$};
\draw [very thick] (0,-5.5)--(3.5,-2)--(3.5,0);
\draw [very thick] (0,-4.5)--(0.5,-4)--(0.5,-3)--(1.5,-2)--(1.5,-1)--(1,-1.5)--(2,-1.5)--(3,-0.5)--(5,-0.5)--(5.5,0);
\draw [very thick](0,-3.5)--(1,-3.5)--(3,-1.5)--(4,-1.5)--(4.5,-1)--(4.5,0);
\end{tikzpicture}
\begin{tikzpicture}[scale=0.5]
\draw (0,-6)--(0,-2);
\draw (2,-5)--(2,-1);
\draw (4,-4)--(4,0);
\draw (6,-3)--(6,0);
\draw(8,-2)--(8,0);
\draw (10,-1)--(10,0);
\draw (4,0)--(12,0);
\draw (2,-1)--(10,-1);
\draw (0,-2)--(8,-2);
\draw (0,-3)--(6,-3);
\draw (0,-4)--(4,-4);
\draw (0,-5)--(2,-5);
\node at (5,-0.5) {\tiny{$x_3$}};
\node at (3,-1.5) {\tiny{$x_2$}};
\node at (1,-2.5) {\tiny{$x_1$}};
\node at (7,-0.5) {\tiny{$x_3-t_1$}};
\node at (5,-1.5) {\tiny{$x_2+x_3$}};
\node at (3,-2.5) {\tiny{$x_1+x_2$}};
\node at (1,-3.5) {\tiny{$z_1+x_1$}};
\node at (9,-0.5) {\tiny{$x_3-t_2$}};
\node at (7,-1.5) {\tiny{$x_2-t_1$}};
\node at (5,-2.5) {\tiny{$x_1+x_3$}};
\node at (3,-3.5) {\tiny{$z_1+x_2$}};
\node at (1,-4.5) {\tiny{$z_2+x_1$}};
\end{tikzpicture}
\caption{An extended excited Young diagram}\label{fig:knpipedream}
\end{figure}
\end{example}

The boxes of each diagonal of this skew Young diagram are indexed by the corresponding simple reflection; this allows us to assign to an extended EYD an element of $\Fc_n$. There is another indexing of its boxes, shown on the same figure on the right: each box, except those on the topmost antidiagonal, corresponds to a \emph{binomial} (as opposed to a monomial in our case), i.e. the sum of variables $x_i$, $z_i$, or $-t_i$. Each extended EYD with $d$ crosses thus produces a polynomial with $2^d$ monomials. In our setting, it corresponds to $2^d$ different pipe dreams. In the next subsection we describe the procedure producing these pipe dreams from the extended EYD.

\begin{remark} These two constructions of pipe dreams for are somewhat parallel to the constructions of pipe dreams of double Schubert polynomials due to Bergeron--Billey and Fomin--Kirillov, respectively. In the former paper, the authors assign a pipe dream to each monomial of double Schubert polynomial; this is similar to our construction presented in Section~\ref{sec:main}. Corollary~\ref{cor:double} is thus a direct generalization of \cite[(4.1)]{BilleyBergeron93}. Fomin and Kirillov, on the other hand, assign to each pipe dream a product of several binomials (see \cite[Thm~6.2]{FominKirillov96}), just like in~Theorem~5  of the paper by Kirillov and Naruse \cite{KirillovNaruse17} discussed in this section. The latter presentations are thus more efficient (or, put it differently, coarser): each diagram corresponds to several monomials, as opposed to one. Algebraically, the relation between these two constructions is nothing but the Cauchy expansion formula (see Theorem~\ref{doubleschub} above).\footnote{We are grateful to the referee for this remark.}
\end{remark}

\subsection{Producing pipe dreams from an extended EYD}\label{ssec:kirnarbijection}

Starting from an extended excited Young diagram, one can produce several  pipe dreams corresponding to the same permutation, such that the sum of the monomials corresponding to the pipe dreams is equal to the weight of the extended EYD.

	Let us illustrate the bijection of our pipe dreams with those from the paper \cite{KirillovNaruse17} with an example of the type $B$. For the types $C$ and $D$, this bijection is constructed similarly.
	
Informally, a pipe dream from~\cite{KirillovNaruse17} can be obtained from our pipe dream by pushing the bottom staircase block, $n$ copies of $\Gamma$-blocks, and the top staircase block one into another. Our aim is to ``pull these blocks apart'' in such a way that the shape of the pipe dream and the number of significant elements in each of the blocks remain unchanged.

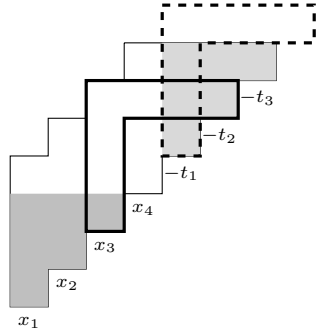
\begin{figure}[h!]
\begin{tikzpicture}[scale=0.5]
\draw(3,0)--(7,0)--(7,-1)--(6,-1)--(6,-2)--(5,-2)--(5,-3)--(4,-3)--(4,-4)--(3,-4)--(3,-5)--(2,-5)--(2,-6)--(1,-6)--(1,-7)--(0,-7)--(0,-3)--(1,-3)--(1,-2)--(2,-2)--(2,-1)--(3,-1)--(3,0);
\fill[black!25!white] (0,-7)--(0,-4)--(3,-4)--(3,-5)--(2,-5)--(2,-6)--(1,-6)--(1,-7)--(0,-7);
\fill[black!15!white] (4,0)--(7,0)--(7,-1)--(6,-1)--(6,-2)--(5,-2)--(5,-3)--(4,-3)--cycle;
\draw[very thick] (2,-1)--(6,-1)--(6,-2)--(3,-2)--(3,-5)--(2,-5)--cycle;
\draw[dashed,very thick] (4,1)--(8,1)--(8,0)--(5,0)--(5,-3)--(4,-3)--cycle;
\node[below] at (0.5,-7) {\tiny{$x_1$}};
\node[below] at (1.5,-6) {\tiny{$x_2$}};
\node[below] at (2.5,-5) {\tiny{$x_3$}};
\node[below] at (3.5,-4) {\tiny{$x_4$}};
\node[below] at (4.5,-3) {\tiny{$-t_1$}};
\node[below] at (5.5,-2) {\tiny{$-t_2$}};
\node[below] at (6.5,-1) {\tiny{$-t_3$}};
\end{tikzpicture}
\caption{An extended EYD. The regions corresponding to lower (resp. upper) staircase blocks are shaded in dark (resp. light) grey; the  $\Gamma$-blocks corresponding to $t_1$ and $x_3$  are highlighted.}\label{fig:pullapart}
\end{figure}

To do this, for each binomial $x_i-t_j$, $x_i+x_j$, or $z_i+x_j$ corresponding to a cross in the extended EYD let us select one of the two terms.  We want to construct a pipe dream with the same permutation and the monomial equal to the product of the selected terms. That means that the crosses in the lower/upper staircase blocks of this pipe dream should correspond to the crosses in the extended EYD marked by $z_i$ and $-t_j$ respectively, and each cross in the extended EYD marked by $x_i$ should produce a cross or a faucet in the $i$-th $\Gamma$-block of the resulting pipe dream. This will be done as follows. 

The diagram can be represented as the union of several $\Gamma$-blocks (an example with two of them is shown on Figure~\ref{fig:pullapart}). They will be indexed, counting from top to bottom, by the following variables: $-t_{n-1},\dots,-t_1,x_n,\dots,x_1$. On the figure we write the corresponding variable under each $\Gamma$-block. The $\Gamma$-blocks indexed by the $-t_i$'s contain significant elements only in their vertical parts: their horizontal parts are located ``outside" the pipe dream. 

Let us pull out the $\Gamma$-blocks one by one in the northeast direction, starting from the top one. After shifting the first block far enough, we proceed with the next $\Gamma$-block, counted from the top, and so on. We shift the first $n-1$ blocks, corresponding to the $-t_i$'s, by the same number of positions northeast, in such a way that the shifted blocks would again form an upper staircase of our pipe dream, but it does not intersect any lower blocks. Then we shift all the subsequent $\Gamma$-blocks, corresponding to $x_i$, also starting with the top one, in such a way that finally they are separate (and thus form a pipe dream).

 Further we describe one step of this procedure, i.e., how to shift the topmost, i.e.  the $n$-th, $\Gamma$-block one box up and right.

\begin{figure*}
$$
\includegraphics[width=14cm]{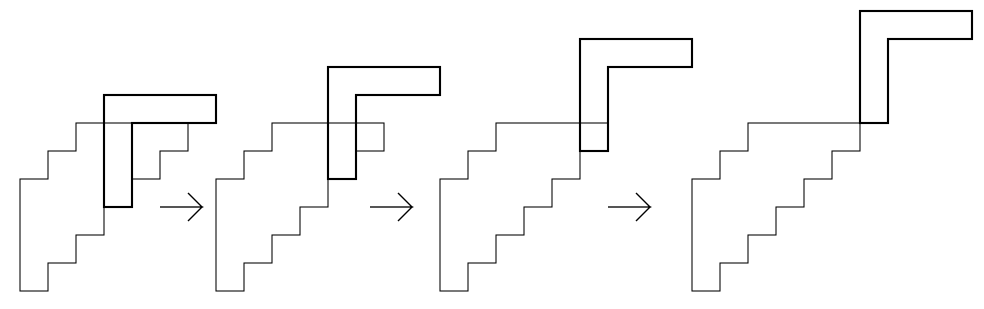}
$$
\caption{Shifting of the topmost $\Gamma$-block northeast}	
\end{figure*}

Let us draw circles around the significant elements corresponding to the variable $x_n$, i.e. the elements which belong to the block we are shifting.

All the significant elements in the topmost row necessarily belong to the topmost $\Gamma$-block; we shift them one box northeast.

$$\includegraphics[width=2cm]{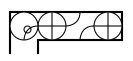}$$

Now consider the crosses in the vertical part of the $\Gamma$-block that we are shifting. Denote the column containing this vertical part and the column immediately right of it by $V$ and $V+1$, respectively. Note that all nonencircled crosses in the column $V$ correspond to the horizontal parts of their $\Gamma$-blocks (or belong to the rows of the staircase block). So we need to shift some crosses in such a way that:
\begin{itemize}
	\item the shape of the pipe dream remains the same;
	\item the number of encircled crosses in the column $V+1$ after the shift is equal to the number of encircled crosses in the column $V$ before the shift;
	\item the number of nonencircled crosses in each row remains the same.
\end{itemize}
Let us move the encircled crosses one by one, starting from the top. We distinguish between the two cases:
\begin{itemize}
	\item There is a cross to the right of the encircled cross. This means that the encircled cross does not belong to the bottom line of the $\Gamma$-block. In this case, we swap the encircled cross with the nonencircled one (so the encircled cross moves one box down with respect to the $\Gamma$-block, see the figure below).
	$$\includegraphics[width=3.5cm]{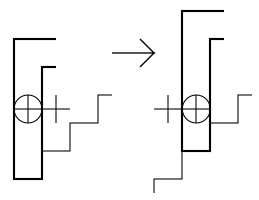}$$
	This operation does not involve any elements to the top of the encircled cross.
	$$\includegraphics[width=3.5cm]{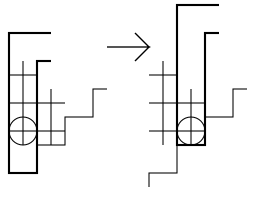}$$
	\item There is an elbow to the right of the encircled cross. We need to distinguish between the following subcases.
	\begin{itemize}
		\item If above $\crosscirc\elbows$ there are two elbows $\elbows\elbows$, the encircled cross is shifted northeast.
		$$\includegraphics[width=3.5cm]{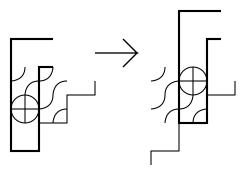}$$
		\item If above $\crosscirc\elbows$ there are several rows with a cross followed by an elbow: $\cross\elbows$, the whole column of crosses is shifted northeast, and we put a circle around the topmost of them.
			$$\includegraphics[width=3.5cm]{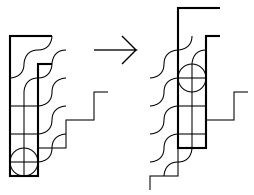}$$
		\item If above $\crosscirc\elbows$ there are several lines with double crosses: $\cross\cross$, the encircled cross is shifted by a ladder move.
		$$\includegraphics[width=3.5cm]{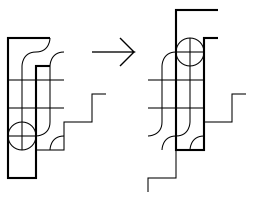}$$
		\item In the general case we proceed as follows. Consider the columns $V$ and $V+1$. Above $\crosscirc\elbows$ we have some sequence of pairs $\cross\cross$ and $\cross\elbows$, and on the top of this sequence there are two elbows $\elbows\elbows$ (since our pipe dream is reduced, the configuration $\elbows\cross$ cannot  occur below $\elbows\elbows$). Starting from the top, we pull the crosses from the configurations $\cross\elbows$ through $\cross\cross$ using ladder moves and encircle the topmost of the shifted crosses.
		$$\includegraphics[width=3.5cm]{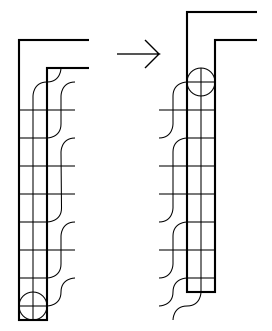}$$
	\end{itemize}
	
\end{itemize}

The figure below shows the procedure of shifting a $\Gamma$-block in case of several encircled crosses.

		$$\includegraphics[width=15cm]{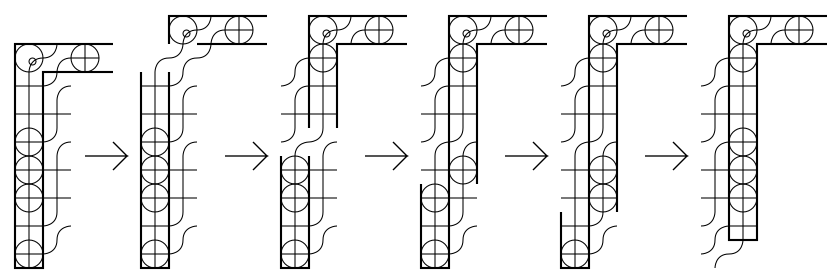}$$
We shift the encircled crosses one by one, starting from the top. However, as we proceed, the order of the encircled crosses may change.

Finally, is is clear that the operation of ``pulling apart'' of $\Gamma$-blocks is invertible: we can bring the blocks back together, starting from the bottom ones and keeping in mind which significant element belongs to which block. This would give us an excited Young diagram together with one of the  monomials occuring in its weight.

The procedure for the types $C$ and $D$ is similar, so we do not  describe it here.

\begin{remark}
T.\,Ikeda and H.\,Naruse~\cite{IkedaNaruse13} proposed a description of $P$- and $Q$-Schur functions in terms of the so-called \emph{excited Young diagrams}. These functions are particular cases of Schubert polynomials of types $B$, $C$ and $D$ (see Section~\ref{sec:grassmann}), corresponding to Grassmannian permutations. The procedure described above is a generalization of the \emph{separation of variables} described in~\cite[Sec.~10]{IkedaNaruse13}.
\end{remark}

\section{Bottom pipe dreams}\label{sec:bottom}

\subsection{Admissible moves and extended Lehmer codes of signed permutations}\label{ssec:lehmer}


Let $w\in\Sc_n$ or $w\in\Fc_n$. Consider the set of pipe dreams of a given shape: $\PD_{A}(w)$ or $\PD_{\Fc_n}(w)$. We introduce a partial order on this set as follows. Let  $D$ and $D'$ be two pipe dreams of the same shape that coincide everywhere except two boxes  $a$ and $b$, and the row containing $a$ is above the row containing $b$. We shall say that  $D$ is obtained from $D'$ by an \emph{admissible move} (notation: $D\prec D'$) if the elements in $a$ and $b$ are located in one of the following nine ways:
\[
\begin{array}{c||c|c}
	& D' & D \\
	\hline\hline
	a \rule{0pt}{13pt} & \cross & \elbows \\
	\hline
	b \rule{0pt}{13pt} & \elbows & \cross
	\end{array} \;\; 
	\begin{array}{c||c|c}
	& D' & D \\
	\hline\hline
	a \rule{0pt}{13pt} & \crosssign & \elbows \\
	\hline
	b \rule{0pt}{13pt} & \elbows & \cross
	\end{array}\;\;
	\begin{array}{c||c|c}
	& D' & D \\
	\hline\hline
	a \rule{0pt}{13pt} & \cross & \elbows \\
	\hline
	b \rule{0pt}{13pt} & \elbows & \crosssign
	\end{array}\;\;
	\begin{array}{c||c|c}
	& D' & D \\
	\hline\hline
	a \rule{0pt}{13pt} & \crosssign & \elbows \\
	\hline
	b \rule{0pt}{13pt} & \elbows & \crosssign
	\end{array}\;\;
	\begin{array}{c||c|c}
	& D' & D \\
	\hline\hline
	a \rule{0pt}{13pt} & \elbowsign & \elbowup \\
	\hline
	b \rule{0pt}{13pt} & \elbowup & \elbowsign
	\end{array}
\]
\[
	\begin{array}{c||c|c}
	& D' & D \\
	\hline\hline
	a \rule{0pt}{13pt} & \elbowssign & \crosssign \\
	\hline
	b \rule{0pt}{13pt} & \elbows & \cross
	\end{array} \;\; 
	\begin{array}{c||c|c}
	& D' & D \\
	\hline\hline
	a \rule{0pt}{13pt} & \elbowssign & \cross \\
	\hline
	b \rule{0pt}{13pt} & \elbows & \crosssign
	\end{array}\;\;
	\begin{array}{c||c|c}
	& D' & D \\
	\hline\hline
	a \rule{0pt}{13pt} & \crosssign & \elbows \\
	\hline
	b \rule{0pt}{13pt} & \cross & \elbowssign
	\end{array}\;\;
	\begin{array}{c||c|c}
	& D' & D \\
	\hline\hline
	a \rule{0pt}{13pt} & \cross & \elbows \\
	\hline
	b \rule{0pt}{13pt} & \crosssign & \elbowssign
	\end{array}
\]
(For the entries in the first line, $D$ is obtained from $D'$ by shifting one significant element down.)

\begin{example} Let $D$, $D'$ be pipe dreams of the type $A$. In this case, only the first of these nine moves is allowed. As a particular case of this move, we have a  \emph{ladder move} defined in~\cite{BilleyBergeron93}:
\[
\begingroup
	\setlength\arraycolsep{0pt}
	\renewcommand{\arraystretch}{0.07}
	\begin{matrix}
	\elbows & \cross\\
	\cross & \cross\\[3pt]
	\hdotsfor{2}\\[3pt]
	\cross & \cross\\
	\elbows & \elbows
	\end{matrix} \mapsto
	\begin{matrix}
	\elbows & \elbows\\
	\cross & \cross\\[3pt]
	\hdotsfor{2}\\[3pt]
	\cross & \cross\\
	\cross & \elbows
	\end{matrix}
	\endgroup
\]
(here the dots represent a sequence of crosses). 
\end{example}
 
Extending the relation   $\prec$ by transitivity, we obtain a partial order on the set of  pipe dreams. This is indeed a partial order. To show this, let us introduce an order on the set of variables as follows: $z_{n-1}<z_{n-2}<\ldots<z_1<x_1<x_2<\ldots$. Consider the corresponding lexicographic order $>$ on the set of monomials. It is clear that for $D\prec D'$ we have $\xx^{\alpha(D)}\zz^{\beta(D)}< \xx^{\alpha(D')}\zz^{\beta(D')}$.

Our next goal is to show that for each $w$ the set $\PD_{A}(w)$ (resp. $\PD_{\Fc_n}(w)$) has a unique minimal element with respect to the partial order $\prec$.  Such a pipe dream is called \emph{the bottom pipe dream} of $w$.

Recall the definition of the Lehmer code of a permutation. We also need to generalize the notion of the Lehmer code for a signed permutation.

\begin{definition}
	Let $w\in\Sc_n$. The \emph{Lehmer code} $L(w)$ of this permutation is a sequence $(L_1(w),\ldots, L_n(w))$, where 
	$$
	L_i(w)=\#\{j>i \mid w(i)>w(j)\}.
	$$
\end{definition}

\begin{definition}
	Let $w\in\BCc_n$. Consider $w$ as a bijection from $\{-n,\ldots,-1,1,\ldots, n\}$ to itself. The \emph{Lehmer code} of $w$ is  a sequence  $L(w)=(L_1(w),\ldots, L_n(w))$ defined as $$L_i(w)=\#\{j>i \mid w(i)>w(j)\}.$$ The \emph{signed Lehmer code} of $w$ is the pair $(L(w),N(w)$), where  $$N(w)=\left(N_1(w)>N_2(w)>\dots>N_{s(w)}(w)\mid w^{-1}\left(N_i(w)\right)<0 \text{ for all } i\right) $$  is the decreasing sequence of positive entries with negative preimages.
\end{definition}

It is well-known (cf., for instance, \cite[p.~9]{Macdonald91}) that a permutation is uniquely defined by its Lehmer code. This immediately implies that a signed permutation is uniquely defined by its signed Lehmer code.

\begin{remark} There seems to be no prevalent notion of the Lehmer code for signed permutations in the literature. For example, the definitions of Lehmer codes for signed permutations given in the papers~\cite{GeckKim97} and~\cite{LascouxSchutzenberger96} are different from ours, and from each other as well. 
\end{remark}



\subsection{Type $A$ pipe dreams}\label{ssec:bottoma}

Here we recall the proof of existence of bottom pipe dream for type $A$ permutations, due to S.\,Billey and N.\,Bergeron~\cite{BilleyBergeron93} (cf. also~\cite{Knutson12}).
\begin{theorem}
	Let $w\in\Sc_n$. There exists a unique \emph{bottom} pipe dream $D_b\in\PD_{A}(w)$ with all crosses adjusted to the left (i.e., not containing fragments of the form $\elbows\cross$). It can be constructed as follows: for all $i=1,\dots, n$ the $i$-th row of $D_b$ contains exactly $L_i(w)$ left-adjusted crosses (here $L(w)=(L_1(w),\ldots, L_n(w))$ is the Lehmer code of $w$.)
	
	Each pipe dream  $D\in\PD_{A}(w)$ can be brought to $D_b$ by a sequence of {ladder moves}.
\end{theorem}

\begin{proof}
	First observe that every pipe dream without fragments of the form $\elbows\cross$ is reduced.
	
	It is easy to check that if for each  $i=1,\ldots,n$ the  $i$-th row of the pipe dream $D_b$ has $L_i(w)$ crosses adjusted to the left, then $D_b$ has indeed the shape $w\in\Sc_n$. Since a permutation is uniquely determined by its Lehmer code, $D_b$ is the unique bottom pipe dream in $\PD_{A}(w)$.
	
Consider a pipe dream $D\in\PD_{A}(w)$. Since it is reduced, it cannot contain a fragment of the form 
	$\begingroup
	\setlength\arraycolsep{0pt}
	\renewcommand{\arraystretch}{0.07}
	\begin{matrix}
	\elbows & \cross\\
	\cross & \cross\\[3pt]
	\hdotsfor{2}\\[3pt]
	\cross & \cross\\
	\cross & \elbows
	\end{matrix}\endgroup$. If $D$ is not the bottom pipe dream, it contains fragments of the form $\elbows\cross$. Let us take the lowest row containing such fragments; among them, take the leftmost occurence. The cross in it can be shifted downstairs by means of a ladder move. Repeating such an operation, we will obtain the bottom pipe dream for $w$.
\end{proof}
\begin{example}
Let $w=126543\in\Sc_6$. The Lehmer code of this permutation equals $L(w)=(0,0,3,2,1,0)$. Let us show how to bring one of the pipe dreams of shape $w$ to the bottom one.
$$\begingroup
\setlength\arraycolsep{0pt}
\renewcommand{\arraystretch}{0.07}
\begin{matrix}
& 1 & 2 & 3 & 4 & 5 & 6  \\
1 & \elbows & \elbows & \cross & \elbows & \elbows & \elbow & \\
2 & \elbows & \cross & \cross & \elbows & \elbow & \\
3 & \cross & \cross & \cross & \elbow & \\
4 & \elbows & \elbows & \elbow & \\
5 & \elbows & \elbow & \\
6 & \elbow & \\
\end{matrix} \mapsto \;
\begin{matrix}
& 1 & 2 & 3 & 4 & 5 & 6  \\
1 & \elbows & \elbows & \cross & \elbows & \elbows & \elbow & \\
2 & \elbows & \elbows & \cross & \elbows & \elbow & \\
3 & \cross & \cross & \cross & \elbow & \\
4 & \cross & \elbows & \elbow & \\
5 & \elbows & \elbow & \\
6 & \elbow & \\
\end{matrix} \mapsto \;
\begin{matrix}
& 1 & 2 & 3 & 4 & 5 & 6  \\
1 & \elbows & \elbows & \cross & \elbows & \elbows & \elbow & \\
2 & \elbows & \elbows & \elbows & \elbows & \elbow & \\
3 & \cross & \cross & \cross & \elbow & \\
4 & \cross & \cross & \elbow & \\
5 & \elbows & \elbow & \\
6 & \elbow & \\
\end{matrix} \mapsto \;
\begin{matrix}
& 1 & 2 & 3 & 4 & 5 & 6  \\
1 & \elbows & \elbows & \elbows & \elbows & \elbows & \elbow & \\
2 & \elbows & \cross & \elbows & \elbows & \elbow & \\
3 & \cross & \cross & \cross & \elbow & \\
4 & \cross & \cross & \elbow & \\
5 & \elbows & \elbow & \\
6 & \elbow & \\
\end{matrix} \mapsto \;
\begin{matrix}
& 1 & 2 & 3 & 4 & 5 & 6  \\
1 & \elbows & \elbows & \elbows & \elbows & \elbows & \elbow & \\
2 & \elbows & \elbows & \elbows & \elbows & \elbow & \\
3 & \cross & \cross & \cross & \elbow & \\
4 & \cross & \cross & \elbow & \\
5 & \cross & \elbow & \\
6 & \elbow & \\
\end{matrix}
\endgroup$$
\end{example}

\subsection{B-signed pipe dreams}\label{ssec:bottomb}
\begin{theorem}\label{bbottom}
	Let $w\in\Bc_n$. There exists a unique \emph{bottom} b-signed pipe dream $D_b\in\PD_{\Bc_n}(w)$ that satisfies the following conditions:
	\begin{itemize}
		\item $D_b$ does not contain fragments of the form $\elbows\cross$ and $\elbowup\cross$;
		\item The numbers of significant elements in its  $\Gamma$-blocks form a strictly decreasing sequence (ending by zeroes).
	\end{itemize}
	
	This pipe dream $D_b$ can be constructed as follows. Let $(L(w),N(w))$ be the signed Lehmer code of $w$. For each $i=1,\dots, n$ the $i$-th row of the staircase block in $D_b$ contains $L_i(w)$ left-adjusted crosses. And for each $j=1,\dots,s(w)$ the horizontal part of the $j$-th $\Gamma$-block contains $N_j(w)$  left-adjusted significant elements (i.e. an elbow with a faucet and  $N_i(w)-1$ crosses).
	
	Moreover, each $b$-signed pipe dream $D\in\PD_{\Bc_n}(w)$ can be brought to $D_b$ by a sequence of admissible moves.
\end{theorem}

\begin{proof}
	By definition, the bottom $b$-signed pipe dream $D_b$	satisfies the following conditions:
	\begin{itemize}
		\item All the crosses in the staircase block are left-adjusted. 
		\item All the significant elements in each $\Gamma$-block are located in its horizontal part and are left-adjusted. Hence each nonempty $\Gamma$-block has a faucet in its corner. Denote the number of significant elements in the $i$-th $\Gamma$-block by $\mu_i$.
		\item The total number of faucets equals $s(w)$, which is the number of nonempty $\Gamma$-blocks. Since the numbers of significant elements in $\Gamma$-blocks decrease, we have $n\geq\mu_1>\mu_2>\ldots>\mu_{s(w)}>0$.
	\end{itemize}	
	
It is easy to observe that such b-signed pipe dream $D_b$ with $L_i(w)$ crosses in $i$-th row of the staircase block and $N_j(w)$ significant elements in $j$-th $\Gamma$-block is reduced and has the shape $w$. Since a permutation $w\in\BCc_n$ is uniquely determined by its signed Lehmer code $(L(w),N(w))$, $D_b$ is the unique bottom b-signed pipe dream in $\PD_{B}(w)$.  

It remains to show that each non-bottom $b$-signed pipe dream $D\in\PD_{\Bc_n}(w)$ can be brought to the bottom one by shifting significant elements down.

Suppose that $D$ contains a fragment  $\elbows\cross$ or $\elbowup\cross$. Consider the lowest of them; let $a$ be its box containing a cross. Let us go down along the strands crossing at $a$, and consider the next pairs of adjacent boxes containing these strands. They can look as follows:
\begin{itemize}
	\item $\elbows\elbow$ or $\elbows\elbows$. Let  $b$  be the left box of the first such pair. Then the strands cross at $a$ and ``nearly meet'' in $b$. Note that if one of these strands has a faucet on it, it is located below $b$. Thus we can move the cross from $a$ into $b$ without changing the shape of our b-signed pipe dream.
	\item $\cross\cross$: in this case the two strands continue passing next to each other.
	\item $\elbows\cross$ cannot occur, because $a$ was the lowest box with such a condition.
	\item $\cross\elbows$ only can occur if one of the strands had a faucet on it (otherwise $D$ is nonreduced). But before the faucet the corresponding strand should turn left, hence there is a fragment of the form $\elbows\elbows$ before it, and this case was already considered.
\end{itemize}

Summarizing, we see that the two strands crossing in $a$ follow next to each other and do not contain faucets before passing through a fragment of the form $\elbows\elbows$ or $\elbows\elbow$. Since each strand contains at least one elbow in the staircase block, we necessarily obtain such a fragment. So the cross from $a$ can be moved down:
$$\begingroup
\setlength\arraycolsep{0pt}
\renewcommand{\arraystretch}{0.07}
\begin{matrix}
&&&&&&\elbowup&\cross\\
&&&&&\elbowup&\elbows&\elbow\\[4pt]
&&&&\hdotsfor{3}\\[4pt]
&&&\elbowup&\elbows&\elbow&\\
&&&\cross&\cross&&\\
&&\elbowup&\elbows&\elbow&&\\[4pt]
&\hdotsfor{3}\\[4pt]
\elbowup&\elbows&\elbow&&&\\
\cross&\cross&&&&\\
\elbows&\elbow&&&&
\end{matrix}\mapsto\begin{matrix}
&&&&&&\elbowup&\elbows\\
&&&&&\elbowup&\elbows&\elbow\\[4pt]
&&&&\hdotsfor{3}\\[4pt]
&&&\elbowup&\elbows&\elbow&\\
&&&\cross&\cross&&\\
&&\elbowup&\elbows&\elbow&&\\[4pt]
&\hdotsfor{3}\\[4pt]
\elbowup&\elbows&\elbow&&&\\
\cross&\cross&&&&\\
\cross&\elbow&&&&
\end{matrix}\endgroup
$$
This allows us to bring $D$ to the following form: all significant elements in the $\Gamma$-blocks  are located in their horizontal parts and are left-adjusted. 

Now suppose that the number of significant elements in $\Gamma$-blocks does not form a strictly decreasing sequence. Let these blocks have $\mu_1,\mu_2,\ldots$  significant elements and let  $i$ be the minimal number satisfying $\mu_{i+1}\geq \mu_i>0$ or $\mu_{i+1}>0, \mu_{i}=0$. 

If $\mu_i=0$ and $\mu_{i+1}>0$, we can move all the elements from the $(i+1)$-st block into the $i$-th one; obviously, this does not change the shape of our pipe dream. Thus we can reduce the situation to the case $\mu_1>0,\mu_2>0,\ldots,\mu_{s(w)}>0$, with all the following blocks being empty.

In the next examples we will draw the horizontal parts of the $\Gamma$-blocks adjacent to each other, ignoring the elbows between them.

A reduced pipe dream cannot contain a fragment of the form
$\begingroup
\setlength\arraycolsep{0pt}
\renewcommand{\arraystretch}{0.07}
\begin{matrix}
&\elbowsign\\
\elbowsign&\elbows
\end{matrix}\endgroup$, so the horizontal parts of all non-empty $\Gamma$-blocks except the top one must contain at least two crosses. Also, reduced pipe cannot contain $\begingroup
\setlength\arraycolsep{0pt}
\renewcommand{\arraystretch}{0.07}
\begin{matrix}
&&\elbowsign\\
&\elbowsign&\cross\\
\elbowsign&\cross&\elbows
\end{matrix}\endgroup$, so the horizontal parts of all non-empty $\Gamma$-blocks except the top two must contain at least three crosses. Proceeding in this way further, we obtain $\mu_j\geq s(w)-j+1$.

Now consider the box in the  $(i+1)$-th $\Gamma$-block on the $\mu_i$-th position, counted from the left; denote this box by $a$. Since $\mu_{i+1}\geq\mu_i$, this box contains a cross. Consider the pair of strands crossing in $a$ and follow then down.

It is clear that both of these strands have faucets; moreover, they enter the staircase block at adjacent positions. Similarly to the previous case we show that they pass through a box $b$ of the staircase block containing an elbow $\elbows$. Since both strands have a faucet between $a$ and $b$, we can move the cross from $a$ to $b$ without changing the shape of our b-signed pipe dream, as shown on the diagram below:
$$
\begingroup
\setlength\arraycolsep{0pt}
\renewcommand{\arraystretch}{0.07}
\begin{matrix}
&&\elbowsign&\cross&\cross&\cross&\elbow\\
&\elbowsign&\cross&\cross&\elbows&\elbow&\\[4pt]
&\hdotsfor{2}&&&&\\[4pt]
&\cross&\cross&&&&\\
&\elbows&\elbow&&&&\\
\end{matrix}\mapsto
\begin{matrix}
&&\elbowsign&\cross&\elbows&\cross&\elbow\\
&\elbowsign&\cross&\cross&\elbows&\elbow&\\[4pt]
&\hdotsfor{2}&&&&\\[4pt]
&\cross&\cross&&&&\\
&\cross&\elbow&&&&\\
\end{matrix}\endgroup
$$

Repeating these operations, we will obtain the bottom $b$-signed pipe dream $D_b$.
\end{proof}

\subsection{C-signed pipe dreams}\label{ssec:bottomc}

The case of $c$-signed pipe dreams is treated similarly to the case of $b$-signed pipe dreams. For $w\in\Cc_n$ the bottom pipe dream $D_b\in\PD_{\Cc_n}(w)$ looks like the bottom pipe dream $D'_b\in\PD_{\Bc_n}(w)$, with one exception: all the elbows with faucets are located in the vertical parts of the $\Gamma$-blocks. Let us state an analogue of Theorem~\ref{bbottom}:

\begin{theorem}\label{cbottom}
Let $w\in\BCc_n$. There exists a unique \emph{bottom} c-signed pipe dream $D_b\in\PD_{\Cc_n}(w)$ with the following properties:
	\begin{itemize}
			\item all the crosses in the staircase block are left-adjusted;
		\item every nonempty $\Gamma$-block contains an elbow with a faucet in the vertical part of the block;
		\item all the crosses in the $\Gamma$-blocks are located in the horizontal parts and left-adjusted: the box indexed by $0$ contains an elbow $\elbowup$ and is followed by a sequence of crosses;
		\item the numbers of significant elements in $\Gamma$-blocks form a strictly decreasing sequence.
	\end{itemize}
	This pipe dream $D_b$ has $L_i(w)$ crosses in the $i$-th row of the staircase block and $N_j(w)$ significant elements in the $j$-th $\Gamma$-block (here $(L(w),N(w))$ is the signed Lehmer code of $w\in\BCc_n$).
	
	Any $c$-signed pipe dream  $D\in\PD_{\Cc_n}(w)$ can be brought to $D_b$ by a sequence of admissible moves.
	\end{theorem}

\subsection{D-signed pipe dreams}\label{ssec:bottomd}

\begin{theorem}\label{dbottom}
	Let $w\in\Dc_n$. There exists a unique \emph{bottom} d-signed pipe dream $D_b\in\PD_{\Dc_n}(w)$ satisfying the following properties:
	\begin{itemize}
		\item $D_b$ does not contain fragments of the form $\elbows\cross$ (however, $\elbowup\cross$ may occur);
		\item  the numbers of significant elements in the $\Gamma$-blocks form a strictly decreasing sequence (ending by zeroes);
		\item The $1'$-boxes of nonempty $\Gamma$-blocks contain crosses $\cross$ and crosses with faucets $\crosssign$. These elements alternate: the blocks with odd numbers contain crosses with faucets $\crosssign$, while those with even numbers contain ordinary crosses $\cross$.
	\end{itemize} 
	This pipe dream $D_b$ can be constructed as follows. Let $(L(w), N(w))$ be the signed Lehmer code for $w\in\Dc_n$ and let
	\[
	\mu(w)=\begin{cases}
	\left(N_1(w)-1, N_2(w)-1,\ldots, N_{s(w)}(w)-1\right),&\text{ if } N_{s(w)}(w)>1;\\
	\left(N_1(w)-1, N_2(w)-1,\ldots, N_{s(w)-1}(w)-1\right),&\text{ if } N_{s(w)}(w)=1.
	\end{cases}
	\]
	For each $i=1,\dots,n$ the $i$-th row of the staircase block contains $L_i(w)$ left-adjusted crosses. For each $j=1,\dots, \ell(\mu)$ the horizontal part of $j$-th $\Gamma$-block contains $\mu_j(w)$ left-adjusted crosses. If $j$ is odd, then there is a faucet on the cross in the corner box.
	
	Moreover, each $d$-signed pipe dream $D\in\PD_{\Dc_n}(w)$ can be brought to $D_b$ by a sequence of admissible moves.
\end{theorem}

\begin{proof} The proof is similar to the proof of Theorem~\ref{bbottom}. 
	
	The bottom $d$-signed pipe dream $D_b$ constructed as in the statement of the theorem is reduced and has the shape $w\in\Dc_n$. Since the number of sign changes in  $w\in\Dc_n$ is even, the partition $\mu(w)$ uniquely determines $N(w)$ and each permutation $w\in\Dc_n$ is uniquely determined by the pair $(L(w),\mu(w))$. Thus, $D_b$ is the unique bottom d-signed pipe dream in $\PD_{D}(w)$.

Let us show how to reduce a $d$-signed pipe dream $D\in\PD_{\Dc_n}(w)$ to $D_b$. 

The cross from the lowest fragment $\elbows\cross$ can be moved down: to do this, we follow the two crossing strands down; since they are not the leftmost, these two strands cannot contain the fragment of the form $\crosssign\cross$. Then we proceed similarly to the proof of Theorem~\ref{bbottom}.

These operations can bring our pipe dream to the following form: the significant elements are only in the horizontal parts of the $\Gamma$-blocks; in each row they are left-adjusted. Further in the examples we draw the horizontal parts of the blocks one under another, ignoring the elbows between them.

If certain $\Gamma$-block is empty, we can move to it all the significant elements of the block above it. So we can suppose that all the empty blocks are located above the nonempty ones.

Now look at the elements in the boxes of the $\Gamma$-blocks indexed by $1'$. In the odd blocks we should have crosses with faucets $\crosssign$, and in the even ones crosses $\cross$. Suppose this is not true; take the lowest ($i$-th) block where this does not hold, and distinguish between the following possibilities:
\begin{itemize}
	\item The number $i$ is odd, and there is a cross $\cross$ in the corner of the $i$-th block. In the corner of the  $(i-1)$-st block there is also a cross $\cross$.  Since $D$ is reduced, the $(i-1)$-st block also contains at least one more cross. Then we can move the cross from the corner of the $i$-th block down similarly to Theorem~\ref{bbottom} (since both strands do not have faucets on them):
	$$
	\begingroup
	\setlength\arraycolsep{0pt}
	\renewcommand{\arraystretch}{0.07}
	\begin{matrix}
	&&\elbowup&\cross&\\
	&\elbowup&\cross&\cross\\
	\elbowup&\crosssign&\cross&\cross\\[4pt]
	&&\hdotsfor{2}&\\[4pt]
	&&\cross&\cross&\\
	&&\elbows&\elbow&\\
	\end{matrix}\mapsto
	\begin{matrix}
	&&\elbowup&\elbows&\\
	&\elbowup&\cross&\cross\\
	\elbowup&\crosssign&\cross&\cross\\[4pt]
	&&\hdotsfor{2}&\\[4pt]
	&&\cross&\cross&\\
	&&\cross&\elbow&\\
	\end{matrix}\endgroup
	$$
	\item The number  $i$ is even, and there is a cross with a faucet in the corner of the $i$-th block (denote this box by $a$). The corner of the $(i-1)$-st block also contains an element $\crosssign$. Since  $D$ is reduced, the  $(i-1)$-st block should contain at least one more cross. Similarly to the previous cases we can show that the strands crossing at $a$ also ``nearly meet'' at some other box $b$ containing an elbow $\elbows$. Then we can replace the cross with a faucet at $a$ by an elbow joint and replace the elbow at $b$ by a cross; this would preserve the shape of the permutation:
	$$
	\begingroup
	\setlength\arraycolsep{0pt}
	\renewcommand{\arraystretch}{0.07}
	\begin{matrix}
	&&\elbowup&\crosssign&\\
	&\elbowup&\crosssign&\cross\\
	\elbowup&\cross&\cross&\cross\\[4pt]
	&&\hdotsfor{2}&\\[4pt]
	&&\cross&\cross&\\
	&&\elbows&\elbow&\\
	\end{matrix}\mapsto
	\begin{matrix}
	&&\elbowup&\elbows&\\
	&\elbowup&\crosssign&\cross\\
	\elbowup&\cross&\cross&\cross\\[4pt]
	&&\hdotsfor{2}&\\[4pt]
	&&\cross&\cross&\\
	&&\cross&\elbow&\\
	\end{matrix}\endgroup
	$$
	\item The number  $i$ is odd, there is  an elbow with two faucets  $\elbowssign$ in the corner of the $i$-th block: 
	$$\begingroup
	\setlength\arraycolsep{0pt}
	\renewcommand{\arraystretch}{0.07}
	\begin{matrix}
	&&\elbowup&\elbowssign&\\
	&\elbowup&\cross&\cross\\
	\elbowup&\crosssign&\cross&\cross\\[4pt]
	&&\hdotsfor{2}&\\[4pt]
	&&\cross&\cross&\\
	&&\elbows&\elbow&\\
	\end{matrix}\mapsto
	\begin{matrix}
	&&\elbowup&\crosssign&\\
	&\elbowup&\cross&\cross\\
	\elbowup&\crosssign&\cross&\cross\\[4pt]
	&&\hdotsfor{2}&\\[4pt]
	&&\cross&\cross&\\
	&&\cross&\elbow&\\
	\end{matrix}\endgroup
	$$ 
		\item The number $i$ is even, there is  an elbow with two faucets  $\elbowssign$ in the corner of the $i$-th block: 
	$$\begingroup
	\setlength\arraycolsep{0pt}
	\renewcommand{\arraystretch}{0.07}
	\begin{matrix}
	&&\elbowup&\elbowssign&\\
	&\elbowup&\crosssign&\cross\\
	\elbowup&\cross&\cross&\cross\\[4pt]
	&&\hdotsfor{2}&\\[4pt]
	&&\cross&\cross&\\
	&&\elbows&\elbow&\\
	\end{matrix}\mapsto
	\begin{matrix}
	&&\elbowup&\cross&\\
	&\elbowup&\crosssign&\cross\\
	\elbowup&\cross&\cross&\cross\\[4pt]
	&&\hdotsfor{2}&\\[4pt]
	&&\cross&\cross&\\
	&&\cross&\elbow&\\
	\end{matrix}\endgroup
	$$
\end{itemize}

Now suppose that the numbers of significant elements in the $\Gamma$-blocks do not form a dercreasing sequence. Let these blocks contain $\mu_1,\mu_2,\ldots$ significant elements, and let  $i$ be the least number such that $\mu_{i+1}\geq \mu_i>0$. 

If $\mu_i=1$, the crosses  $\crosssign$ and $\cross$ from the two neighboring corners can be shifted into one elbow joint $\elbowssign$:
	$$\begingroup
\setlength\arraycolsep{0pt}
\renewcommand{\arraystretch}{0.07}
\begin{matrix}
\elbowup&\crosssign\\
\cross&\elbows\\
\end{matrix}\mapsto\begin{matrix}
\elbowup&\elbows\\
\elbowssign&\elbows\\
\end{matrix}\endgroup
$$
	$$\begingroup
\setlength\arraycolsep{0pt}
\renewcommand{\arraystretch}{0.07}
\begin{matrix}
\elbowup&\cross\\
\crosssign&\elbows\\
\end{matrix}\mapsto\begin{matrix}
\elbowup&\elbows\\
\elbowssign&\elbows\\
\end{matrix}\endgroup
$$

Otherwise consider the box $a$ in the $(i+1)$-st $\Gamma$-block at the  $\mu_i$-th position counted from the left. Similarly to Theorem~$\ref{bbottom}$, the cross at this box can be moved down:
$$\begingroup
\setlength\arraycolsep{0pt}
\renewcommand{\arraystretch}{0.07}
\begin{matrix}
&&\elbowup&\cross&\cross&\cross&\elbow\\
&\elbowup&\crosssign&\cross&\elbows&\elbow&\\[4pt]
&\hdotsfor{2}&&&&\\[4pt]
&\cross&\cross&&&&\\
&\elbows&\elbow&&&&\\
\end{matrix}\mapsto
\begin{matrix}
&&\elbowup&\cross&\elbows&\cross&\elbow\\
&\elbowup&\crosssign&\cross&\elbows&\elbow&\\[4pt]
&\hdotsfor{2}&&&&\\[4pt]
&\cross&\cross&&&&\\
&\cross&\elbow&&&&\\
\end{matrix}\endgroup
$$
$$\begingroup
\setlength\arraycolsep{0pt}
\renewcommand{\arraystretch}{0.07}
\begin{matrix}
&&&\elbowup&\crosssign&\cross&\cross&\elbow\\
&&\elbowup&\cross&\cross&\elbows&\elbow&\\
&\elbowup&\crosssign&\cross&\cross&\elbow&\\[4pt]
&&\hdotsfor{2}&&&&\\[4pt]
&&\cross&\cross&&&&\\
&&\elbows&\elbow&&&&\\
\end{matrix}\mapsto
\begin{matrix}
&&&\elbowup&\crosssign&\elbows&\cross&\elbow\\
&&\elbowup&\cross&\cross&\elbows&\elbow&\\
&\elbowup&\crosssign&\cross&\cross&\elbow&\\[4pt]
&&\hdotsfor{2}&&&&\\[4pt]
&&\cross&\cross&&&&\\
&&\cross&\elbow&&&&\\
\end{matrix}\endgroup
$$

Let us summarize the previous discussion. Consider a d-signed pipe dream $D\in\PD_{\Dc_n}(w)$. If it is not the bottom pipe dream, we can bring one of the significant elements down:
\begin{itemize}
	\item if $D$ has a fragment $\elbows\cross$, we can move the cross from the lowest of such fragments;
	\item if $D$ has no $\elbows\cross$,  but certain nonempty block is above an empty one, we can shift all the elements from this block;
	\item if  $D$ has no $\elbows\cross$, all the empty blocks are above the nonempty ones, but at some of the boxes $1'$ the alternation of crosses $\crosssign$ and $\cross$ is violated, we can move down the element from this box; 
	\item if $D$ has no $\elbows\cross$, all the empty blocks are above the nonempty ones, the crosses $\crosssign$ and $\cross$ at the boxes $1'$ are alternating, but the number of significant elements is not decreasing, we can move the lowest cross that violates this rule.
\end{itemize}
\end{proof}


The bottom pipe dream of $w\in\Fc_\infty$ corresponds to the lexicographically highest monomial of $\Ff_w$. This immediately implies the result by Billey and Haiman that Schubert polynomials form a basis (cf. \cite[Thm 3, Thm 4]{BilleyHaiman95}).

\begin{theorem}\label{thm:BCDbasis} Each family of Schubert polynomials $\{\Bf_w\mid w \in \BCc_\infty\}$, $\{\Cf_w\mid w \in \BCc_\infty\}$, $\{\Df_w\mid w \in \Dc_\infty\}$ forms a $\QQ$-basis in $\QQ[\zz, p_1(\xx),p_3(\xx),\ldots]$.
\end{theorem}
\begin{proof}
As before, we consider the order $\ldots>z_2>z_1>x_1>x_2>\ldots$ on the set of variables and the corresponding lexicographic order $\prec$ on the set of monomials. The highest (with respect to this order) monomial $\mathrm{in}_\prec(\Ff_w)$ in a Schubert polynomial $\Ff_w$, where $w\in\Fc_n\subset \Fc_\infty$, corresponds to the bottom pipe dream $D_b\in \PD_{\Fc_n}(w)$.
	
From the description of the bottom pipe dreams we see that:
	\begin{itemize}
		\item for different $w\in\Fc_\infty$ the highest terms of the Schubert polynomials $\Ff_w$ are distinct, hence the Schubert polynomials of each given type ($B$, $C$ or $D$) are linearly independent;
		\item for Schubert polynomials of a given type ($B$, $C$ or $D$) and each $\beta=(\beta_1,\beta_2,\ldots,0,0,\ldots)$ and $\alpha=(\alpha_1>\alpha_2>\ldots)$ there exists an element $w\in\Fc_\infty$ such that the highest term of $\Ff_w$ equals $\xx^\alpha\zz^\beta$.
	\end{itemize}

Now let us show that the highest term of each monomial $f\in\QQ[\zz,p_1(\xx),p_3(\xx),\ldots]$ satisfies the following property: the powers of $x_i$ occuring in it strictly decrease.

Let $f\in\QQ[p_1(\xx),p_3(\xx),\ldots]$, and let the highest term of $f$ be equal to $x_1^{\alpha_1}x_2^{\alpha_2}\ldots x_s^{\alpha s}$. Since  $f$ is a symmetric function, we have $\alpha_1\geq \alpha_2\geq\ldots\geq\alpha_s$. Suppose that $\alpha_1=\alpha_2=\alpha$.

Let $x_1^{\beta_1}x_2^{\beta_2}x_3^{\beta_3}\ldots$ be another monomial occuring in $f$. Then $f$ also contains the monomial $x_1^{\beta_2}x_2^{\beta_1}x_3^{\beta_3}\ldots$, so $\beta_1\leq\alpha,\beta_2\leq \alpha$ and we have either $\beta_1=\beta_2=\alpha$, or $\beta_1+\beta_2<\alpha_1+\alpha_2=2\alpha$. This means that
\[
f(x_1,x_2,x_3,\ldots)=x_1^\alpha x_2^\alpha g(x_3,x_4,\ldots)+h(x_1,x_2,x_3\ldots).
\]
Here  $g$  is a nonzero symmetric function, and  $h$ is a symmetric function with the following property:  for each of its monomials the sums of powers of $x_1$ and $x_2$ is less than $2\alpha$. The function $f$ is \emph{supersymmetric}:
\[
f(y,-y,x_1,x_2,\ldots)=f(x_1,x_2,\ldots),
\]
since every odd power sum $p_{2k-1}(\xx)$ satisfies this property. On the other hand, 
$$
f(y,-y,x_1,x_2,\ldots)=\pm y^{2\alpha}g(x_1,x_2,x_3,\ldots)+h(y,-y,x_1,\ldots),
$$
where the power of $h$ with respect to $y$ does not exceed $2\alpha-1$. So  $f(y,-y,x_1,x_2,\ldots)$ depends on $y$ and hence cannot be equal to $f(x_1,x_2,\ldots)$. This contradiction shows that $\alpha_1>\alpha_2$.

Now let us express $f(\xx)$ as
\[
f(\xx)=x^{\alpha_1}g(x_2,x_3,\ldots)+h(x_1,x_2,\ldots),
\]
where $g\in\QQ[p_1(x_2,x_3,\ldots),p_3(x_2,x_3,\ldots),\ldots]$, and the power of $h$ with respect to $x_1$ does not exceed $\alpha_1-1$. The highest term of  $g(x_2,x_3,\ldots)$ is equal to $x_2^{\alpha_2}x_3^{\alpha_3}\ldots x_s^{\alpha s}$. A similar reasoning applied to $g$ shows that $\alpha_2>\alpha_3$. Proceeding in this way further, we obtain the desired inequalities $\alpha_1>\alpha_2>\ldots>\alpha_s$. 
\end{proof}

\begin{remark} Note that the Schubert polynomials, viewed as elements of the ring $\QQ[\zz, p_1(\xx),p_3(\xx),\ldots]$, do not necessarily have integer coefficients. For instance, if $w=s_0s_1s_0$, the corresponding Schubert polynomial of type $B$ equals
\[
\Bf_{s_0s_1s_0}=\sum_{i<j}(x_i^2x_j+x_ix_j^2)+2\sum_{i<j<k}x_ix_jx_k=\frac 1 3 p_1^3(\xx)-\frac 1 3 p_3(\xx).
\]
More examples can be found in Appendix~\ref{appendix}.

However, if we consider them as elements of the ring $\ZZ[\zz]\otimes \Omega[\xx]$, where $\Omega[\xx]$ is the ring of supersymmetric functions in $\xx$, then the Schubert polynomials of types $B$ and $D$ form bases of this ring. The argument for proving this is similar to the one in type $A$: every monomial $\xx^\alpha\zz^\beta$ with $\alpha_1>\alpha_2>\dots$ occurs as the highest term in some Schubert polynomial with coefficient 1.
\end{remark}

\section{Schubert polynomials of Grassmannian permutations}\label{sec:grassmann}

\subsection{Schubert polynomials of Grassmannian permutations are Schur polynomials}\label{ssec:grassmannschur}

It  is well known that the Schur polynomials appear as the Schubert polynomials of permutations with a unique descent (such permutations are called Grassmannian ones). This fact can be proved combinatorially (cf., for instance, the notes~\cite{Knutson12} by A.\,Knutson). The Schur polynomial $s_\lambda(z_1,\dots,z_n)$, where $\lambda$ is a partition, can be obtained as the sum of monomials $z^T$ over Young tableaux $T$ of shape $\lambda$ filled by the integers not exceeding  $n$; one can construct a bijection between the Young tableaux of a given shape and the pipe dreams of the Grassmannian permutation corresponding to $\lambda$. For the reader's convenience, we recall this result as Theorem~\ref{thm:a_grass} and provide its proof.

It is well-known that this result can be generalized to the case of other classical groups. In this case the ordinary Schur polynomials are replaced by the $P$- and $Q$-Schur polynomials: $P$ for the types $B$ and $D$, while $Q$ corresponds to the type $C$. This observation is essentialy due to P.\,Pragacz~\cite{Pragacz91}; in~\cite[Theorem 3]{BilleyHaiman95}, this fact was shown using a modification of the Edelmann--Greene correspondence. The main goal of this section is to give a proof of this fact without referring to the Edelmann--Greene correspondence; instead, we mimic the type $A$ proof and construct the bijection between the pipe dreams of a Grassmannian permutation and circled shifted Young tableaux, which index monomials in $P$- and $Q$-Schur functions.

\begin{definition} Let $2\leq k\leq n-1$. A permutation $w\in\Sc_n$ is said to be \emph{$k$-Grassmannian},  if $w(1)<w(2)<\ldots<w(k)>w(k+1)<w(k+2)<\ldots<w(n)$. A permutation is said to be Grassmannian, if it is $k$-Grassmannian for some $k$. In other words, a Grassmannian permutation has a unique descent at $k$.
\end{definition}
	
A $k$-Grassmannian permutation $w\in\Sc_n$ bijectively corresponds to the partition $\lambda(w)=(\lambda_1,\dots,\lambda_k)$ of the number $\ell(w)$ into at most $k$ parts: it is given by $\lambda_i=w(i)-i$.

\begin{definition} Let $\Fc_n=\BCc_n$ or $\Dc_n$. A permutation $w\in\Fc_n$ is said to be  \emph{Grassmannian} if $w(1)<w(2)<\ldots<w(n)$.
\end{definition}

Let $k$ be a number such that $w(k)<0<w(k+1)$.  A permutation thus defines a \emph{strict} partition, i.e. a partition of strictly decreasing integers, $-w(1)>-w(2)>\dots>-w(k)>0$. This provides a bijection between strict partitions and Grassmannian permutations in $\BCc_\infty$.

In the one-line notation, saying that a permutation $w\in\BCc_\infty$ is Grassmannian is equivalent to saying that $w=\overline\imath_1\dots \overline\imath_k i_{k+1}\dots i_n$, with all $i_j>0$ distinct, $i_1>i_2>\dots>i_k$, and $i_{k+1}<i_{k+2}<\dots<i_n$: all the entries with bars precede those without bars, the absolute values of the barred entries decrease, and the values of the entries without bars increase. The corresponding strict partition is $(i_1,i_2,\ldots,i_k)$.

In this section we prove that the Schubert polynomials of Grassmannian permutations for $\Fc_n=\BCc_n$ and $\Dc_n$ are equal to  $P$- (in types $B$ and $D$) and $Q$-Schur functions (in type $C$), respectively, by constructing an explicit bijection between the pipe dreams and the circled Young tableaux.

We start with recalling the combinatorial description of (usual) Schur polynomials.

\begin{definition}
Let $\lambda=(\lambda_1\geq\lambda_2\geq\ldots\geq\lambda_l>0)$ be a partition  and $k\geq 0$. Consider the \emph{Young diagram} (in the English notation) with rows of length $\lambda_1,\dots,\lambda_l$. A \emph{(semistandard) Young tableau of shape $\lambda$} is a labelling of its boxes by the numbers $1,\dots,k$ weakly increasing along the rows and strictly increasing along the columns. The set of all semistandard Young tableaux of shape $\lambda$ will be denoted by $\SSYT_k(\lambda)$.
	
If $T\in\SSYT_k(\lambda)$, denote by $\zz^T$ the monomial obtained as the product of  $z_i$ over all occurencies of $i$ in the tableau $T$. A \emph{Schur polynomial} is defined as the sum of monomials $\zz^T$ over all semistandard Young tableaux of a given shape:
\[
s_{\lambda}(z_1, \ldots, z_k) = \sum\limits_{T \in \SSYT_k(\lambda)}\zz^T.
\]
\end{definition}

It is well-known (cf., for example,~\cite{Fulton97}) that $s_\lambda$ is a symmetric polynomial in $z_1,\dots,z_k$. 

The following theorem is a fundamental property of Schubert polynomials (see, for instance,~\cite[{(4.8)}]{Macdonald91}). Its bijective proof is probably folklore; it can be found, for example, in~\cite{Knutson12}. For the reader's convenience, we give this proof here. It will be generalized for the cases $B$, $C$, and $D$ in the next subsection.

\begin{theorem}\label{thm:a_grass}
	Let $w\in\Sc_n$ be a $k$-Grassmannian permutation. Let
\[
\lambda(w)=(\lambda_1(w),\ldots,\lambda_l(w))=(w(k)-k, \ldots, w(2)-2, w(1)-1).
\]
Then we have
\[
\Sf_w(\zz)=s_{\lambda(w)}(z_1,z_2,\dots,z_k).
\]
\end{theorem}

\begin{proof}
We first observe that  $L_i(w)=w(i)-i$ for  $i\leq k$ and $L_i(w)=0$ for $i>k$. This means that for the bottom pipe dream $D_b\in\PD_{A}(w)$ its crosses are located in the boxes forming the Young diagram $\lambda$ flipped upside down:
		\begin{figure}[ht]
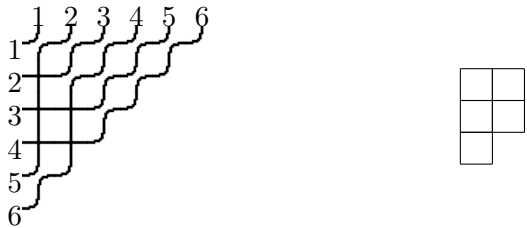

		\centering
		\begin{subfigure}{.3\textwidth}
			\centering
			$$	\begingroup
			\setlength\arraycolsep{0pt}
			\renewcommand{\arraystretch}{0.07}
			\begin{matrix}
			& 1 & 2 & 3 & 4 & 5 & 6  \\
			1 & \elbows & \elbows & \elbows & \elbows & \elbows & \elbow & \\
			2 & \cross & \elbows & \elbows & \elbows & \elbow & \\
			3 & \cross & \cross & \elbows & \elbow & \\
			4 & \cross & \cross & \elbow & \\
			5 & \elbows & \elbow & \\
			6 & \elbow & \\
			\end{matrix}
			\quad
			\endgroup
			$$
		\end{subfigure}%
		\begin{subfigure}{.3\textwidth}
			\centering
			\begin{displaymath}
			\tableau{ \ & \  \\
				\ & \  \\
				\  } 
			\end{displaymath} 			
		\end{subfigure}
		\caption{The bottom pipe dream for  $w=135624 $ and the Young diagram $\lambda(w)$}
	\end{figure}

Now let  $D\in\PD_{A}(w)$ be an arbitrary pipe dream of shape $w$. Let us assign to $D$ a Young  tableau $T\in\SSYT_k(\lambda(w))$ as follows. First let us bring it to $D_b$ by moving some crosses down. Now we replace each cross in $D_b$ by the number $k-i+1$, where $i$ is the number of the row which initially contained this cross.

One can check that all the crosses were only shifted southwest:
\[
\begingroup
\setlength\arraycolsep{0pt}
\renewcommand{\arraystretch}{0.07}
\begin{matrix}
\elbows&\cross\\
\elbows&\elbows
\end{matrix}\mapsto
\begin{matrix}
\elbows&\elbows\\
\cross&\elbows
\end{matrix}
\endgroup
\]
This means that the order of crosses along each strand remains unchanged. Hence the numbers along the rows of $T$ weakly increase, and the numbers along the columns strictly increase.
	\begin{figure}[ht!]
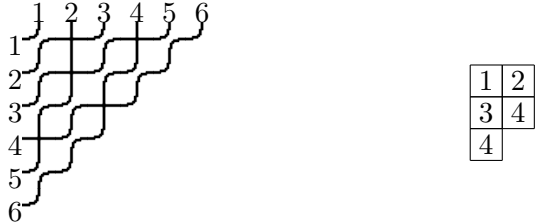

	\centering
	\begin{subfigure}{0.3\textwidth}
		\centering
		$$	\begingroup
		\setlength\arraycolsep{0pt}
		\renewcommand{\arraystretch}{0.07}
		\begin{matrix}
		& 1 & 2 & 3 & 4 & 5 & 6  \\
		1 & \elbows & \cross & \elbows & \cross & \elbows & \elbow & \\
		2 & \elbows & \cross & \elbows & \elbows & \elbow & \\
		3 & \elbows & \elbows & \cross & \elbow & \\
		4 & \cross & \elbows & \elbow & \\
		5 & \elbows & \elbow & \\
		6 & \elbow & \\
		\end{matrix}
		\quad
		\endgroup
		$$
	\end{subfigure}
	\begin{subfigure}{.3\textwidth}
		\centering
		\begin{displaymath}
		\tableau{ 1 & 2  \\
			3 & 4  \\
			4  } 
		\end{displaymath} 		
	\end{subfigure}
	\caption{Pipe dream and the corresponding Young tableau}
\end{figure}

This correspondence between $\PD_{A}(w)$ and $\SSYT_k(\lambda(w))$ is a bijection. Indeed, let us start with a tableau $T$; we will read the crosses in $D_b$ from right to left, from top to bottom, and lift each of the crosses into the row prescribed by the value inside the corresponding box. Since the numbers in $T$ weakly increase along the rows and strictly increase along the columns, each of the crosses will not ``bump'' into the crosses lifted before, so all of them will be lifted to the prescribed positions.

Let $\widetilde\zz=(z_k,z_{k-1},\ldots,z_1)$. If  $T\in\SSYT_k(\lambda(w))$ is the tableau corresponding to  $D\in\PD_{A}(w)$, the monomial $\zz^{\beta(D)}$ is equal to $\widetilde\zz^T$.

Hence
\[
s_\lambda(\widetilde\zz)=\sum_{T\in\SSYT(\lambda(w))}\widetilde\zz^T=\sum_{D \in \PD_{A}(w)}\zz^{\beta(D)}=\Sf_w(\zz).
\]

Since the Schur polynomials are symmetric, we have $s_\lambda(\zz)=\Sf_w(z_1,\dots,z_k)$.
\end{proof}

\subsection{Schubert polynomials of types $B$, $C$ and $D$ are $P$- and $Q$-Schur functions}\label{ssec:grassmannpq}

Let $\mu=(\mu_1,\dots,\mu_\ell)$, where $\mu_1>\mu_2>\ldots>\mu_\ell$ be a partition into distinct parts. Consider a \emph{shifted Young diagram}, as shown on the figure below: its $i$-th row consists of $\mu_i$ boxes, and the rows are aligned along the NW-SE diagonal. The partition $\mu=(\mu_1,\mu_2,\ldots,\mu_\ell)$ is called the \emph{shape} of this diagram.
	\begin{figure}[ht!]
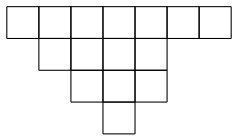

		$$
		\tableau{\ & \ & \ & \ & \ & \ & \ \\
		& \ & \ & \ & \ &&\\
		&& \ & \ & \ & &\\
		&&& \ &&&} 
		$$
		\caption{Shifted Young diagram $\mu=(7,4,3,1)$}
	\end{figure}

Let us fill the boxes of a shifted Young diagram by the numbers $1^\circ<1<2^\circ<2<\ldots$ according to the following rules:
\begin{itemize}
	\item the numbers weakly increase along the rows and the columns;
	\item a row cannot contain two equal numbers with circles;
	\item a column cannot contain two equal numbers without circles.
\end{itemize}
Such an object is called a \emph{circled Young tableau} of shape $\mu$. The set of all such tableaux of shape $\mu$ will be denoted by $\CYT(\mu)$. 

Let $\CYT'(\mu)\subset\CYT(\mu)$ be the subset of circled Young tableaux such that the leftmost number of each row has no circle.

To each tableau $T\in\CYT(\mu)$ we can assign a monomial $\xx^T$ obtained as the product of $x_i$ over all the occurencies of $i^\circ$ and $i$ in $T$.

\begin{definition} The \emph{$P$- and $Q$-Schur functions} are defined as the sums of monomials over all circled Young tableaux:
\begin{eqnarray*}
Q_\mu(\xx)=\sum_{T\in\CYT(\mu)}\xx^T;\\
P_\mu(\xx)=\sum_{T\in\CYT'(\mu)}\xx^T.
\end{eqnarray*}
\end{definition}

\begin{remark}
Let $T\in\CYT(\mu)$ be a circled Young tableau. Note that we can add or remove a circle at the number of the leftmost box of each row, still obtaining a valid circled tableau. This means that
\[
P_\mu(\xx)=2^{-\ell(\mu)}Q_\mu(\xx),
\]
where  $\ell(\mu)$ is the number of parts in the partition  $\mu$.
\end{remark}

The following theorem was proved in~\cite[Prop.~3.13, 3.14]{BilleyHaiman95} using a modification of the Edelmann--Greene correspondence. Here we prove it with a different (and easier) bijective argument.

\begin{theorem}\label{PQschub}
Let $w=\overline{w}_1\overline{w}_2\ldots\overline{w}_{s(w)} w_{s(w)+1}\ldots w_{n}\in \BCc_n$ be a Grassmannian permutation with the decreasing segment $w_1,\ldots,w_{s(w)}$ and  the increasing segment $w_{s(w)+1},\ldots,w_n$. Let $\mu=(w_1,\ldots,w_{s(w)})$. Then $\Bf_w=P_{\mu}$ and $\Cf_w=Q_{\mu}$.
\end{theorem}

\begin{proof} Consider the bottom pipe dream $D_b\in\PD_{\Bc_n}(w)$. The positions of its crosses can be deduced from Theorem~\ref{bbottom}: since $-w_1<-w_2<\ldots<-w_{s(w)}<w_{s(w)+1}<\ldots<w_n$, its Lehmer code equals $L(w)=(0,0,\ldots,0)$, hence the staircase block of $D_b$ does not contain any crosses. On the other hand, the  $i$-th $\Gamma$-block for $i\leq s(w)$ contains $w_i$ left-adjusted significant elements.
	
Let us establish a correspondence between the significant elements of $D_b$ and the boxes in the shifted Young diagram in a natural way: the $j$-th significant element (counted from the left) of the $i$-th $\Gamma$-block corresponds to the $j$-th box of the $i$-th row in the shifted Young diagram. If two boxes are located in the same column of the Young diagram, there is a strand passing vertically through both corresponding significant elements. Meanwhile, two boxes in the same row of the Young diagram correspond to a horizontal segment of a strand.

Since each $D\in\PD_{\Bc_n}(w)$ can be reduced to $D_b$ by shifting down significant elements, the staircase block of $D$ also does not contain crosses. We assign to $D$ a circled Young tableau $T\in\CYT'(\mu)$ as follows. For each significant element from  $D_b$ let us look at its initial position in  $D$. If it was in the vertical part (excluding the corner) of the $i$-th $\Gamma$-block, we put $i^\circ$ into the corresponding box of our tableau.  If it was in the horizontal part (including the corner), we put~$i$.

One can show that for Grassmannian permutations the significant elements will only move along the NE-SW diagonal:
	$$\begingroup
	\setlength\arraycolsep{0pt}
	\renewcommand{\arraystretch}{0.07}
	\begin{matrix}
&&\elbowup&\cross\\
&\elbowup&\elbows&\elbow\\
\elbowup&\elbows&\elbow&
	\end{matrix}\mapsto\begin{matrix}
&&\elbowup&\elbows\\
&\elbowup&\elbows&\elbow\\
\elbowup&\cross&\elbow&
	\end{matrix}\endgroup
	$$
		$$\begingroup
	\setlength\arraycolsep{0pt}
	\renewcommand{\arraystretch}{0.07}
	\begin{matrix}
	&&\elbowsign&\elbows\\
	&\elbowup&\elbows&\elbow\\
	\elbowup&\elbows&\elbow&
	\end{matrix}\mapsto\begin{matrix}
	&&\elbowup&\elbows\\
	&\elbowup&\elbows&\elbow\\
	\elbowsign&\elbows&\elbow&
	\end{matrix}\endgroup,
	$$
so the order of the elements corresponding to each strand will remain unchanged. This implies the non-decreasing of the numbers in $T$ along the rows and the columns. Also, if a strand passes through two significant elements horizontally (resp. vertically), these elements cannot be located in the same vertical (resp. horizontal) part of the same $\Gamma$-block. This implies that the circled integers are not repeating along the rows of $T$, and the non-circled ones are not repeating along the columns. 

This correspondence between $\PD_{\Bc_n}(w)$ and $\CYT'(\mu)$ is a bijection. Indeed,  consider a tableau  $T\in\CYT'(\mu)$. Let us read the significant elements of $D_b$ one by one, right to left top to bottom, and lift each of them into the vertical or horizontal part of the $\Gamma$-block prescribed by the entry in the corresponding box. It is easy to see that the elements will not ``bump'' into each other, so all of them will be lifted to the prescribed positions.

To conclude, note that if a $b$-signed pipe dream $D$ corresponds to the tableau $T$, then $\xx^T=\xx^{\alpha(D)}=\xx^{\alpha(D)}\zz^{\beta(D)}$. This means that
\[
\Bf_w(\xx,\zz)=\sum_{D\in\PD_{\Bc_n}(w)}\xx^{\alpha(D)}\zz^{\beta(D)}=\sum_{T\in\CYT'(\mu)}\xx^T=P_\mu(\xx).
\]

In the case of $c$-signed pipe dreams a similar reasoning gives us the following relation:
\[
\Cf_w(\xx,\zz)=\sum_{D\in\PD_{\Cc_n}(w)}\xx^{\alpha(D)}\zz^{\beta(D)}=\sum_{T\in\CYT(\mu)}\xx^T=Q_\mu(\xx).
\]
\end{proof}

We conclude by a similar statement in type $D$.

\begin{theorem}\label{thm:dgrassmann}
Let $w=\overline{w}_1\overline{w}_2\ldots\overline{w}_{s(w)} w_{s(w)+1}\ldots w_{n}\in \Dc_n$ be a Grassmannian permutation, with the decreasing segment $w_1,\ldots,w_{s(w)}$ and the increasing segment $w_{s(w)+1},\ldots,w_n$. Let $\mu'=(w_1-1,w_2-1,\ldots,w_{s(w)}-1)$ (if the last entry of the sequence is zero, we omit it). Then $\Df_w=P_{\mu'}$.
\end{theorem}
\begin{proof}
Similar to the proof of Theorem~\ref{PQschub}.
\end{proof}

	\appendix
	
	\section{More examples}\label{appendix}
	
	In the appendix we present an alternative, more compact, notation for pipe dreams. Then we list all the elements, except the identity, for the group $\BCc_2$, as well as some elements of the group $\Dc_3$. For each element we list all the signed pipe dreams of types $B$ and $D$ respectively, with the corresponding monomials. 
	

While working with usual pipe dreams, it is convenient to omit elbows and to draw only crosses. Thus, a pipe dream for a permutation $w\in S_n$ is represented by a staircase Young diagram $(n-1,\dots,2,1)$ with some boxes filled by crosses. 

For the pipe dreams of other classical types, we will do the same: a pipe dream of type $B_n$ will be represented by a sequence of Young diagrams consisting of a staircase $(n-1,\dots,2,1)$ and several hooks $(n,1,\dots,1)$, with some boxes filled by crosses and eventually by the sign $\circ$. The latter sign can only be located in the upper-left corner of a hook; it represents a faucet. Type $C$ is treated similarly.

Likewise, for type $D$  the upper-left corner of each hook can be either empty or contain one of the signs: $+$, $\oplus$, or $\%$, representing a cross, a cross with a faucet, or an elbow with two faucets, respectively.

In the tables below we omit the empty $\Gamma$-blocks. The total weight of all such pipe dreams is a polynomial in $z_i$ times a quasisymmetric function in $x_i$, given in the right column. The first (staircase) Young diagram is shaded gray.
	

\ytableausetup{smalltableaux}

\begin{longtable}{|l|l|l|}
\caption{B-signed pipe dreams for $w\in \BCc_2$}\\
\hline
$w$ & Pipe dreams & Weight\\ 
\hline
$\overline 1 2=s_0$ & \ytableaushort{}*[*(white!25!gray)]{1} \quad \ytableaushort{\circ}*{2,1}& $x_i$\\
\hline
\multicolumn{3}{l}{$\Bf_{s_0}=\sum_{i} x_i=p_1(\xx)$}\\
\hline
$21=s_1$ & \ytableaushort{+}*[*(white!25!gray)]{1} \quad \ytableaushort{\circ}*{2,1}& $z_1$\\ 
& \ytableaushort{}*[*(white!25!gray)]{1} \quad \ytableaushort{\none,+}*{2,1}& $x_i$\\ 
& \ytableaushort{}*[*(white!25!gray)]{1} \quad \ytableaushort{\none +}*{2,1}& $x_i$\\
\hline
\multicolumn{3}{l}{$\Bf_{s_1}=z_1+2\sum_{i} x_i=z_1+2p_1(\xx)$}\\
\hline
$\overline{2}1=s_1s_0$ & \ytableaushort{}*[*(white!25!gray)]{1} \quad \ytableaushort{\circ+}*{2,1}& $x_i^2$\\ 
& \ytableaushort{}*[*(white!25!gray)]{1} \quad \ytableaushort{\circ}*{2,1}
\quad \ytableaushort{\none,+}*{2,1}& $x_ix_j$\\ 
& \ytableaushort{}*[*(white!25!gray)]{1} \quad \ytableaushort{\circ}*{2,1}\quad \ytableaushort{\none+}*{2,1}& $x_ix_j$\\
\hline
\multicolumn{3}{l}{$\Bf_{s_1s_0}=\sum_i x_i^2+2\sum_{i<j} x_ix_j=p_1^2(\xx)$}\\
\hline
$2\overline{1}=s_0s_1$ & \ytableaushort{+}*[*(white!25!gray)]{1} \quad \ytableaushort{\circ}*{2,1}& $z_1x_i$\\ 
& \ytableaushort{}*[*(white!25!gray)]{1} \quad \ytableaushort{\circ,+}*{2,1}
& $x_i^2$\\ 
& \ytableaushort{}*[*(white!25!gray)]{1} \quad \ytableaushort{\none,+}*{2,1}\quad \ytableaushort{\circ}*{2,1}& $x_ix_j$\\
& \ytableaushort{}*[*(white!25!gray)]{1} \quad \ytableaushort{\none +}*{2,1}\quad \ytableaushort{\circ}*{2,1}& $x_ix_j$\\
\hline
\multicolumn{3}{l}{$\Bf_{s_0s_1}=z_1\sum_{i}x_i+\sum_i x_i^2+2\sum_{i<j}x_ix_j=z_1p_1(\xx)+p_1^2(\xx)$}\\
\hline
$\overline{21}=s_0s_1s_0$ & \ytableaushort{}*[*(white!25!gray)]{1} \quad \ytableaushort{\circ+}*{2,1}\quad \ytableaushort{\circ}*{2,1}& $x_i^2x_j$\\
& \ytableaushort{}*[*(white!25!gray)]{1} \quad \ytableaushort{\circ}*{2,1}\quad \ytableaushort{\circ,+}*{2,1}& $x_ix_j^2$\\
& \ytableaushort{}*[*(white!25!gray)]{1} \quad \ytableaushort{\circ}*{2,1}\quad \ytableaushort{\none,+}*{2,1}\quad \ytableaushort{\circ}*{2,1} &$x_ix_jx_k$\\
& \ytableaushort{}*[*(white!25!gray)]{1} \quad \ytableaushort{\circ}*{2,1}\quad \ytableaushort{\none+}*{2,1}\quad \ytableaushort{\circ}*{2,1} &$x_ix_jx_k$\\
\hline
\multicolumn{3}{l}{$\Bf_{s_0s_1s_0}=\sum_{i<j}(x_i^2x_j+x_ix_j^2)+2\sum_{i<j<k}x_ix_jx_k=\frac 1 3 p_1^3(\xx)-\frac 1 3 p_3(\xx)$}\\
\hline
$1\overline{2}=s_1s_0s_1$ & \ytableaushort{+}*[*(white!25!gray)]{1} \quad \ytableaushort{\circ+}*{2,1} & $z_1x_i^2$\\
&\ytableaushort{+}*[*(white!25!gray)]{1} \quad \ytableaushort{\circ}*{2,1}\quad
\ytableaushort{\none,+}*{2,1} & $z_1x_ix_j$\\
&\ytableaushort{+}*[*(white!25!gray)]{1} \quad \ytableaushort{\circ}*{2,1}\quad
\ytableaushort{\none+}*{2,1} & $z_1x_ix_j$\\
& \ytableaushort{}*[*(white!25!gray)]{1} \quad \ytableaushort{\circ+,+}*{2,1} & $x_i^3$\\
&\ytableaushort{}*[*(white!25!gray)]{1} \quad \ytableaushort{\circ,+}*{2,1}\quad
\ytableaushort{\none+}*{2,1} & $x_i^2x_j$\\
&\ytableaushort{}*[*(white!25!gray)]{1} \quad \ytableaushort{\circ,+}*{2,1}\quad
\ytableaushort{\none,+}*{2,1} & $x_i^2x_j$\\
&\ytableaushort{}*[*(white!25!gray)]{1} \quad \ytableaushort{\none,+}*{2,1}\quad
\ytableaushort{\circ+}*{2,1} & $x_ix_j^2$\\
&\ytableaushort{}*[*(white!25!gray)]{1} \quad \ytableaushort{\none,+}*{2,1}\quad
\ytableaushort{\circ}*{2,1}\quad \ytableaushort{\none,+}*{2,1} & $x_ix_jx_k$\\
&\ytableaushort{}*[*(white!25!gray)]{1} \quad \ytableaushort{\none,+}*{2,1}\quad
\ytableaushort{\circ}*{2,1}\quad \ytableaushort{\none+}*{2,1} & $x_ix_jx_k$\\
&\ytableaushort{}*[*(white!25!gray)]{1} \quad \ytableaushort{\none+}*{2,1}\quad
\ytableaushort{\circ+}*{2,1} & $x_ix_j^2$\\
&\ytableaushort{}*[*(white!25!gray)]{1} \quad \ytableaushort{\none+}*{2,1}\quad
\ytableaushort{\circ}*{2,1}\quad \ytableaushort{\none,+}*{2,1} & $x_ix_jx_k$\\
&\ytableaushort{}*[*(white!25!gray)]{1} \quad \ytableaushort{\none+}*{2,1}\quad
\ytableaushort{\circ}*{2,1}\quad \ytableaushort{\none+}*{2,1} & $x_ix_jx_k$\\
\hline
\multicolumn{3}{l}{$\Bf_{s_1s_0s_1}=z_1\sum_i x_i^2 + 2z_1\sum_{i<j}x_ix_j+2\sum_{i<j}(x_i^2x_j+x_ix_j^2)+$}\\
\multicolumn{3}{r}{$+4\sum_{i<j<k}x_ix_jx_k=z_1p_1^2(\xx)+\frac 2 3 p_1^3(\xx)+\frac 1 3 p_3(\xx)$}\\
\hline
$\overline{12}=s_0s_1s_0s_1$ &\ytableaushort{}*[*(white!25!gray)]{1} \quad \ytableaushort{\circ+}*{2,1}\quad
\ytableaushort{\circ+}*{2,1} & $x_i^2x_j^2$\\
$=s_1s_0s_1s_0$ &\ytableaushort{}*[*(white!25!gray)]{1} \quad \ytableaushort{\circ+}*{2,1}\quad
\ytableaushort{\circ}*{2,1} \quad \ytableaushort{\none,+}*{2,1} &$x_i^2x_jx_k$\\
&\ytableaushort{}*[*(white!25!gray)]{1} \quad \ytableaushort{\circ+}*{2,1}\quad
\ytableaushort{\circ}*{2,1} \quad \ytableaushort{\none+}*{2,1} &$x_i^2x_jx_k$\\
&\ytableaushort{}*[*(white!25!gray)]{1} \quad \ytableaushort{\circ}*{2,1}\quad
\ytableaushort{\circ+,+}*{2,1}  &$x_i^2x_j^3$\\
&\ytableaushort{}*[*(white!25!gray)]{1} \quad \ytableaushort{\circ}*{2,1}\quad
\ytableaushort{\circ,+}*{2,1} \quad \ytableaushort{\none,+}*{2,1} &$x_ix_j^2x_k$\\
&\ytableaushort{}*[*(white!25!gray)]{1} \quad \ytableaushort{\circ}*{2,1}\quad
\ytableaushort{\circ,+}*{2,1} \quad \ytableaushort{\none+}*{2,1} &$x_ix_j^2x_k$\\
&\ytableaushort{}*[*(white!25!gray)]{1} \quad \ytableaushort{\circ}*{2,1}\quad
\ytableaushort{\none,+}*{2,1} \quad \ytableaushort{\circ+}*{2,1} &$x_ix_jx_k^2$\\
&\ytableaushort{}*[*(white!25!gray)]{1} \quad \ytableaushort{\circ}*{2,1}\quad
\ytableaushort{\none,+}*{2,1} \quad \ytableaushort{\circ}*{2,1}\quad \ytableaushort{\none,+}*{2,1} &$x_ix_jx_kx_m$\\
&\ytableaushort{}*[*(white!25!gray)]{1} \quad \ytableaushort{\circ}*{2,1}\quad
\ytableaushort{\none,+}*{2,1} \quad \ytableaushort{\circ}*{2,1}\quad \ytableaushort{\none+}*{2,1} &$x_ix_jx_kx_m$\\
&\ytableaushort{}*[*(white!25!gray)]{1} \quad \ytableaushort{\circ}*{2,1}\quad
\ytableaushort{\none+}*{2,1} \quad \ytableaushort{\circ+}*{2,1} &$x_ix_jx_k^2$\\
&\ytableaushort{}*[*(white!25!gray)]{1} \quad \ytableaushort{\circ}*{2,1}\quad
\ytableaushort{\none+}*{2,1} \quad \ytableaushort{\circ}*{2,1}\quad \ytableaushort{\none,+}*{2,1} &$x_ix_jx_kx_m$\\
&\ytableaushort{}*[*(white!25!gray)]{1} \quad \ytableaushort{\circ}*{2,1}\quad
\ytableaushort{\none+}*{2,1} \quad \ytableaushort{\circ}*{2,1}\quad \ytableaushort{\none+}*{2,1} &$x_ix_jx_kx_m$\\
 &\ytableaushort{+}*[*(white!25!gray)]{1} \quad \ytableaushort{\circ+}*{2,1}\quad
\ytableaushort{\circ}*{2,1} & $z_1x_i^2x_j$\\
 &\ytableaushort{+}*[*(white!25!gray)]{1} \quad \ytableaushort{\circ}*{2,1}\quad
\ytableaushort{\circ,+}*{2,1} & $z_1x_ix_j^2$\\
&\ytableaushort{+}*[*(white!25!gray)]{1} \quad \ytableaushort{\circ}*{2,1}\quad
\ytableaushort{\none,+}*{2,1} \quad \ytableaushort{\circ}*{2,1} &$z_1x_ix_jx_k$\\
&\ytableaushort{+}*[*(white!25!gray)]{1} \quad \ytableaushort{\circ}*{2,1}\quad
\ytableaushort{\none+}*{2,1} \quad \ytableaushort{\circ}*{2,1} &$z_1x_ix_jx_k$\\
&\ytableaushort{}*[*(white!25!gray)]{1} \quad \ytableaushort{\circ+,+}*{2,1}\quad
\ytableaushort{\circ}*{2,1} & $x_i^3x_j$\\
&\ytableaushort{}*[*(white!25!gray)]{1} \quad \ytableaushort{\circ,+}*{2,1}\quad
\ytableaushort{\circ,+}*{2,1} & $x_i^2x_j^2$\\
&\ytableaushort{}*[*(white!25!gray)]{1} \quad \ytableaushort{\circ,+}*{2,1}\quad
\ytableaushort{\none,+}*{2,1} \quad \ytableaushort{\circ}*{2,1} &$x_i^2x_jx_k$\\
&\ytableaushort{}*[*(white!25!gray)]{1} \quad \ytableaushort{\circ,+}*{2,1}\quad
\ytableaushort{\none+}*{2,1} \quad \ytableaushort{\circ}*{2,1} &$x_i^2x_jx_k$\\
&\ytableaushort{}*[*(white!25!gray)]{1} \quad \ytableaushort{\none,+}*{2,1}\quad
\ytableaushort{\circ+}*{2,1} \quad \ytableaushort{\circ}*{2,1} &$x_ix_j^2x_k$\\
&\ytableaushort{}*[*(white!25!gray)]{1} \quad \ytableaushort{\none,+}*{2,1}\quad
\ytableaushort{\circ}*{2,1} \quad \ytableaushort{\circ,+}*{2,1} &$x_ix_jx_k^2$\\
&\ytableaushort{}*[*(white!25!gray)]{1} \quad \ytableaushort{\none,+}*{2,1}\quad
\ytableaushort{\circ}*{2,1} \quad \ytableaushort{\none,+}*{2,1}\quad \ytableaushort{\circ}*{2,1} &$x_ix_jx_kx_m$\\
&\ytableaushort{}*[*(white!25!gray)]{1} \quad \ytableaushort{\none,+}*{2,1}\quad
\ytableaushort{\circ}*{2,1} \quad \ytableaushort{\none+}*{2,1}\quad \ytableaushort{\circ}*{2,1} &$x_ix_jx_kx_m$\\
&\ytableaushort{}*[*(white!25!gray)]{1} \quad \ytableaushort{\none+}*{2,1}\quad
\ytableaushort{\circ+}*{2,1} \quad \ytableaushort{\circ}*{2,1} &$x_ix_j^2x_k$\\
&\ytableaushort{}*[*(white!25!gray)]{1} \quad \ytableaushort{\none+}*{2,1}\quad
\ytableaushort{\circ}*{2,1} \quad \ytableaushort{\circ,+}*{2,1} &$x_ix_jx_k^2$\\
&\ytableaushort{}*[*(white!25!gray)]{1} \quad \ytableaushort{\none+}*{2,1}\quad
\ytableaushort{\circ}*{2,1} \quad \ytableaushort{\none,+}*{2,1}\quad \ytableaushort{\circ}*{2,1} &$x_ix_jx_kx_m$\\
&\ytableaushort{}*[*(white!25!gray)]{1} \quad \ytableaushort{\none+}*{2,1}\quad
\ytableaushort{\circ}*{2,1} \quad \ytableaushort{\none+}*{2,1}\quad \ytableaushort{\circ}*{2,1} &$x_ix_jx_kx_m$\\
\hline
\multicolumn{3}{l}{$\Bf_{s_1s_0s_1s_0}=z_1\sum_{i<j}(x_i^2x_j+x_ix_j^2)+2z_1\sum_{i<j<k}x_ix_jx_k+$}\\
\multicolumn{3}{c}{$+\sum_{i<j}(2x_i^2x_j^2+x_ix_j^3+x_i^3x_j)+4\sum_{i<j<k}(x_i^2x_jx_k+x_ix_j^2x_k+x_ix_jx_k^2)+$}\\
\multicolumn{3}{r}{$+8\sum_{i<j<k<m}x_ix_jx_kx_m=\frac 1 3 z_1p_1^3(\xx)-\frac 1 3 z_1p_3(\xx)+\frac 1 3 p_1^4(\xx)-\frac 1 3 p_1(\xx)p_3(\xx)
$}\\
\end{longtable}

\begin{longtable}{|l|l|l|}
\caption{D-signed pipe dreams for some $w\in \Dc_3$}\\
\hline
$w$ & Pipe dreams & Weight\\ 
\hline
$\overline 1 \overline 2=s_1s_{\hat 1}$ & \ytableaushort{+}*[*(white!25!gray)]{1}\quad \ytableaushort{\oplus}*{1}& $z_1x_i$\\
& \ytableaushort{}*[*(white!25!gray)]{1}\quad \ytableaushort{\%}*{1}& $x_i^2$\\
& \ytableaushort{}*[*(white!25!gray)]{1}\quad \ytableaushort{+}*{1}\quad \ytableaushort{\oplus}*{1}& $x_ix_j$\\
& \ytableaushort{}*[*(white!25!gray)]{1}\quad \ytableaushort{\oplus}*{1}\quad \ytableaushort{+}*{1}& $x_ix_j$\\
\hline
\multicolumn{3}{l}{$\Df_{s_1s_{\hat 1}}=z_1\sum_i x_i+\sum_{i}x_i^2+2\sum_{i<j}x_i x_j=z_1p_1(\xx)+p_1^2(\xx)$}\\
\hline
$312=s_2s_1$ & \ytableaushort{++}*[*(white!25!gray)]{2,1} & $z_1^2$\\
 & \ytableaushort{+}*[*(white!25!gray)]{2,1}\quad  \ytableaushort{\none,+}*{2,1} & $z_1x_i$\\
 & \ytableaushort{+}*[*(white!25!gray)]{2,1}\quad  \ytableaushort{\none+}*{2,1} & $z_1x_i$\\
 & \ytableaushort{}*[*(white!25!gray)]{2,1}\quad  \ytableaushort{++}*{2,1} & $x_i^2$\\
 & \ytableaushort{}*[*(white!25!gray)]{2,1}\quad  \ytableaushort{+}*{2,1}\quad
 \ytableaushort{\none,+}*{2,1} & $x_ix_j$\\
 & \ytableaushort{}*[*(white!25!gray)]{2,1}\quad  \ytableaushort{+}*{2,1}\quad
 \ytableaushort{\none+}*{2,1} & $x_ix_j$\\
 \hline
 \multicolumn{3}{l}{$\Df_{s_2s_1}=z_1^2+2z_1\sum_i x_i+\sum_i x_i^2+2\sum_{i<j}x_ix_j=z_1^2+2z_1p_1(\xx)+p_1^2(\xx)$}\\
\hline
 $\overline{2}3\overline{1}=s_{\hat 1}s_2$ & \ytableaushort{\none,+}*[*(white!25!gray)]{2,1} & $z_2x_1$\\
 & \ytableaushort{\none+}*[*(white!25!gray)]{2,1}\quad  \ytableaushort{\oplus}*{2,1} & $z_1x_i$\\
 & \ytableaushort{}*[*(white!25!gray)]{2,1}\quad  \ytableaushort{\oplus,+}*{2,1} & $x_i^2$\\
 & \ytableaushort{}*[*(white!25!gray)]{2,1}\quad  \ytableaushort{\none,+}*{2,1} \quad  \ytableaushort{\oplus}*{2,1}& $x_ix_j$\\
 & \ytableaushort{}*[*(white!25!gray)]{2,1}\quad  \ytableaushort{\none+}*{2,1} \quad  \ytableaushort{\oplus}*{2,1}& $x_ix_j$\\
 \hline
 \multicolumn{3}{l}{$\Df_{s_{\hat 1}s_2}=z_1\sum_i x_i + z_2\sum_i x_i + \sum_i x_i^2 +2\sum_{i<j}x_ix_j=(z_1+z_2)p_1(\xx)+p_1^2(\xx)$}\\
\hline
$\overline{13}2=s_2s_1s_{\hat 1}$ & \ytableaushort{+}*[*(white!25!gray)]{2,1} \quad  \ytableaushort{\oplus+}*{2,1}& $z_1x_i^2$\\
& \ytableaushort{+}*[*(white!25!gray)]{2,1}\quad  \ytableaushort{\oplus}*{2,1} \quad  \ytableaushort{\none+}*{2,1}& $z_1x_ix_j$\\
& \ytableaushort{+}*[*(white!25!gray)]{2,1}\quad  \ytableaushort{\oplus}*{2,1} \quad  \ytableaushort{\none,+}*{2,1}& $z_1x_ix_j$\\
& \ytableaushort{}*[*(white!25!gray)]{2,1} \quad  \ytableaushort{\%+}*{2,1}& $x_i^3$\\
& \ytableaushort{}*[*(white!25!gray)]{2,1}\quad  \ytableaushort{\%}*{2,1} \quad  \ytableaushort{\none,+}*{2,1}& $x_i^2x_j$\\
& \ytableaushort{}*[*(white!25!gray)]{2,1}\quad  \ytableaushort{\%}*{2,1} \quad  \ytableaushort{\none+}*{2,1}& $x_i^2x_j$\\
& \ytableaushort{}*[*(white!25!gray)]{2,1}\quad  \ytableaushort{+}*{2,1} \quad  \ytableaushort{\oplus+}*{2,1}& $x_ix_j^2$\\
& \ytableaushort{}*[*(white!25!gray)]{2,1}\quad  \ytableaushort{\oplus}*{2,1} \quad  \ytableaushort{++}*{2,1}& $x_ix_j^2$\\
& \ytableaushort{}*[*(white!25!gray)]{2,1}\quad  \ytableaushort{\oplus}*{2,1} \quad  \ytableaushort{+}*{2,1} \quad  \ytableaushort{\none,+}*{2,1}& $x_ix_jx_k$\\
& \ytableaushort{}*[*(white!25!gray)]{2,1}\quad  \ytableaushort{\oplus}*{2,1} \quad  \ytableaushort{+}*{2,1} \quad  \ytableaushort{\none+}*{2,1}& $x_ix_jx_k$\\
& \ytableaushort{}*[*(white!25!gray)]{2,1}\quad  \ytableaushort{+}*{2,1} \quad  \ytableaushort{\oplus}*{2,1} \quad  \ytableaushort{\none,+}*{2,1}& $x_ix_jx_k$\\
& \ytableaushort{}*[*(white!25!gray)]{2,1}\quad  \ytableaushort{+}*{2,1} \quad  \ytableaushort{\oplus}*{2,1} \quad  \ytableaushort{\none+}*{2,1}& $x_ix_jx_k$\\
\hline
 \multicolumn{3}{l}{$\Df_{s_2s_1s_{\hat 1}}=z_1\sum_i x_i^2 + 2z_1\sum_{i<j}x_ix_j+2\sum_{i<j}(x_i^2x_j+x_ix_j^2)+$} \\\multicolumn{3}{r}{$+4\sum_{i<j<k}x_ix_jx_k=z_1p_1^2(\xx)+\frac 2 3 p_1^3(\xx)+\frac 1 3 p_3(\xx)$}\\
\end{longtable}
\def\cprime{$'$}

\end{document}